\documentclass{article}

\usepackage{graphicx} % Required for inserting images
\usepackage[T1]{fontenc}
\usepackage[utf8]{inputenc}
\usepackage[english]{babel}
\usepackage{amsmath}
\usepackage{amsthm}
\usepackage{amssymb}
\usepackage{tikz}
\usetikzlibrary{arrows.meta,decorations.pathreplacing,positioning}

\usepackage{graphicx}
\usepackage{faktor}
\usepackage{xcolor, colortbl}
\usepackage{makecell}
\usepackage{caption}
\usepackage{subcaption}
\usepackage{comment}
\usepackage{setspace}
\usepackage[11pt]{moresize}
\usepackage{wasysym}
\usepackage{float}
\usepackage[margin=1in]{geometry}
\usepackage{mathrsfs}
\usepackage{url}
\usepackage{authblk}
\usepackage{esint}
\usepackage{bbm}
\usepackage{hyperref}
\usepackage{cleveref}
\usepackage{comment}
\usepackage{bm}
\hypersetup{
colorlinks = true,
urlcolor   = blue,
citecolor  = blue,
linkcolor  = blue,
}

\DeclareMathOperator{\id}{id}
\DeclareMathOperator{\dist}{dist}
\DeclareMathOperator{\opt}{opt}

\DeclareMathOperator{\Lip}{Lip}

\DeclareMathOperator{\loc}{loc}
\DeclareMathOperator{\diver}{div}

\DeclareMathOperator{\supp}{supp}

\DeclareMathOperator{\even}{even}

\theoremstyle{plain}
\newtheorem{thm}{Theorem}[section]

\newtheorem*{thm*}{Theorem}
\newtheorem{lem}[thm]{Lemma}
\newtheorem{prop}[thm]{Proposition}
\newtheorem{cor}[thm]{Corollary}
\newtheorem{defn}[thm]{Definition}
\theoremstyle{remark}
\newtheorem{rmk}[thm]{Remark}

\newcommand{\R}{\mathbb{R}}
\newcommand{\Sph}{\mathbb{S}}
\newcommand{\T}{\mathbb{T}}
\newcommand{\Sp}{\mathbb{S}}

\newcommand{\Z}{\mathbb{Z}}
\newcommand{\N}{\mathbb{N}}

\newcommand{\eps}{\varepsilon}

\newcommand{\fr}{\penalty-20\null\hfill$\blacksquare$}

\numberwithin{equation}{section}

\title{Quantitative Convergence of Wasserstein Gradient Flows \\ of Kernel Mean Discrepancies}
\author{Lénaïc Chizat, Maria Colombo, Roberto Colombo, Xavier Fernández-Real}
\date{}

\begin{document}
\maketitle
\begin{abstract}
We study the quantitative convergence of Wasserstein gradient flows of Kernel Mean Discrepancy (KMD) (also known as Maximum Mean Discrepancy (MMD)) functionals.
Our setting covers in particular the training dynamics of shallow neural networks in the infinite-width and continuous time limit, as well as interacting particle systems with pairwise Riesz kernel interaction in the mean-field and overdamped limit.
Our main analysis concerns the model case of KMD functionals given by the squared Sobolev  distance $ \mathscr{E}^{\nu}_{s}(\mu)= \frac{1}{2}\lVert \mu-\nu\rVert_{\dot H^{-s}}^{2}$ for any $s\geq 1 $ and $\nu$ a fixed probability measure on the $d$-dimensional torus.
First, inspired by Yudovich theory for the $2d$-Euler equation, we establish existence and uniqueness in natural weak regularity classes.
Next, we show that for $s=1$ the flow converges globally at an exponential rate under minimal assumptions, while for $s>1$ we prove local convergence at polynomial rates that depend explicitly on $s$ and on the Sobolev regularity of $\mu$ and $\nu$. These rates hold both at the energy level and in higher regularity classes and are tight for $\nu$ uniform.
We then consider the gradient flow of the population loss for shallow neural networks with ReLU activation, which can be cast as a Wasserstein--Fisher--Rao gradient flow on the space of nonnegative measures on the sphere $\mathbb{S}^d$. 
Exploiting a correspondence with the Sobolev energy case with $s=(d+3)/2$, we derive an explicit polynomial local convergence rate for this dynamics.
Except for the special case $s=1$, even non-quantitative convergence was previously open in all these settings.
We also include numerical experiments in dimension $d=1$ using both PDE and particle methods which illustrate our analysis.
\end{abstract}


\tableofcontents

\bigskip

\section{Introduction}

Let $\mathscr{P}(M)$ denote the set of probability measures on a $d$-dimensional smooth manifold $M$. Given a target $\nu\in \mathscr{P}(M)$, we study the well-posedness and long-time behavior of the Wasserstein gradient flow of \emph{Kernel Mean Discrepancy} (KMD) functionals (also known as \emph{Maximum Mean Discrepancy} (MMD) functionals\footnote{The terminology MMD was introduced in~\cite{gretton2012kernel} where it refers to a general integral probability metric (IPM). Here, following~\cite{wainwright2019high}, we adopt the name KMD to indicate that we use a kernel-based IPM. In machine learning, it is often assumed that the kernel generates a RKHS (e.g.~$s>d/2$ in the Riesz case of~\eqref{eq:energy-Riesz-intro}) but this assumption is not needed for our analysis.}) starting from some $\bar \mu\in \mathscr{P}(M)$. These are functionals of the form 
\begin{equation}\label{eq:energy-general-intro}
    \mathscr{E}^{\nu}(\mu):= \frac{1}{2}\int_{M} \int_{M}K(x,y)d(\mu-\nu)(x)d(\mu-\nu)(y),
\end{equation}
where $K:M\times M\to \mathbb{R}$ is a symmetric and conditionally positive definite (but not necessarily continuous) kernel.
The dynamics can be expressed as solutions to the Cauchy problem for active-scalar continuity equations  of the form  
\begin{equation}\label{eq:PDE-general-intro}
    \partial_{t}\mu_{t}=\diver \left(\mu_{t}\nabla \mathcal{K}\left(\mu_{t}-\nu\right)\right)\qquad \text{in $(0,T)\times M$},\qquad\qquad \mu_{0}=\bar\mu,
\end{equation}
where $\mathcal{K}$ is the positive semidefinite operator given by
\begin{equation*}
    \mathcal{K}(\eta)(x):=\int_{M}K(x,y)d\eta(y).
\end{equation*}

Equation~\eqref{eq:PDE-general-intro} can be interpreted as the evolution of an overdamped system of positively charged particles with law $\mu_t$, interacting with a fixed negatively charged background $\nu$ through the potential $K$.

\medskip 

\paragraph{Motivations.} Interest in studying KMD-type (or MMD-type) discrepancies comes from machine learning and statistics. For instance, \eqref{eq:PDE-general-intro} describes the mean-field (infinite-width) limit of the training dynamics of shallow (i.e.~one hidden layer) neural networks, where $\mu_t$ represents the evolving distribution of the parameters, the objective functional $\mathscr{E}^{\nu}$ coincides with the population loss, and the kernel $K$ depends on the activation function and on the input data distribution~\cite{mei2018mean,rotskoff2022trainability, sirignano2020mean,chizat2018global} (see also \cite{fernandez2022continuous}). Such dynamics have also been motivated in the context of generative modeling where the goal is to find a map transporting an easy-to-sample source density towards a target density (see e.g.~\cite{arbel2019maximum, altekruger2023neural, galashovdeep, chen2025regularized}). In this context,~\eqref{eq:PDE-general-intro} can be interpreted as a simplified model of the training dynamics which consists in minimizing with gradient-based optimization the objective $\theta \mapsto \mathscr{E}^{\nu}((f_\theta)_\# \mu_0)$ where $f_\theta$ is a neural network with parameters $\theta$ and $\mu_0$ a reference probability measure~\cite{unterthiner2018coulomb}. 

\medskip 

\paragraph{Quantitative convergence.} Despite the apparent simplicity of the quadratic objective~\eqref{eq:energy-general-intro}, the qualitative and quantitative convergence properties of its Wasserstein gradient flow are poorly understood beyond special cases. The key obstruction is geometric: while $\mathscr E^\nu$ is convex for the linear structure on measures (by positivity of the kernel), it is typically not geodesically convex in $(\mathscr P(M),W_2)$, so the standard contraction and quantitative convergence mechanisms for Wasserstein gradient flows in geodesically (uniformly) convex scenarios \cite{ambrosio2005gradient} do not apply.

Existing works (see references above) have sought to obtain guarantees on the long-time behavior. For instance,~\cite{chizat2018global} shows that if the support of $\bar \mu$ satisfies an ``omnidirectionality'' condition and if the Wasserstein--Fisher--Rao gradient flow (see Section~\ref{subsec:arccos-intro})  converges, then it is towards a global minimizer; but this result does not apply to Wasserstein gradient flows. Even for Wasserstein--Fisher--Rao flows, it is also only a partial guarantee since it does not establish convergence and, a fortiori, offers no quantitative information on convergence rates. Also,~\cite{arbel2019maximum} shows that  under a boundedness condition on $\mu_t - \nu$ along the flow, then the Wasserstein gradient flow converges globally at a rate $O(1/t)$. However, it is unclear when this assumption may hold or fail, or even if global convergence is to be expected in general, in light of counter-examples in closely related settings~\cite{safran2018spurious}. All in all, long-time convergence guarantees for Wasserstein (or Wasserstein--Fisher--Rao) gradient flows of KMDs is still an open question, even at the local and non-quantitative level. 

\medskip
 
\paragraph{A model case: Riesz kernels on the torus.} We focus on the model case of Riesz kernels on the $d$-torus. We stress, however, that our approach is more broadly applicable, as illustrated by our analysis of ReLU shallow neural networks in \Cref{subsec:arccos-intro} (where $M=\Sph^d$ and $K$ is a different kernel).

Let $M$ be the $d$-dimensional torus $\T^{d} \cong \R^d / \Z^d$, and let  $\mathcal{K}$ be the inverse Laplacian to some power $s\ge 1$,  
\begin{equation*}
    \mathcal{K}(\eta):=(-\Delta)^{-s}\eta,\qquad \int_{\T^{d}} \eta=0.
\end{equation*}
The corresponding energy $\mathscr{E}^{\nu}_{s}(\mu)$, \eqref{eq:energy-general-intro}, is the homogeneous Sobolev $\dot H^{-s}$-discrepancy between $\mu$ and $\nu$,
\begin{equation}\label{eq:energy-Riesz-intro}
    \mathscr{E}^{\nu}_{s}(\mu)= \frac{1}{2}\lVert \mu-\nu\rVert_{\dot H^{-s}}^{2}= \frac{1}{2}\int_{\T^{d}}\int_{\T^{d}}K_{s}(x-y)d(\mu-\nu)(x)d(\mu-\nu)(y),
\end{equation}
where $K_{s}:\T^{d}\to \R$ is the Riesz kernel, solving $(-\Delta)^{s}K_{s}=\delta_{0}-1$ in the sense of distributions. 
In this framework,  the evolution \eqref{eq:PDE-general-intro} takes the form
\begin{equation}\label{eq:PDE-GF}
    \left\{
    \begin{array}{rclll}
         \partial_{t}\mu+\diver(\mu v)&=&0\qquad &\text{in $(0,T)\times \T^{d}$},\\
            v_{t}&=&-\nabla K_{s}*(\mu_{t}-\nu)\qquad &\forall t\in (0,T),\\
            \mu_{0}&=&\bar \mu.
    \end{array}
    \right.
\end{equation}
(See \Cref{def:solution-active-scalar-equation} for the precise notion of weak solution to \eqref{eq:PDE-GF}).

Riesz kernel interactions, as above,   have already attracted the attention of both theoretical and applied communities. As $s\ge 1$ varies, the different asymptotic behavior of short-distance interactions $K_{s}(x-y)\sim c_{d,s}|x-y|^{2s-d}$ leads to a range of regularity regimes, shown in \Cref{fig:riesz} (see details in \Cref{lem:asymptotic-behavior-kernels}), which in turn dictate diverse qualitative behaviors of solutions to \eqref{eq:PDE-GF}. We can identify at least three values for $s$ of particular interest:
\begin{itemize}
    \item [1)] $\left(s=1\right)$. In this case, \eqref{eq:energy-Riesz-intro} is precisely the Coulomb interaction energy of the signed distribution of charge $\mu-\nu$. Due to its physical significance, Coulomb interactions have been studied extensively in the mathematical literature (see \cite{serfaty2024lectures}, \cite{boufadeneVialard2023} and references therein).
 
    \item [2)] $\left(s=\frac{d}{2}+\frac{1}{2}\right)$. In this case, the kernel $K(x-y)$ behaves as the negative distance $-|x-y|$ for short interactions. Negative distance kernels~\cite{szekely2013energy} and their Wasserstein gradient flows have been studied for instance in \cite{hagemann2023posterior, hertrich2023generative, hertrich2024wasserstein}, with applications to machine learning and generative modeling.   
    
    \item [3)] $\left(s=\frac{d}{2}+\frac{3}{2}\right)$. As explained in detail in \Cref{subsec:arccos-intro}, this case is relevant to understanding gradient flows on infinite-width shallow ReLU neural networks. 
\end{itemize}

In this paper we provide an analysis of the equation \eqref{eq:PDE-GF} for all $s\ge 1$, addressing both questions of well-posedness in natural classes and quantitative convergence to minimizers.

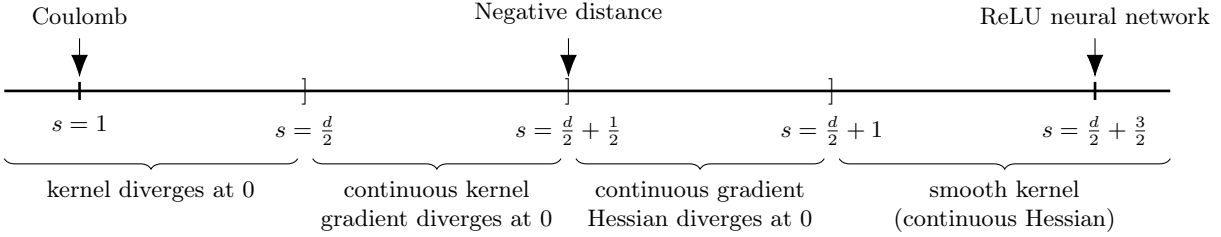
\begin{figure}
\centering
\begin{tikzpicture}[
  >=Latex,
  axis/.style={line width=1.0pt},
  tick/.style={line width=1.0pt},
  pt/.style={circle, fill=black, inner sep=1.6pt},
  lab/.style={font=\small},
  seglab/.style={
    font=\small,
    align=center,
    text width=3.6cm,
    fill=white,
    inner sep=2pt
  },
  ann/.style={font=\small, align=center},
  brace/.style={decorate, decoration={brace, amplitude=4pt, mirror}}
]

% --- intentionally non-uniform horizontal scale
\def\xA{0.0}    % s=1
\def\xB{3.0}    % s=d/2
\def\xC{6.5}    % s=d/2+1/2
\def\xD{10.0}    % s=d/2+1
\def\xE{13.5}   % s=d/2+3/2

% --- axis
\draw[axis] (\xA-1.0,0) -- (\xE+1.0,0);

% --- ticks + dots

\foreach \x/\t in {
  \xA/{$s=1$},
  \xB/{$s=\frac d2$},
  \xC/{$s=\frac d2+\frac12$},
  \xD/{$s=\frac d2+1$}
}{
  %\draw[tick] (\x,0.12) -- (\x,-0.12);
  %\node[pt] at (\x,0) {};
  \node[lab, below=6pt] at (\x,0) {\t};
}
\def\yB{-0.7}      % y-level of braces
\def\xStart{9.2}   % left boundary (e.g. s=d/2+1)
\def\xBraceEnd{12.8} % where the visible brace stops (for label placement)
\def\xAxisEnd{15.0}  % where the axis ends (arrow goes to here)

%\tickright{\xB}
%\node[pt] at (\xA,0) {};
\draw[tick] (\xA,0.12) -- (\xA,-0.12);
\node at (\xB,0) {$]$};
\node at (\xC,0) {$]$};
\node at (\xD,0) {$]$};

% --- regime labels (mirrored braces below axis)
\draw[brace] (\xA-1.0,-0.9) -- node[seglab, below=6pt]
  {kernel diverges at $0$} (\xB-0.1,-0.9);

\draw[brace] (\xB+0.1,-0.9) -- node[seglab, below=6pt]
  {continuous kernel \\ gradient diverges at $0$} (\xC-0.1,-0.9);

\draw[brace] (\xC+0.1,-0.9) -- node[seglab, below=6pt]
  {continuous gradient \\  Hessian diverges at $0$} (\xD-0.1,-0.9);

\draw[brace] (\xD+0.1,-0.9) -- node[seglab, below=6pt]
  {smooth kernel \\ (continuous Hessian)} (\xE+1,-0.9);

% --- arrows for special values
% Coulomb
\draw[-{Latex[length=3mm]}] (\xA,0.7) -- (\xA,0.18);
\node[ann, above=2pt] at (\xA,0.7) {Coulomb};

% Negative distance
\draw[-{Latex[length=3mm]}] (\xC,0.7) -- (\xC,0.18);
\node[ann, above=2pt] at (\xC,0.7)
  {Negative distance};

% ReLU neural network
\draw[tick] (\xE,0.12) -- (\xE,-0.12);
\draw[-{Latex[length=3mm]}] (\xE,0.7) -- (\xE,0.18);
\node[ann, above=2pt] at (\xE,0.7)
  {ReLU neural network};
\node[lab, below=6pt] at (\xE,0)
  {$s=\frac d2+\frac32$};

% --- title (lifted safely above everything)
%\node[lab, above=38pt] at ({(\xA+\xE)/2},0)
%  {Riesz kernel parameter $s$ on $\mathbb{R}^d$};

\end{tikzpicture}
\caption{Regularity regimes of Riesz kernels $K_s$ on the $d$-torus  (bracket orientation indicates inclusion/exclusion of the endpoint) and some values of particular interest (arrows). The negative distance and ReLU neural network cases are not strictly speaking Riesz kernels on the $d$-torus, but their regularity on the diagonal matches that of the indicated exponent. For $d=1$, the kernel is never singular for $s\geq 1$ (as then the Coulomb case coincides with the negative distance case).}\label{fig:riesz}
\end{figure}

\subsection{Well-posedness results}\label{subsec:well-posedness-intro}

Our convergence results below rest on a robust well-posedness theory for \eqref{eq:PDE-GF}. We therefore begin by proving existence, uniqueness, stability, and propagation of regularity in natural classes. These results are of independent interest and will be crucial for the quantitative long-time analysis. 

For every $s\ge 1$, we identify a natural weak class of solutions, the guiding principle being that solutions within the class should generate a (quasi)Lipschitz vector field. 
We denote by $\mathcal{M}(\T^d)$ the space of finite (signed) measures in $\T^d$ and by $L^{p,1}(\T^{d})$ the Lorentz space (see \eqref{eqn:Lorentz} below; alternatively, one may replace $L^{p,1}$ with $L^q$ for some $q>p$). We always identify absolutely continuous measures with their density with respect to Lebesgue. Finally, we set
    \begin{equation}\label{eq:def-yudovich-class}
        \mathscr{X}_{s}(\T^{d}):= \begin{cases}
            L^{\infty}(\T^{d})\qquad\qquad &\text{if $s=1$},\\
            L^{p,1}(\T^{d})\qquad \text{for\,\, $p=\frac{d}{2s-2}$}\qquad\qquad  &\text{if $s\in \left(1,\frac{d}{2}+1\right)$},\\
            \mathcal{M}(\T^{d})\qquad\qquad &\text{if $s\ge \frac{d}{2}+1$}.
        \end{cases}
    \end{equation}

\begin{prop}[Local well-posedness]\label{prop:local-well-posedness-intro}
    Let $s\ge 1$ and let $\mathscr{X}_{s}(\T^{d})$ be given by \eqref{eq:def-yudovich-class}.  Then, for every 
    $\bar\mu , \nu \in \mathscr{P}\cap \mathscr{X}_{s}(\T^{d})$ 
    there exist a maximal time of existence $T>0$ and a unique maximal solution $\mu \in L_{\loc}^{\infty}([0,T);\mathscr{X}_{s}(\T^{d}))$ of equation \eqref{eq:PDE-GF} in the sense of \Cref{def:maximal-solutions}. Moreover:
    \begin{itemize}
        \item If $s\ge \frac{d}{2}+1$, then $T=\infty$.
        \item  If $s\in \left[1,\frac{d}{2}+1\right)$,  we have  $T<\infty $ if and only if $\limsup_{t\to T^{-}}\lVert \mu_{t}\rVert_{L^{p}}=+\infty$, where $p=\frac{d}{2s-2}\in (1,\infty]$. 
    \end{itemize}

        \noindent Finally, solutions propagate H\"older and Sobolev regularity (see \Cref{prop:propagation-regularity} and \Cref{prop:propagation-sobolev}).
\end{prop}

  The result is inspired by Yudovich's theory for $L^{\infty}$-solutions of $2d$-Euler equations in vorticity form \cite{yudovich1963non}, although our proof follows the more recent approach developed in \cite{marchioro2012mathematical} and \cite{loeper2006uniqueness}. See also \cite{bertozzi2012aggregation, bertozzi2011lp} for the well-posedness of a related model without target measure in $\R^d$. 
 Even if the result we prove is for $M = \T^d$ and $K = K_s$ the Riesz kernel, the same methods apply to more general $d$-dimensional manifolds (e.g.~$M = \Sph^{d}$ or $M = \R^d$) and other kernels with comparable asymptotic behaviors.

  For $s>d/2+1$, $K_s$ is semiconvex, so global well-posedness of \eqref{eq:PDE-GF} in $\mathcal M$ follows from the general theory \cite{ambrosio2005gradient} (see also \cite{crippa2013existence}). At the endpoint $s=d/2+1$, $K_s$ is not semiconvex, but we still obtain global well-posedness of \eqref{eq:PDE-general-intro} in $\mathcal M$ using the log-Lipschitz regularity of the induced velocity field.

 We prove \Cref{prop:local-well-posedness-intro} in Section~\ref{sec:well-posedness} and we obtain there further properties, such as the quantitative stability of solutions in the Wasserstein metric with respect to variations of the initial and target measures $\bar\mu$ and $\nu$.
Moreover, consistently with the formal identification outlined above, we observe in \Cref{prop:gradient-flow-structure} that the solution $\mu$ of \eqref{eq:PDE-GF} given by \Cref{prop:local-well-posedness-intro} is a Wasserstein gradient flow for the Riesz discrepancy energy $\mathscr{E}^{\nu}_{s}$ (see \Cref{def:Wasserstein-gradient-flow-H-s}). In fact, $\mu\in \Lip_{\loc}\left([0,T);\left(\mathscr{P}(\T^{d}),W_{2}\right)\right)$ and the energy dissipation identity holds:
        \begin{equation}\label{eq:energy-dissipation-identity}
        \frac{d}{dt}\mathscr{E}^{\nu}_{s}(\mu_{t})=-\int_{\T^{d}}|\nabla K_{s}*(\mu_{t}-\nu)|^{2}d\mu_{t}\qquad  \forall t\in (0,T).
        \end{equation}

\subsection{Convergence results}\label{subsec:convergence-intro}

Next we state the main results: \Cref{thm:convergence-s=1} for $s=1$ and \Cref{thm:convergence-s>1} for $s>1$, addressing convergence to the target for solutions of \eqref{eq:PDE-GF}. The split reflects a qualitative change in the dynamics: the endpoint case $s=1$ enjoys additional structure (e.g.~a maximum principle), while the regime $s>1$ is technically more demanding and is the central case of this work.

\begin{thm}[Global convergence to the target: $s=1$]\label{thm:convergence-s=1}
Let $s=1$, $\bar\mu, \nu \in \mathscr{P}\cap L^{\infty}(\T^{d})$ and let $\mu\in L^{\infty}_{\loc}([0,T);L^{\infty}(\T^{d}))$ be the maximal solution of \eqref{eq:PDE-GF} given by \Cref{prop:local-well-posedness-intro}. Then, $T=\infty$ and it holds:
    \begin{gather}
            \min\{\inf \bar{\mu}, \inf \nu\}\le \inf\mu_{t}\le \sup \mu_{t}\le \max \{\sup \bar{\mu}, \sup \nu\}\qquad \forall t\in [0,\infty);\label{eq:max-princ-s=1-intro}\\
            \text{$\mu_{t}$ converges  weakly-$*$ in $L^{\infty}(\T^{d})$ to $\nu$ as $t\to \infty$}.\label{eq:weak*-conv-Linfty_coulomb}
    \end{gather}
 Suppose, moreover, that $\bar \mu, \nu \ge \alpha>0$ almost everywhere in $\T^{d}$. Then the following hold:
    \begin{itemize}
        \item [i)] (Exponential weak convergence in energy and $W_2$). It holds:
        \begin{equation}\label{eq:exponential-decay-H^-1-coulomb}
            \alpha^{1/2}W_{2}(\mu_{t},\nu)\le \lVert \mu_{t}-\nu\rVert_{\dot H^{-1}}\le \lVert \bar\mu -\nu\rVert_{\dot H^{-1}}e^{-\alpha t}\qquad \forall t\in [0,\infty).
        \end{equation}
        \item [ii)] (Uniform convergence). Suppose that $\nu \in C(\T^{d})$ has a Dini modulus of continuity\footnote{A modulus of continuity is a strictly increasing continuous and concave function $\omega:[0,\infty)\to [0,\infty)$ such that $\omega(0)=0$. We say that $f:\T^{d}\to \R$ has modulus of continuity $\omega$ if $|f(x) - f(y)|\le \omega(|x-y|)$ for all $x, y\in T^d$.         We say that $\omega$ is a Dini modulus of continuity if it holds
        $
            \int_{0}^{1}\frac{\omega(r)}{r}dr<\infty.
        $
        }. Then
        \begin{equation}\label{uniform-convergence-Dini-s=1}
            \lVert \mu_{t}-\nu\rVert_{L^{\infty}}\to 0 \qquad \text{as $t\to \infty$}.
        \end{equation}
        Moreover, the uniform convergence is exponential as soon as $\nu$ is H\"{o}lder continuous. 
        \item [iii)] (Smooth convergence). Let $\gamma>d/2$ and suppose that $\bar\mu \in H^{\gamma}(\T^{d})$, $\nu \in H^{\gamma+1}(\T^{d})$ (see \Cref{subsec:Sobolev-Riesz-Torus}). Then, there exists a constant $C>0$ depending only on $d,\alpha, \gamma$, $\lVert \bar \mu\rVert_{H^{\gamma}}$, and $\lVert \nu\rVert_{H^{\gamma+1}}$ such that 
        \begin{equation*}
            \lVert \mu_{t}-\nu\rVert_{\dot H^{\gamma}}\le \left(\lVert \bar\mu-\nu\rVert_{\dot H^{\gamma}}^{2}+C(1-e^{-t/C})\right)^{1/2}e^{-\alpha t + C(1-e^{-t/C})}\qquad \forall t\in [0,\infty).
        \end{equation*}
    \end{itemize}
\end{thm}

\begin{rmk}[Sharpness of the lower bounds]
\label{rmk:exponential-filling}
    We stress that the lower bound on the initial measure $\bar\mu$ is not essential to obtain  
    exponential convergence to the target $\nu$, as discussed in \Cref{subsubsec:relaxation-lowerbound-initial} (see \Cref{prop:relaxation-lower-bound-initial-measure}). Indeed, if $\nu\ge \alpha>0$, any region where $\bar\mu$ vanishes is “filled up” exponentially fast in time; we call this \emph{exponential filling of holes} (see \Cref{lem:exponential-filling-holes}). By contrast, numerical experiments in \Cref{ssec:numerical_experiments} show that a positive lower bound on $\nu$ is needed to ensure exponential convergence. \fr
\end{rmk}

In \cite{boufadeneVialard2023} the authors initiated the study of \eqref{eq:PDE-GF}  for $s = 1$, proving \eqref{eq:max-princ-s=1-intro} and the resulting exponential convergence \eqref{eq:exponential-decay-H^-1-coulomb} assuming the stronger assumption of H\"older continuity of the initial and target densities.

The results of the present paper are obtained by an independent, self-contained approach and both strengthen and complete those conclusions: 
\begin{itemize}
    \item[i)] First, we work in the more general setting of bounded densities (the most general space where we expect well-posedness), and start by establishing the maximum principle \eqref{eq:max-princ-s=1-intro}, hence deducing the exponential convergence under the corresponding positive lower-bound assumptions. 
    \item[ii)] We further prove unconditional global qualitative convergence of solutions to minimizers \eqref{eq:weak*-conv-Linfty_coulomb}, derive quantitative rates once a positive lower bound  is available, and extend the convergence of solutions to higher-regularity topologies under suitable regularity of the data. In particular, a Dini-continuous target implies uniform convergence, whereas no uniform convergence can be expected for discontinuous targets. 
    \item[iii)]Finally, we are also able to prove exponential convergence of the energy without requiring any lower bound on the initial measure (see \Cref{rmk:exponential-filling}), as illustrated by numerical simulations in \Cref{ssec:numerical_experiments}.
\end{itemize}

 \medskip 
Next, we state our quantitative local convergence result in the case $s>1$. We consider this to be the main result of the present work. 

\begin{thm}[Local convergence to the target: $s>1$]\label{thm:convergence-s>1}
Let $s>1$, $\gamma>d/2$, $\bar\mu \in \mathscr{P}\cap H^{\gamma}(\T^{d})$, $\nu \in \mathscr{P}\cap H^{\gamma+s}(\T^{d})$ such that $\bar\mu, \nu \ge \alpha>0$, and let $\mu \in L^{\infty}_{\loc}([0,T);H^{\gamma})$ be the unique maximal solution of \eqref{eq:PDE-GF} given by \Cref{prop:local-well-posedness-intro}. There exist constants $C,\delta>0$ depending only on $d,s,\gamma,\alpha, \lVert \bar\mu \rVert_{H^{\gamma}}$, and $\lVert \nu\rVert_{H^{\gamma+s}}$ such that, if $\lVert \bar\mu-\nu\rVert_{\dot H^{-s}}\le \delta$, then 
\begin{equation}\label{eq:conclusion-thm-conv-s>1}
    T=\infty,\qquad \lVert\mu_{t}-\nu\rVert_{\dot H^{-s}}\le \lVert \bar\mu -\nu\rVert_{\dot H^{-s}}(1+t/C)^{-\frac{\gamma+s}{2(s-1)}}\qquad \text{and}\qquad \lVert \mu_{t}\rVert_{H^{\gamma}}\le C\qquad \forall t\in [0,\infty).
\end{equation}
In particular, by Sobolev interpolation,
\begin{equation*}
    \lVert \mu_{t}-\nu\rVert_{\dot H^{\gamma'}}\le C^{\frac{\gamma'+s}{\gamma+s}}\lVert \bar\mu-\nu\rVert_{\dot H^{-s}}^{\frac{\gamma-\gamma'}{\gamma+s}}(1+t/C)^{-\frac{\gamma-\gamma'}{2(s-1)}}\qquad \forall t\in [0,\infty),\quad \forall \gamma'\in [-s,\gamma].
\end{equation*}
\end{thm}
\begin{rmk}[On the locality assumption] In \cite{safran2018spurious} the authors construct examples of strict local minimizers for the population square-loss of discrete shallow neural networks with ReLU activation function, a system strictly related to our case $s=\frac{d+3}{2}$, as explained in \Cref{subsec:arccos-intro}.  
This suggests that the locality assumption $\lVert \bar \mu-\nu\rVert_{\dot H^{-s}}\le \delta$ might be necessary for a clean quantitative convergence result to hold in the case $s>1$. Note that, on the other hand, such discrete examples of local minimizers cannot be found for the Coulomb interaction ($s=1$), because of Earnshaw's theorem from electrostatics, according to which there are no stable stationary configurations of point charges for the Coulombian potential. \fr
\end{rmk}

\begin{rmk}[Sharpness of the polynomial decay rate] Under a control of the initial datum in $H^{\gamma}$, the exponent $\frac{\gamma+s}{2(s-1)}$ is sharp in the energy decay from \eqref{eq:conclusion-thm-conv-s>1}. To motivate this, let us take $\nu=1$ and look at the linearized equation for $\sigma_{t}:=\mu_{t}-1$:
\begin{equation*}
    \partial_{t}\sigma_{t}=-(-\Delta)^{1-s}\sigma_{t}.
\end{equation*}
Expanding  $\sigma_{t}$  in Fourier and solving for its coefficients $\hat{\sigma}_{k}(t)$ we find $\hat{\sigma}_{k}'(t)=-(2\pi|k|)^{2-2s}\hat{\sigma}_{k}(t)$, that gives
\begin{equation}\label{eq:decay-fourier-example-sharpness-polynomial}
    \hat{\sigma}_{k}(t)= \hat{\sigma}_{k}(0)\exp\left(-(2\pi|k|)^{2-2s}t\right)\qquad \forall t\ge 0,\quad \forall k\in \Z^{d}\setminus \{0\}.
\end{equation}
Now, for any integer $n\ge 1$, let us consider the initial datum $\sigma_{0}^{n}(x):= \sqrt{2}(2\pi n)^{-\gamma}\cos(2\pi n x_{1})$, for which we have $\lVert \sigma^{n}_{0}\rVert_{\dot H^{\gamma}}=1$. From \eqref{eq:decay-fourier-example-sharpness-polynomial} we find 
\begin{equation*}
    \lVert \sigma_{t}^{n}\rVert_{\dot H^{-s}}= \sqrt{2}(2\pi n)^{-\gamma-s}\exp \left(-(2\pi n)^{2-2s}t\right)\gtrsim t^{-\frac{\gamma+s}{2(s-1)}}\qquad \text{for $t=n^{\frac{1}{2s-2}}$}.
\end{equation*}
This shows that among all initial data with unit $\dot H^{\gamma}$-norm, we cannot hope for the $\dot H^{-s}$-norm at time $t$ to be less than $c_{s,\gamma}t^{-\frac{\gamma+s}{2(s-1)}}$, thus proving the optimality of the decay rate in \Cref{thm:convergence-s>1}. \fr 
\end{rmk}

 Hence, for $s>1$ we obtain a polynomial relaxation rate (in $\dot H^{-s}$, and by interpolation in intermediate Sobolev norms) under a natural small-discrepancy assumption, in a family of Riesz-type regimes that includes in particular the negative-distance (``energy distance'') kernel in the case $s = \tfrac{d}{2}+\tfrac12$, widely used in statistics and in recent flow-based methods for imaging and generative modeling \cite{szekely2013energy,hagemann2023posterior,hertrich2023generative}. 
  While global-in-time convergence for the  negative-distance kernel in dimensions $d\ge2$ is outside the reach of geodesic-convexity techniques \cite{boufadeneVialard2023,duong2024distancekernel} (due to the lack of uniform geodesic convexity of the functional), our analysis isolates a robust mechanism---a local {\L}ojasiewicz inequality propagated by higher-order energy estimates---that can furthermore be transported to other kernels and geometries; we illustrate this by treating the arccos/ReLU kernel on the sphere in \Cref{subsec:arccos-intro} below.
We refer to \Cref{subsec:idea-proof} for the ideas of the proof behind the convergence results in Theorems~\ref{thm:convergence-s=1} and \ref{thm:convergence-s>1}.

\subsection{Quantitative convergence for infinite-width shallow neural networks}\label{subsec:arccos-intro}

Consider an infinite-width ReLU Neural Network, that is, a function $f_\mu:\R^{d+1}\to \R$ parameterized by a probability measure $\mu \in \mathscr{P}(\R^{d+2})$ via the expression
\begin{equation}
\label{eq:fmu_intro}
    f_{\mu}(x):=\int_{\R^{d+2}}\Phi(w, x)d\mu(w),\qquad w=(a,b)\in \R\times \R^{d+1},\quad \Phi(w,x)=a(b\cdot x)_{+}.
\end{equation}
Notice that $f_\mu$ is a positively 1-homogeneous function of $x$, so we might as well restrict its inputs to the unit sphere $\Sph^d$. For a given $\nu \in \mathcal{P}(\R^{d+2})$, we consider the mean square energy
\begin{gather*}
    \mathscr{E}^{\nu}_{\Sp^{d}}(\mu):=\frac{1}{2}\int_{\Sp^{d}}|f_{\mu}-f_{\nu}|^{2}dx=\frac{1}{2}\int_{\R^{d+2}}\int_{\R^{d+2}}\hat K(w,w')d(\mu-\nu)(w)d(\mu-\nu)(w'),\\
    \hat K(w,w')=\int_{\Sp^{d}}aa'(b\cdot x)_{+}(b'\cdot x)_{+}dx,\qquad w=(a,b), w'=(a',b') \in \R\times\R^{d+1}.
\end{gather*}
With an initialization $\bar \mu \in \mathscr{P}(\R^{d+2})$, the Wasserstein gradient flow of $\mathscr{E}^{\nu}_{\Sp^{d}}$ is given by the equation
\begin{equation}\label{eq:WGF-arccos}
    \left\{
    \begin{array}{rclll}
        \partial_{t}\mu_{t}+\diver\left(\mu_{t}v_{t}\right)&=&0\quad &\text{in $(0,T)\times \R^{d+2}$},\\
           v_{t}&=&-\nabla \hat{\mathcal{K}}(\mu_{t}-\nu)\quad &\forall t\in (0,T),\\
           \mu_{0}&=&\bar\mu,
    \end{array}
    \right.
\end{equation}
where 
\begin{equation*}
     \hat{\mathcal{K}}(\eta)(w):=\int_{\R^{d+2}}\hat K(w,w')d\eta(w')\qquad \forall \eta\in \mathcal{M}(\R^{d+2}).
\end{equation*}
This dynamics represents the evolution of the parameters of a shallow ReLU Neural Network trained with (stochastic) gradient descent on the population square loss, with initial weights independently drawn from $\bar \mu$, with input data uniform on the sphere and with Bayes predictor $f_\nu$, in the small learning rate and infinite width limit (see~\cite{mei2019mean, wojtowytsch2020convergence, chizat2020implicit} for details on this link). 

Exploiting the structure of $\Phi(w, x)$ (in particular, the 1-homogeneity separately in $a$ and $b$), we may further reduce to expressing the output along the evolution as
\[
f_{\mu_{t}}(x) = \int_{\Sph^d} (b\cdot x)_+ d\mu_{t}(b)
\]
for some (signed) even measure on $\Sph^d$, which we still denote $\mu_{t}$ as an abuse of notation. Moreover, up to an additive constant depending only on the regularity of $f_{\bar \mu}$ and $f_\nu$ (initial and target outputs), we may further assume that $\mu_{t}$ is a nonnegative measure along the training dynamics, which becomes non-conservative and is given by
\begin{equation}\label{eq:PDE-Wasserstein-Fisher-Rao-sphere10}
    \left\{
    \begin{array}{rclll}
        \partial_{t}\mu_{t}+\diver_{\Sp^{d}}\left(\mu_{t}v_{t}\right)&=&-4\mathcal{K}(\mu_{t}-\nu)\mu_{t}\qquad &\text{in $(0,T)\times \Sp^{d}$},\\
           v_{t}&=&-\nabla_{\Sp^{d}} \mathcal{K}(\mu_{t}-\nu)\qquad &\forall t\in (0,T),\\
           \mu_{0}&=&\bar\mu,
    \end{array}
    \right.
\end{equation}
where we denote
\begin{equation}
\label{eq:def-arccos-operator+kernel}
    \mathcal{K}(\eta)(x)=\int_{\Sp^{d}}K(x,y)d\eta(y)\quad \forall \eta\in \mathcal{M}(\Sp^{d}),\qquad K(x,y)=\int_{\Sp^{d}}(x\cdot \xi)_{+}(y\cdot \xi)_{+}d\xi,\quad x,y \in \Sp^{d}.
\end{equation}
By doing so, we are slightly limiting the expressivity of the network to outputs that are even and sufficiently regular. We refer the reader to \Cref{sec:WFR} where we give further details on this reduction.

In what follows, we therefore consider data $\bar\mu,\nu\in \mathcal{M}_{+}(\Sp^{d})$\footnote{The space of (nonnegative) finite measures on $\Sp^{d}$.}, even on $\Sp^{d}$, and we study solutions $t\mapsto \mu_{t}\in \mathcal{M}_{+}(\Sp^{d})$ of equation \eqref{eq:PDE-Wasserstein-Fisher-Rao-sphere10}. We prove that a unique global solution exists in the class of even nonnegative measures, weakly-$*$ continuous in time, and it propagates H\"older and Sobolev regularity of the data (see \Cref{prop:well-posedness-WFR} and \Cref{cor:arccos-propagation-sobolev}). A slightly different approach is needed here with respect to the well-posedness theory for \eqref{eq:PDE-GF}, since we have to deal with the additional non-conservative term $-4\mathcal{K}(\mu_{t}-\nu)\mu_{t}$ in the right-hand side. 

The dynamics of \eqref{eq:PDE-Wasserstein-Fisher-Rao-sphere10} can be interpreted as a Wasserstein--Fisher--Rao gradient flow (see e.g.~\cite{gallouet2017jko, liero2023fine, chizat2022sparse}) for the energy $\mathscr{E}^{\nu}_{\Sp^{d}}:\mathcal{M}_{+}(\Sp^{d})\to [0,\infty)$ given by
\begin{equation}\label{eq:def-arccos-energy-intro}
    \mathscr{E}^{\nu}_{\Sp^{d}}(\mu)=\frac{1}{2}\int_{\Sp^{d}}\int_{\Sp^{d}}K(x,y)d(\mu-\nu)(x)d(\mu-\nu)(y)=\frac{1}{2}\int_{\Sp^{d}}\mathcal{K}(\mu-\nu)(x)d(\mu-\nu)(x),
\end{equation}
whose corresponding dissipation identity reads as
\begin{equation}\label{eq:energy-dissipation-Fisher-Rao}
    \frac{d}{dt}\mathscr{E}^{\nu}_{\Sp^{d}}(\mu_{t})=-\int_{\Sp^{d}}\left(|\nabla_{\Sp^{d}} \mathcal{K}(\mu_{t}-\nu)|^{2}+4|\mathcal{K}(\mu_{t}-\nu)|^{2}\right)d\mu_{t}\qquad \forall t\in (0,\infty).
\end{equation}
Analyzing the spectral behavior of the operator $\mathcal{K}$ (see \Cref{subsec:analysis-sphere-arccos}) we show that 
\begin{equation}\label{eq:comparability-energy-arccos-H^-s}
    \mathscr{E}^{\nu}_{\Sp^{d}}(\mu)\approx_{d} (\mu(\Sp^{d})-\nu(\Sp^{d}))^{2}+\lVert \mu-\nu\rVert_{\dot H^{-s}(\Sp^{d})}^{2}\qquad \forall \mu,\nu \in \mathcal{M}_{+}(\Sp^{d})\,\,\text{even},\qquad s=\frac{d+3}{2}.
\end{equation}
Then, adapting the arguments of \Cref{thm:convergence-s>1} to this setting, we obtain the following local polynomial convergence guarantees (see \Cref{subsec:quantitative-convergence-WFR}). Although there are some local guarantees for shallow neural networks in the mean-field limit when $\nu$ is a sparse measure~\cite{akiyama2021learnability, chizat2022sparse, li2020learning, zhou2021local, zhou2024does}, our result is, to the best of our knowledge, the first convergence result that applies in the case where $\nu$ has a density and belongs to a truly infinite dimensional space.

\begin{thm}[Convergence for neural networks]\label{thm:convergence-arccos}
Let $d\in \N$ be odd and $s:=\frac{d+3}{2}\in \N$. Let $\gamma\in \N$, $\gamma\ge s-1$ and $\bar\mu\in H^{\gamma}(\Sp^{d})$, $\nu\in  H^{\gamma+1}(\Sp^{d})$ be even in $\Sp^{d}$ and such that $\bar\mu, \nu \ge \alpha>0$. There exist constants $C,\delta>0$ depending only on $d,\gamma, \alpha, \lVert \bar\mu \rVert_{H^{\gamma}}$, and $\lVert \nu\rVert_{H^{\gamma+1}}$ such that, if $\mathscr{E}^{\nu}_{\Sp^{d}}(\bar \mu)\le \delta$, then the solution $\mu$ of \eqref{eq:PDE-Wasserstein-Fisher-Rao-sphere10} satisfies 
\begin{equation}\label{eq:conclusion-convergence-arccos}
  \mathscr{E}^{\nu}_{\Sp^{d}}(\mu_{t})\le \mathscr{E}^{\nu}_{\Sp^{d}}(\bar\mu)(1+t/C)^{-\frac{\gamma+s}{s-1}}\qquad \text{and}\qquad \lVert \mu_{t}\rVert_{H^{\gamma}}\le C\qquad \forall t\in [0,\infty).
\end{equation}
\end{thm}

\noindent Some remarks are in order:
\begin{rmk}
    Without the restriction to nonnegative measures, the neural network can represent any regular enough target function (without the need to add a constant $S$ as in \Cref{cor:fnu} below); but we then have instead a two-species evolution system (see \eqref{eq:WFR-two-species} in \Cref{sec:WFR}). Obtaining a convergence theorem in this case  is more challenging for various reasons. For instance, there is an inherent ambiguity in the expected limit at infinite time for a pair $(\mu_{t}^{+},\mu_{t}^{-})$, given that each $(\hat{\nu}^{+},\hat{\nu}^{-})$ satisfying $\hat\nu^{+}-\hat\nu^{-}=\nu^{+}-\nu^{-}$ attains the minimal energy for the system. Furthermore, guaranteeing higher Sobolev regularity by means of energy estimates is made more difficult by the nonlocal interaction between the two terms $\mu_{t}^{+}$ and $\mu_{t}^{-}$. \fr
\end{rmk}
\begin{rmk}
    Our assumption that the dimension $d$ is odd is only made for convenience in order to deal with integer derivatives, but we do not expect it to be relevant in practice. \fr
\end{rmk}

For concreteness, let us state the assumptions of \Cref{thm:convergence-arccos} directly in terms of the regularity of the target function $f_\nu$. As previously mentioned, restricting ourselves to dynamics over nonnegative measures implies that we can represent all regular enough even functions only up to a constant $S$. 
    
\begin{cor}\label{cor:fnu}
    Let $d\in \N$ be odd, $s=\frac{d+3}{2} \in \N$, and $\beta\in \N$ be such that $\beta\ge 2s$. Let $g\in H^{\beta}(\Sp^{d})$ be an even function and let $M>0$. There are $K>0$ depending only on $d, \beta$, and $C,\delta>0$ depending on $d,\beta, M$ and $\lVert g\rVert_{H^{\beta}}$ such that the following holds. 

    Let $S=K (\lVert g\rVert_{H^{\beta}}+1)$ and 
    consider the gradient flow $\mu_{t}$ in \eqref{eq:PDE-Wasserstein-Fisher-Rao-sphere10} with even initialization $\bar \mu\in \mathcal{M}_+(\mathbb{S}^d)$ and target $f=f_{\nu} = g+S$.
    If $\Vert \bar \mu\Vert_{H^{\beta-s-1}}\le M$, and $\lVert f_{\bar \mu}-f\rVert_{L^{2}}^{2}\le \delta$, then 
    $$\lVert f_{\mu_t}-f\rVert_{L^{2}}^{2}\le \lVert f_{\bar \mu}-f\rVert_{L^{2}}^{2}(1+t/C)^{-\frac{\beta-1}{s-1}}\qquad \text{and}\qquad \lVert f_{\mu_t}\rVert_{H^{\beta-1}}\le C\qquad \forall t\in [0,\infty).$$
\end{cor}

\subsection{Numerical illustrations}
\label{ssec:numerical_experiments}
We perform numerical experiments with $d=1$ in three settings: (A) Riesz interaction with $s=1$, (B) Riesz interaction with $s=2$, and (C) shallow ReLU neural networks.  These three settings illustrate \Cref{thm:convergence-s=1}, \Cref{thm:convergence-s>1} and \Cref{thm:convergence-arccos}, respectively\footnote{The code to reproduce the experiments can be found at \url{https://github.com/lchizat/2026-WGF-KMD}.}.

In settings (A) and (B), we use a finite volume discretization with upwind scheme and variable time-step. Given a discretized velocity field $v_t$ (computed by applying the appropriate multiplier in Fourier domain), this scheme builds a globally neutral transfer of mass between neighbouring cells consistent with $v_t$, thereby ensuring exact conservation of mass and of nonnegativity. In setting (C), we use particle discretization with arccos kernel interaction, which is equivalent to  gradient descent on the population loss with an approximation of the target $f_\nu$.

We generate random densities $\nu$ and $\bar \mu$ of the desired Sobolev regularity as follows. First we build an unnormalized density by sampling  its Fourier coefficients with centered normal weights with the appropriate variance decay, in order to control the value of the largest $\gamma$ such that $\nu\in \dot{H}^\gamma$. Then we shift and scale this function to obtain a probability density of minimum value $0$. Finally, we take a mixture with the uniform distribution, in order to have control on the minimum value of the density. In Setting (C) where $\nu$ needs to be discretized, we use a discrete measure with equal Dirac masses located at equi-spaced quantiles.

The results for setting (A) are shown in \Cref{fig:s=1}. All observations are consistent with the conclusion of \Cref{thm:convergence-s=1}. In particular, the convergence rate bound (shown in dotted lines) appears to be rather tight. In case $\nu$ has areas of $0$ densities and $\bar \mu$ has not (not shown), we have observed that $\min \mu_t = \Theta(1/t)$ and that the exponential convergence rate in energy is lost, suggesting that the lower bound assumption on $\nu$ is necessary for exponential convergence (see \Cref{rmk:exponential-filling}). 

The results for setting (B) are shown in \Cref{fig:s=2}. Although our guarantees are local, we observed global convergence in all our experiments, suggesting that counter-examples might be hard to find. The rate of convergence in $\dot{H}^{-s}$ distance $O \Big(t^{-\frac{\min(\gamma_0,\gamma_\nu-s)+s}{2(s-1)}}\Big)$ from \Cref{thm:convergence-s>1} appears to be tight at least in the regime $\gamma_\nu -s \geq \gamma_{0}$.  

The results for setting (C) are shown in \Cref{fig:relu}. Here again we observe global convergence while our guarantee is only local, and the rate of convergence in energy $O \Big(t^{-\frac{\min(\gamma_0,\gamma_\nu-1)+s}{(s-1)}}\Big)$ with $s=(d+3)/2$ from \Cref{thm:convergence-arccos} appears to be very slightly conservative (still, it is asymptotically sharp for $d$ large, where the sharp rate is the one in  \Cref{rmk:improved-regularity-target-convergence-arccos} for all $d$). 
We also observe that the Wasserstein (W) and Wasserstein--Fisher--Rao (WFR) gradient flows have comparable behaviors; this is possible here because $\nu$ is chosen to be a probability measure. If $\nu$ was instead a nonnegative measure of mass different from $1$, then (W) cannot converge to a minimizer while (WFR) can (and locally does).  

\begin{figure}[t]
  \centering
  \begin{subfigure}{0.32\textwidth}
    \centering
    \includegraphics[width=\linewidth]{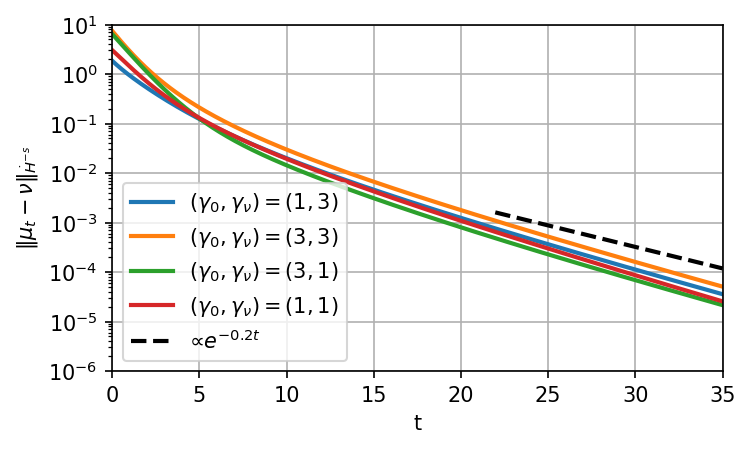}
    \caption{Varying source \& target regularity}
  \end{subfigure}\hfill
  \begin{subfigure}{0.32\textwidth}
    \centering
    \includegraphics[width=\linewidth]{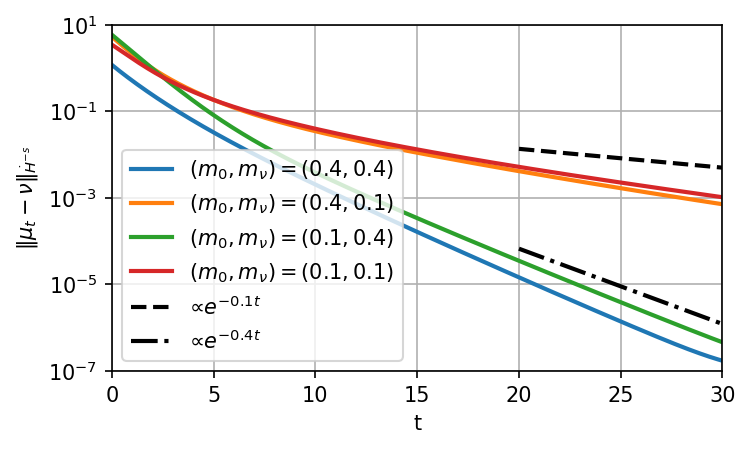}
    \caption{Varying source \& target minimum}
  \end{subfigure}\hfill
  \begin{subfigure}{0.32\textwidth}
    \centering
    \includegraphics[width=\linewidth]{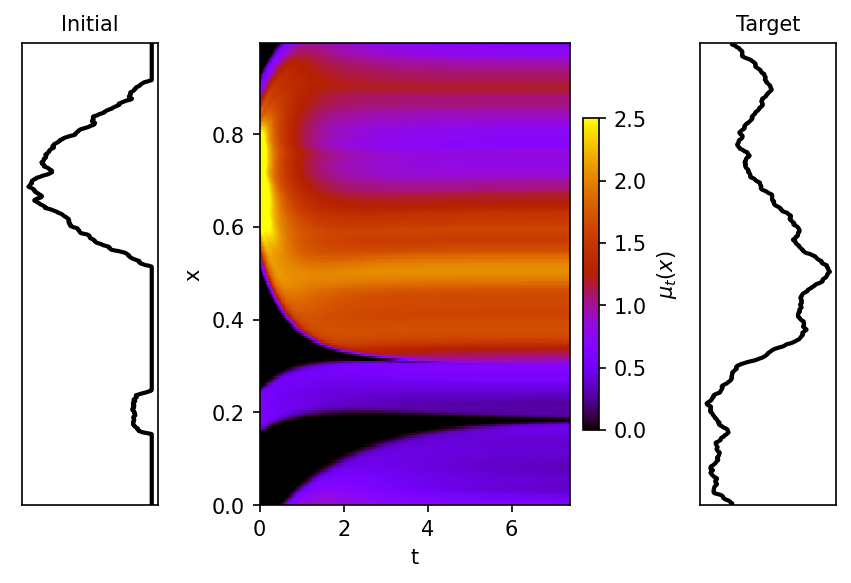}
    \caption{Exponential filling of holes}
  \end{subfigure}
  \caption{The case $s=1$ and $d=1$ integrated with finite volume discretization with upwind scheme. \textbf{(a)} For $s=1$, the regularity of $\bar \mu$ and $\nu$ do not impact the exponential convergence rate. Here $\gamma_0$ (resp. $\gamma_\nu$) is the largest scalar such that $\bar \mu \in \dot{H}^{\gamma_0}$ (resp. $\nu \in \dot{H}^{\gamma_\nu}$) and the densities are lower-bounded by $0.2$ so our theoretical rate is $O(e^{-0.2t})$. \textbf{(b)} As theory predicts, the convergence rate is upper-bounded by the minimum density of $\nu$, and independent of the minimum density of $\bar \mu$. Here $m_0$ (resp.~$m_\nu$) indicates the minimum of $\bar \mu$ (resp.~$\nu$). In this experiment, $(\gamma_0,\gamma_\nu)=(1,1)$. \textbf{(c)} When $\nu$ has a positive lower bound, the zero-density areas of $\bar \mu$ (in black) shrink exponentially fast, as theory predicts (here $(\gamma_0,\gamma_\nu)=(1,1)$); see \Cref{rmk:exponential-filling} and \Cref{lem:exponential-filling-holes}. } 
  \label{fig:s=1}
\end{figure}

\begin{figure}[t]
  \centering
  \begin{subfigure}{0.49\textwidth}
    \centering
    \includegraphics[width=\linewidth]{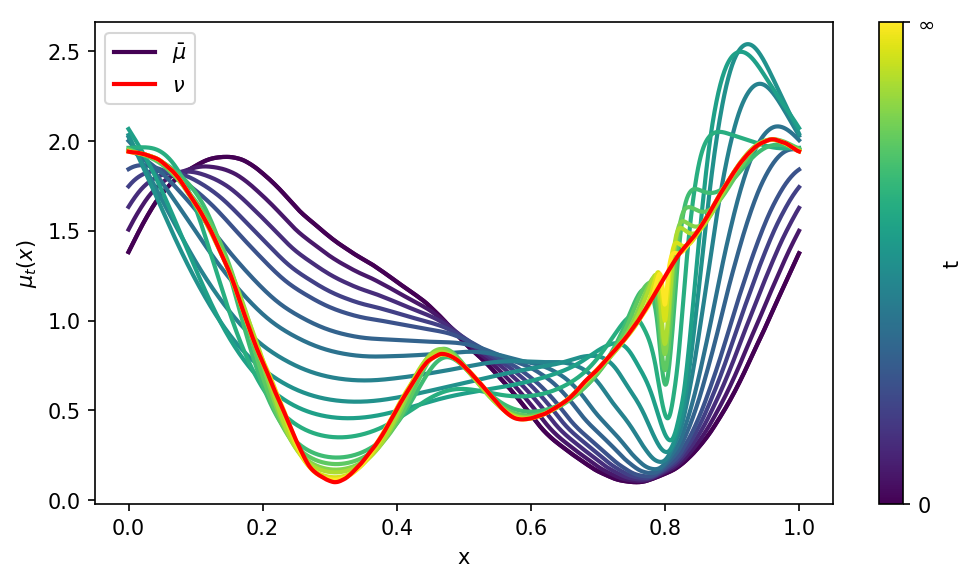}
    \caption{Evolution of the density}
  \end{subfigure}\hfill
  \begin{subfigure}{0.49\textwidth}
    \centering
    \includegraphics[width=\linewidth]{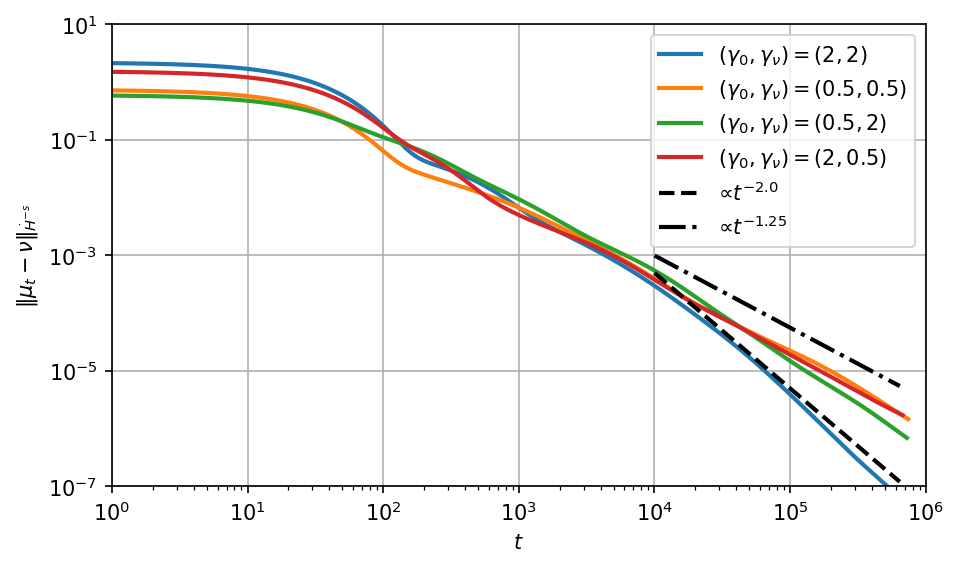}
    \caption{Varying source and target lower bound}
  \end{subfigure}
  \caption{The case $s=2$ and $d=1$ integrated with finite volume discretization with upwind scheme. We write $\gamma_0$ (resp. $\gamma_\nu$) for the largest scalar such that $\bar \mu \in \dot{H}^{\gamma_0}$ (resp. $\nu \in \dot{H}^{\gamma_\nu}$) and the densities are lower-bounded by $0.2$. (a) Snapshots of the density of $\mu_t$ along the evolution (here $(\gamma_0,\gamma_\nu)=(2,2)$). We can observe the absence of a maximum principle (cf.~$x=0.9$), and the appearance of high frequency components (cf.~$x=0.8$). Both phenomena are absent from the case $s=1$. (b) Convergence rate in $\dot{H}^{-2}$ norm. The dashed lines show approximate fits of the asymptotic rates.
  }
  \label{fig:s=2}
\end{figure}

\begin{figure}
    \centering
    \includegraphics[width=0.5\linewidth]{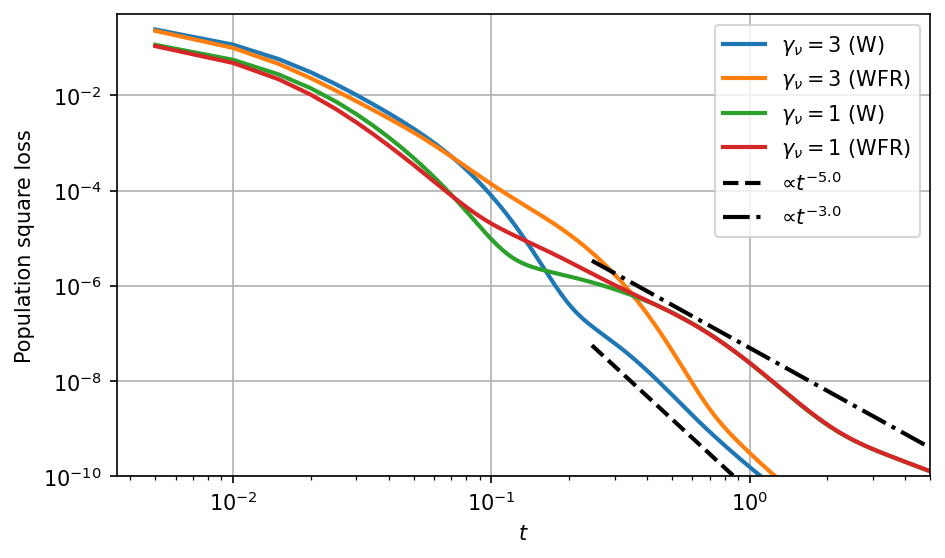}
  \caption{The case of shallow ReLU Neural Network with $d=1$, implemented via gradient descent on the population loss with a ``student'' and ``teacher'' neural network of width $800$ that discretize the measure $\mu_t$ and $\nu$ respectively (in other words, this is an interacting particle system approximation of the PDE with $800$ particles for each measure). We initialize $\bar \mu$ with a uniform density ($\gamma_0=+\infty$) and vary the target regularity (we indicate the largest $\gamma_\nu$ such that $\nu \in \dot{H}^{\gamma_\nu}$). We compare the Wasserstein (W) and Wasserstein--Fisher--Rao (WFR) dynamics. Although (WFR)'s energy decay is slightly faster at initialization (thanks to the extra term in the energy dissipation formula~\eqref{eq:energy-dissipation-Fisher-Rao}), it does not converge faster in general. 
  The dashed lines show approximate fits of the asymptotic rates.
  }
  \label{fig:relu}
\end{figure}

\subsection{Ideas of the proof}\label{subsec:idea-proof}

Let us present the main ideas behind the proofs of our main results, Theorems~\ref{thm:convergence-s=1} and \ref{thm:convergence-s>1}. 

In order to obtain a quantitative energy decay rate we search for a local \L{}ojasiewicz gradient inequality along the evolution $\mu_{t}$, namely
    \begin{equation}\label{eq:Polyak-Lojasiewicz-idea-proof-intro}
        \int_{\T^{d}}|\nabla K_{s}*(\mu_{t}-\nu)|^{2}\mu_{t}\ge c\lVert \mu_{t}-\nu \rVert_{\dot H^{-s}}^{2\beta}\qquad \forall t\ge 0,\quad c>0,\,\beta\ge 1.
    \end{equation}
    In fact, \eqref{eq:energy-dissipation-identity} and \eqref{eq:Polyak-Lojasiewicz-idea-proof-intro} together yield for $\lVert \mu_{t}-\nu\rVert_{\dot H^{-s}}^{2}$ either exponential decay $\sim e^{-ct}$ if $\beta=1$, or polynomial decay $\sim (1+Ct)^{-1/(\beta-1)}$ if $\beta>1$. We point out that such \L{}ojasiewicz gradient inequalities do not hold globally in the space of measures for our energy functionals, but only, as we will see, inside some proper subregion of it. The main difficulty of the proof consists in showing that under suitable assumptions on initial and target measures, the gradient flow remains trapped in such good region. This requires a non-trivial interplay between the energy dissipation identity \eqref{eq:energy-dissipation-identity} and some fine energy estimates for higher order derivatives (see \Cref{lem:energy-estimates}). 

    When $s=1$, \eqref{eq:Polyak-Lojasiewicz-idea-proof-intro} holds with $\beta=1$ and $c=\alpha>0$, thanks to the maximum principle $\inf \mu_{t}\ge \alpha$. In the case $s>1$, however, the same simple estimate does not work for two reasons. First, the maximum principle does not hold for $s>1$; second, even if we had $\inf \mu_{t}\ge c>0$, the resulting quantity in the right-hand side of \eqref{eq:Polyak-Lojasiewicz-idea-proof-intro} would rather involve the lower order norm $\lVert \mu_{t}-\nu\rVert_{\dot H^{1-2s}}^{2}$. These obstructions make the analysis of the case $s>1$ substantially different from the case $s=1$. 
  
    If we search for an inequality like \eqref{eq:Polyak-Lojasiewicz-idea-proof-intro} along the flow, we are thus forced to assume some higher $\dot H^{\gamma}$-regularity on $\mu_{t}-\nu$ for some sufficiently large $\gamma$, and interpolate homogeneous norms as follows:
    \begin{equation}\label{eq:interp-homogeneous-norms-idea-proof-intro}
        \lVert \mu_{t}-\nu\rVert_{\dot H^{-s}}\le \lVert \mu_{t}-\nu\rVert_{\dot H^{1-2s}}^{1-\theta}\lVert \mu_{t}-\nu\rVert_{\dot H^{\gamma}}^{\theta},\qquad \theta= \frac{s-1}{\gamma+2s-1}.
    \end{equation}
    From \eqref{eq:interp-homogeneous-norms-idea-proof-intro} we obtain
    \begin{equation*}
        \int_{\T^{d}}|\nabla K_{s}*(\mu_{t}-\nu)|^{2}\mu_{t}\ge \inf \mu_{t}\left(\lVert \mu_{t}-\nu\rVert_{\dot H^{\gamma}}^{2}\right)^{1-\beta}\left(\lVert \mu_{t}-\nu \rVert_{\dot H^{-s}}^{2} \right)^{\beta},\qquad \beta= 1+\frac{s-1}{\gamma+s}.
    \end{equation*}
    Therefore, in order to get \eqref{eq:Polyak-Lojasiewicz-idea-proof-intro}, and consequently the polynomial decay of $\lVert \mu_{t}-\nu\rVert_{\dot H^{-s}}^{2}$ with exponent $-(\gamma+s)/(s-1)$, we need to make sure that $\inf \mu_{t}$ and $\lVert \mu_{t}-\nu\rVert_{\dot H^{\gamma}}$ remain bounded from below and above, respectively, for all $t\ge 0$. This is the point where energy estimates for the higher norm $\lVert \mu_{t}-\nu\rVert_{\dot H^{\gamma}}^{2}$ come crucially into play, together with the smallness assumption on the initial discrepancy. 

   For simplicity of exposition we illustrate the main mechanism of propagation in time for the \L{}ojasiewicz gradient inequality looking at the linearized equation $\partial_{t}\sigma_{t}=\diver\left(\nu \nabla (-\Delta)^{-s}\sigma_{t}\right)$ for the perturbation $\sigma_{t}:= \mu_{t}-\nu$. The evolution of $\lVert \sigma_{t}\rVert_{\dot H^{\gamma}}^{2}$ is governed by the formula
   \begin{equation*}
       \frac{d}{dt}\lVert \sigma_{t}\rVert_{\dot H^{\gamma}}^{2}=-2\int_{\T^{d}}\nu \nabla (-\Delta)^{-s}\sigma_{t}\cdot \nabla (-\Delta)^{\gamma}\sigma_{t}.
   \end{equation*}
   In the easiest case $\nu =1$, we may simply integrate by parts a power $(\gamma+s)/2$ of the Laplacian, to see that the right-hand side equals $-2\lVert \sigma_{t}\rVert_{\dot H^{\gamma-s+1}}^{2}$.
   We argue similarly in the case when $\nu$ is not constant, but we have to deal with the error terms resulting from derivatives falling on $\nu$. Thanks to Kato--Ponce commutator estimates (that we extend to the periodic setting in \Cref{subsec:appendix-katoponce}), it turns out that these error terms can be controlled by lower order norms of $\sigma_{t}$, provided that some higher regularity on $\nu$ is assumed:
   \begin{equation*}
       \frac{d}{dt}\lVert \sigma_{t}\rVert_{\dot H^{\gamma}}^{2}\le -2\big(\min_{\T^{d}}\nu\big)\lVert \sigma_{t}\rVert_{\dot H^{\gamma-s+1}}^{2}+C\lVert \nu\rVert_{H^{\gamma+s}}\lVert \sigma_{t}\rVert_{\dot H^{\gamma-s+1}}\lVert \sigma_{t}\rVert_{\dot H^{\gamma-s}}.
   \end{equation*}
   Assuming $\nu>0$, interpolating among homogeneous norms, and using Young's inequality, we can absorb the dangerous part of the error term in the right-hand side. Considering also the dissipation identity for $\lVert \sigma_{t}\rVert_{\dot H^{-s}}^{2}$, we end up with the following system of differential inequalities: 
   \begin{equation}\label{eq:system-differential-inequalities-intro}
       \frac{d}{dt}\lVert \sigma_{t}\rVert_{\dot H^{\gamma}}^{2}\le C\lVert \sigma_{t}\rVert_{\dot H^{-s}}^{2},\qquad 
       \frac{d}{dt}\lVert \sigma_{t}\rVert_{\dot H^{-s}}^{2}\le -2\min_{\T^{d}}\nu\left(\lVert \sigma_{t}\rVert_{\dot H^{\gamma}}^{2}\right)^{1-\beta}\left(\lVert \sigma_{t}\rVert_{\dot H^{-s}}^{2}\right)^{\beta},\qquad \beta= 1+\frac{s-1}{\gamma+s}.
   \end{equation}
   When $\gamma$ is sufficiently large, \eqref{eq:system-differential-inequalities-intro} gives uniform boundedness in time for $\lVert \sigma_{t}\rVert_{\dot H^{\gamma}}^{2}$ and polynomial decay with exponent $-(\gamma+s)/(s-1)$ for $\lVert \sigma_{t}\rVert_{\dot H^{-s}}^{2}$, thus concluding the desired convergence result in the linearized setting. Finally, in order to keep nonlinear effects under control,  we need to assume some smallness of the initial perturbation $\sigma_{0}=\bar\mu-\nu$, leading to the locality resulting in \Cref{thm:convergence-s>1}.

\section{Well-posedness theory}\label{sec:well-posedness}
This section is devoted to the proof of \Cref{prop:local-well-posedness-intro}. After fixing notation and introducing the notion of solution to \eqref{eq:PDE-GF} in \Cref{subsec:preliminaries}, we establish in \Cref{subsec:local-well-posedness-weak-class} existence and uniqueness of a maximal solution to \eqref{eq:PDE-GF} in the space $\mathscr{X}_{s}(\T^{d})$, together with a continuation criterion. Finally, in \Cref{subsec:propagation-regularity} we prove propagation of H\"{o}lder and Sobolev regularity from the data throughout the maximal interval of existence.
\subsection{Preliminaries}\label{subsec:preliminaries}
In this section we introduce the notation and basic definitions used throughout the paper in connection with solutions to \eqref{eq:PDE-GF}. In \Cref{subsec:Sobolev-Riesz-Torus} we recall homogeneous Sobolev spaces and Riesz kernels on the torus. In \Cref{subsec:Gradient-flows-Riesz} we define the Riesz kernel mean discrepancy functional $\mathscr{E}^{\nu}_{s}$ and the associated notion of Wasserstein gradient flow on the space of probability measures. Finally, in \Cref{subsec:active-scalar-Riesz} we introduce weak solutions to the active-scalar equation \eqref{eq:PDE-GF} and show that, whenever they belong to the local well-posedness class $\mathscr{X}_{s}(\T^{d})$, they coincide with Wasserstein gradient flows of $\mathscr{E}^{\nu}_{s}$.

\subsubsection{Sobolev spaces and Riesz kernels on $\T^{d}$}\label{subsec:Sobolev-Riesz-Torus}
Let $d\ge 1$ be an integer, and let $\T^d \cong \R^d / \Z^d$ denote the $d$-dimensional torus. 
We consider the standard Fourier orthonormal basis of $L^{2}(\T^{d})$ given by $\{e^{2\pi i k\cdot x}\}_{k\in \Z^{d}}$.
For any periodic distribution $f\in \mathscr{D}'(\T^{d})$ and any $k\in \Z^{d}$, we denote by $\hat{f}_{k}=\langle f,e^{2\pi i k\cdot x}\rangle\in \mathbb{C}$ the $k$-th Fourier coefficient of $f$,
where $\langle \cdot, \cdot\rangle$ is the duality pairing, and we use the notation $\sum_{k\in \Z^{d}}\hat{f}_{k}e^{2\pi i k\cdot x}$ to express $f$ in terms of its Fourier expansion. 

For $\gamma \in \R$, the homogeneous $\gamma$-Sobolev seminorm of $f$ is defined by
\begin{equation*}
    \lVert f\rVert_{\dot H^{\gamma}(\T^{d})}:= \left(\sum_{k\in \Z^{d}\setminus \{0\}}(2\pi |k|)^{2\gamma}|\hat{f}_{k}|^{2}\right)^{1/2}\qquad \forall f\in \mathscr{D}'(\T^{d}),\quad \forall \gamma\in \R.
\end{equation*}
 The homogeneous Sobolev space $ \dot H^{\gamma}(\T^{d})$ consists of zero-mean distributions for which $\lVert \cdot \rVert_{\dot H^{\gamma}}$ is finite.
When $\gamma\ge 0$, we also consider the (inhomogeneous) Sobolev space $H^{\gamma}(\T^{d})\subseteq L^{2}(\T^{d})$ defined by
\begin{equation*}
    H^{\gamma}(\T^{d}):= \left\{f\in L^{2}(\T^{d}):  \lVert f\rVert_{H^{\gamma}(\T^{d})}:=\lVert f\rVert_{L^{2}(\T^{d})}+\lVert f\rVert_{\dot H^{\gamma}(\T^{d})}<\infty\right\}\qquad \forall \gamma\ge 0.
\end{equation*}

For $s>0$, the fractional Laplacian $(-\Delta)^{s}$ is the Fourier multiplier defined on distributions by
\begin{equation*}
    (-\Delta)^{s}f:= \sum_{k\in \Z^{d}}(2\pi |k|)^{2s}\hat{f}_{k}e^{2\pi i k\cdot x}\qquad \forall f\in \mathscr{D}'(\T^{d}),\quad \forall s>0. 
\end{equation*}
The fundamental solution of $(-\Delta)^{s}$ with zero average (i.e.~the distributional solution of $(-\Delta)^{s}K_{s}=\delta_{0}-1$ in $\T^{d}$, where $\delta_{0}$ denotes the Dirac delta at $0$) is the Riesz kernel $K_{s}\in \mathscr{D}'(\T^{d})$, defined by the Fourier expansion
\begin{equation}\label{eq:def-Riesz-kernel}
    K_{s}(x):= \sum_{k\in \Z^{d}\setminus \{0\}}(2\pi |k|)^{-2s}e^{2\pi i k\cdot x}\qquad \forall s>0.
\end{equation}
Consequently, for every $f\in \dot H^{-s}(\T^{d})$, the zero-mean solution of $(-\Delta)^{s}u=f$ in $\T^{d}$ is given by $u=(-\Delta)^{-s}f:= K_{s}*f \in \dot H^{s}(\T^{d})$, where the convolution $*$ is understood in the sense of distributions. In particular, the homogeneous Sobolev norm $\lVert \cdot \rVert_{\dot H^{-s}(\T^{d})}$ can be written as
\begin{equation} \label{eq:seminormKs}
    \|f\|_{\dot{H}^{-s}(\T^{d})}^{2}= \langle f, K_{s}*f\rangle \qquad \forall f\in \dot H^{-s}(\T^{d}).
\end{equation}
We refer to \Cref{subsec-appendix-kernels} for precise regularity properties on $K_{s}$. 

\subsubsection{Gradient flows of Riesz energies in the Wasserstein space of probability measures}\label{subsec:Gradient-flows-Riesz}
Let us recall some notions from the theory of analysis in the space of probability measures and introduce the precise definition of gradient flow for the kernel mean discrepancy (KMD) associated to negative Sobolev norms. We refer the reader to \cite{ambrosio2005gradient} for further details. 

We equip the space $\mathscr{P}(\T^{d})$ of probability measures in the torus with the $2$-Wasserstein distance 
\begin{equation*}
    W_{2}(\mu_{1}, \mu_{2}):= \left(\min_{\gamma \in \Gamma(\mu_{1},\mu_{2})}\int_{\T^{d}\times \T^{d}}|x-y|^{2}d\gamma(x,y)\right)^{1/2},\qquad \mu_{1},\mu_{2}\in \mathscr{P}(\T^{d}),
\end{equation*}
where $\Gamma(\mu_{1},\mu_{2})\subset \mathscr{P}(\T^{d}\times \T^{d})$ is the set of couplings between $\mu_{1}$ and $\mu_{2}$. It is well-known that $(\mathscr{P}(\T^{d}), W_{2})$ is a compact metric space whose notion of convergence coincides with the weak-$*$ convergence of measures. 

Given an interval $I\subseteq \R$, we say that a curve of probability measures $\mu :I\to \mathscr{P}(\T^{d})$ is absolutely continuous, and we write $\mu \in AC(I;\mathscr{P}(\T^{d}))$, if there exists some $g\in L^{1}(I)$ such that 
    \begin{equation}\label{eq:condition-AC-curves-Wasserstein}
        W_{2}(\mu_{b},\mu_{a})\le \int_{a}^{b}g(t)dt\qquad \forall a,b\in I,\, a<b.
    \end{equation}
    In this case, the metric derivative is well-defined
    \begin{equation*}
        |\mu_{t}'|:= \lim_{r\to t}\frac{W_{2}(\mu_{r},\mu_{t})}{|r-t|}<\infty\qquad \text{for a.e.~$t\in I$},
    \end{equation*}
    and $g(t)=|\mu'_{t}|$ is the minimal possible choice for the condition \eqref{eq:condition-AC-curves-Wasserstein} above to hold. 
    
    The following characterization of absolutely continuous curves holds: $\mu:I\to \mathscr{P}(\T^{d})$ is absolutely continuous if and only if it is continuous for the weak-$*$ topology and there exists a Borel vector field $v_{t}(x):I\times \T^{d}\to \R^{d}$ such that $v_{t}\in L^{2}(\T^{d},\mu_{t};\R^{d})$ for a.e.~$t\in I$ and the continuity equation
    \begin{equation*}
        \partial_{t}\mu+\diver(\mu v)=0\qquad \text{in $I\times \T^{d}$}
    \end{equation*}
    is solved in the distributional sense, equivalently
    \begin{equation*}
    \int_{\T^{d}}\varphi d\mu_{b}= \int_{\T^{d}}\varphi d\mu_{a} +\int_{a}^{b}\int_{\T^{d}}\nabla \varphi \cdot v_{r}d\mu_{r}dr\qquad \forall a,b\in I,\, a<b,\qquad  \forall \varphi \in C^{\infty}(\T^{d}).
    \end{equation*}
    In this case, there exists a unique choice of $v$, called the tangent vector field of $\mu$, such that $$\lVert v_{t}\rVert_{L^{2}(\T^{d},\mu_{t};\R^{d})}=|\mu'_{t}|\qquad \text{for a.e.~$t\in I$}.$$

    Let us now introduce the class of energy functionals on $\mathscr{P}(\T^{d})$ considered in this paper. For every $s\ge 1$ and every given probability measure $\nu \in \mathscr{P}(\T^{d})$, we define
 $\mathscr{E}_{s}^{\nu}:\mathscr{P}(\T^{d})\to [0,\infty]$ as 

\begin{equation*}
    \mathscr{E}_{s}^{\nu}(\mu):= 
        \frac{1}{2}\lVert \mu-\nu \rVert_{\dot H^{-s}}^{2}
        =\frac{1}{2}\Bigg(\sup_{\substack{f\in C^{\infty}(\T^{d})\\ \lVert f\rVert_{\dot H^{s}}\le 1}}\int_{\T^{d}}f(d\mu-d\nu)\Bigg)^{2},
\end{equation*}
where the last equality follows by the duality $\dot H^{s}-\dot H^{-s}$ and the density of smooth functions in $\dot H^{s}$.
The functional $\mathscr{E}^{\nu}_{s}$ is proper ($\mathscr{E}^{\nu}_{s}(\nu)=0$), and lower semicontinuous with respect to the Wasserstein metric. Indeed, as seen above, it can be written in terms of the supremum of a family of continuous functions.

For every $\mu \in \{\mathscr{E}^{\nu}_{s}<\infty\}$, the subdifferential of $\mathscr{E}^{\nu}_{s}$ at $\mu$ is the set $\partial \mathscr{E}^{\nu}_{s}(\mu)$ of all vector fields $\xi \in L^{2}(\T^{d}, \mu;\R^{d})$ such that 
    \begin{equation}\label{eq:def-subdifferential-Es}
        \mathscr{E}^{\nu}_{s}(\eta)-\mathscr{E}^{\nu}_{s}(\mu)\ge \inf_{\gamma\in \Gamma_{\opt}(\mu,\eta)}\int_{\T^{d}\times \T^{d}}\xi(x)\cdot (y-x)d\gamma(x,y)+o\left(W_{2}(\mu, \eta)\right)\qquad \forall \eta\in \mathscr{P}(\T^{d}),
    \end{equation}
    where $\Gamma_{\opt}(\mu, \eta)\subset \mathscr{P}(\T^{d}\times \T^{d})$ is the set of optimal couplings between $\mu$ and $\eta$. This corresponds to the notion of reduced subdifferential from \cite[Equation 10.3.12]{ambrosio2005gradient}. 
Next, we give the definition of gradient flow for the energy $\mathscr{E}^{\nu}_{s}$:
\begin{defn}\label{def:Wasserstein-gradient-flow-H-s}
    We say that $\mu \in AC([0,T);\mathscr{P}(\T^{d}))$ is a gradient flow of $\mathscr{E}^{\nu}_{s}$ starting at $\bar\mu\in \mathscr{P}(\T^{d})$ if $\mu_{0}=\bar\mu$ and the tangent vector field $v$ of $\mu$ satisfies
    \begin{equation*}
        -v_{t}\in \partial \mathscr{E}^{\nu}_{s}(\mu_{t})\qquad \text{for a.e.~$t\in (0,T)$}.
    \end{equation*}
\end{defn}
\noindent This notion corresponds precisely to \cite[Definition 11.1.1]{ambrosio2005gradient} applied to the functional $\mathscr{E}^{\nu}_{s}$.

\subsubsection{The active-scalar equation}\label{subsec:active-scalar-Riesz}

Let us introduce the notion of weak solution to \eqref{eq:PDE-GF} that we consider in this paper. For any interval $I\subseteq \R$ and any curve of probability measures $I\ni t\mapsto \mu_{t}\in \mathscr{P}(\T^{d})$, we write $\mu \in C_{w^{*}}(I; \mathscr{P}(\T^{d}))$ if $\mu$ is continuous in time with respect to the weak-$*$ topology.  

\begin{defn}\label{def:solution-active-scalar-equation}
    Let $\bar\mu, \nu \in \mathscr{P}(\T^{d})$ and $s\in [1,\infty)$.
We say that a curve of probability measures $\mu \in C_{w^{*}}([0,T); \mathscr{P}(\T^{d}))$ solves the active-scalar equation \eqref{eq:PDE-GF}
if $\mu_{0}=\bar\mu$, $v_{t}:= -\nabla K_{s}*(\mu_{t}-\nu)$ satisfies $v\in L^{1}_{\loc}((0,T)\times \T^{d},\mu;\R^{d})$, and the continuity equation $\partial_{t}\mu+\diver(\mu v)=0$ holds in the distributional sense, equivalently
\begin{equation*}
    \int_{\T^{d}}\varphi d\mu_{t}= \int_{\T^{d}}\varphi d\bar\mu +\int_{0}^{t}\int_{\T^{d}}\nabla \varphi \cdot v_{r}d\mu_{r}dr\qquad \forall t\in (0,T),\quad \forall \varphi \in C^{\infty}(\T^{d}).
\end{equation*}
\end{defn}

For every $s\ge 1$, the function space in which the local well-posedness of \eqref{eq:PDE-GF} takes place appears in \eqref{eq:def-yudovich-class}.  
We recall that for every $p\in [1,\infty)$ and $q\in [1,\infty]$, the Lorentz space $L^{p,q}(\T^{d})$ is defined as the set of functions $f:\T^{d}\to \R$ such that 
            \begin{equation}\label{eqn:Lorentz}
                \lVert f\rVert_{L^{p,q}(\T^{d})}:= p^{1/q}\lVert \lambda\mathscr{L}^{d}(\left\{|f|>\lambda\right\})^{1/p}\rVert_{L^{q}((0,\infty),d\lambda/\lambda)}<\infty.
            \end{equation}
             The following are some known properties of Lorentz spaces: $L^{p,p}=L^{p}$ and $\lVert f\rVert_{L^{p,p}}=\lVert f\rVert_{L^{p}}$; $L^{p,q}\subset L^{p,r}$ if $q\le r$ and $\lVert f\rVert_{L^{p,r}}\lesssim_{p,q,r}\lVert f\rVert_{L^{p,q}}$; $L^{r}\subset L^{p,q}$ for every $r>p$.
In the sequel, we will consider solutions of \eqref{eq:PDE-GF} according to \Cref{def:solution-active-scalar-equation} that are locally bounded in time in the space $\mathscr{X}_{s}(\T^{d})$, namely $\mu\in L^{\infty}_{\loc}([0,T);\mathscr{X}_{s}(\T^{d}))$. This does not impose any further condition when $s\ge d/2+1$, as by definition $\mu_{t}$ is a probability measure at each time. We will typically deal with solutions extended up to the maximal time of existence:

\begin{defn}\label{def:maximal-solutions}
    Let $s\in [1,\infty)$, $\bar\mu, \nu \in \mathscr{P}\cap\mathscr{X}_{s}(\T^{d})$, and $\mu\in L^{\infty}_{\loc}([0,T);\mathscr{X}_{s}(\T^{d}))$ be a solution of \eqref{eq:PDE-GF} according to \Cref{def:solution-active-scalar-equation}. We say that $\mu$ is a maximal solution of \eqref{eq:PDE-GF} in the space $\mathscr{X}_{s}(\T^{d})$ if the following holds: let $\tilde{T}\ge T$ and $\tilde{\mu}\in L^{\infty}_{\loc}([0,\tilde{T});\mathscr{X}_{s}(\T^{d}))$ be a solution of \eqref{eq:PDE-GF} such that $\tilde{\mu}_{t}=\mu_{t}$ for all $t\in [0,T)$. Then $\tilde{T}=T$ and $\tilde{\mu}=\mu$.
\end{defn}

We now show that solutions of \eqref{eq:PDE-GF} in the class $L^{\infty}_{\loc}([0,T);\mathscr{X}_{s}(\T^{d}))$ are Wasserstein gradient flows of $\mathscr{E}^{\nu}_{s}$ according to \Cref{def:Wasserstein-gradient-flow-H-s}. 
We start with the following lemma, which ensures that the vector field $v_{t}$ prescribed by the active-scalar equation \eqref{eq:PDE-GF} is in the subdifferential of $\mathscr{E}^{\nu}_{s}$ at $\mu_{t}$, as long as $\mu_{t}\in \mathscr{X}_{s}(\T^{d})$. 
\begin{lem}\label{lem:nabla-Ks-in-subdifferential}
    Let $s\ge 1$ and $\mu,\nu \in \mathscr{P}\cap \mathscr{X}_{s}(\T^{d})$. Then, $\nabla K_{s}*(\mu-\nu)\in \partial \mathscr{E}^{\nu}_{s}(\mu)$.
\end{lem}
\begin{proof}
    We denote $\xi:= \nabla K_{s}*(\mu-\nu)$. By \Cref{lem:modulus-continuity-vector-field} and \Cref{lem:W2-Linfty}, $\xi\in C(\T^{d};\R^{d})\subset L^{2}(\T^{d},\mu_{t};\R^{d})$ has modulus of continuity 
    \begin{equation*}
        |\xi(x)-\xi(y)|\lesssim_{d,s,M}|x-y|\log(2+|x-y|^{-1}),\qquad M:=\max\{\lVert \mu\rVert_{\mathscr{X}_{s}},\lVert \nu\rVert_{\mathscr{X}_{s}}\}.
    \end{equation*}
    In particular, by the mean value theorem, we have
    \begin{equation}\label{eq:taylor-expansion-subgradient}
        |K_{s}*(\mu-\nu)(y)-K_{s}*(\mu-\nu)(x)-\xi(x)\cdot (y-x)|\lesssim_{d,s,M}|x-y|^{2}\log(2+|x-y|^{-1})\qquad \forall x,y \in \T^{d}.  
    \end{equation}
    Let us check that condition \eqref{eq:def-subdifferential-Es} holds. Take $\eta\in \mathscr{P}(\T^{d})$ such that $\mathscr{E}^{\nu}_{s}(\eta)<\infty$ and $\gamma \in \Gamma_{\opt}(\mu,\eta)$. Then, using \eqref{eq:seminormKs} we find
    \begin{align*}
        \mathscr{E}^{\nu}_{s}(\eta)-\mathscr{E}^{\nu}_{s}(\mu)&= \frac{1}{2}\lVert \eta-\nu\rVert_{\dot H^{-s}}^{2}-\frac{1}{2}\lVert \mu-\nu\rVert_{\dot H^{-s}}^{2}=\langle K_{s}*(\mu-\nu), \eta-\mu\rangle+\frac{1}{2}\lVert \eta-\mu\rVert_{\dot H^{-s}}^{2}\\
        &\ge \langle K_{s}*(\mu-\nu), \eta-\mu\rangle= \int_{\T^{d}\times \T^{d}}\left(K_{s}*(\mu-\nu)(y)-K_{s}*(\mu-\nu)(x)\right)d\gamma(x,y)\\
        &\ge \int_{\T^{d}\times \T^{d}}\xi(x)\cdot (y-x)d\gamma(x,y)-C(d,s,M)\int_{\T^{d}\times \T^{d}}|x-y|^{2}\log(2+|x-y|^{-1})d\gamma(x,y)\\
        &\ge \int_{\T^{d}\times \T^{d}}\xi(x)\cdot (y-x)d\gamma(x,y)-C(d,s, M)W_{2}^{2}(\mu,\eta)\log\left(2+W_{2}^{-1}(\mu,\eta)\right),
    \end{align*}
    where in the last two steps we used \eqref{eq:taylor-expansion-subgradient} and Jensen's inequality for the concave function $t\mapsto t \log(2+t^{-1/2})$. By the arbitrariness of $\gamma \in \Gamma_{\opt}(\mu,\eta)$, we obtain \eqref{eq:def-subdifferential-Es}.
\end{proof}

\begin{prop}\label{prop:gradient-flow-structure}
    Let $s\ge 1$, $\bar\mu, \nu \in \mathscr{P}\cap \mathscr{X}_{s}(\T^{d})$, and $\mu \in L^{\infty}_{\loc}([0,T);\mathscr{X}_{s}(\T^{d}))$ be a solution of \eqref{eq:PDE-GF} according to \Cref{def:solution-active-scalar-equation}. Then, $\mu$ is a gradient flow of $\mathscr{E}^{\nu}_{s}$ with initial datum $\bar\mu$ according to \Cref{def:Wasserstein-gradient-flow-H-s}. Moreover, $\mu \in \Lip_{\loc}\left([0,T); \left(\mathscr{P}(\T^{d}),W_{2}\right)\right)$ and the energy dissipation identity \eqref{eq:energy-dissipation-identity} holds.
\end{prop}
\begin{proof}
    Let $v_{t}=-\nabla K_{s}*(\mu_{t}-\nu)$ be the vector field generated by the solution $\mu$. By \Cref{lem:nabla-Ks-in-subdifferential}, $v_{t}\in -\partial \mathscr{E}^{\nu}_{s}(\mu_{t})\subset L^{2}(\T^{d},\mu_{t};\R^{d})$ for all $t\in [0,T)$, thus $\mu\in AC([0,T);\mathscr{P}(\T^{d}))$ and it is a gradient flow of $\mathscr{E}^{\nu}_{s}$ starting at $\bar\mu$ according to \Cref{def:Wasserstein-gradient-flow-H-s}. Moreover, since $\mu \in L^{\infty}_{\loc}([0,T);\mathscr{X}_{s}(\T^{d}))$, by \Cref{lem:W2-Linfty} $v$ is locally bounded in $[0,T)\times \T^{d}$, which implies $\mu \in \Lip_{\loc}\left([0,T); \left(\mathscr{P}(\T^{d}),W_{2}\right)\right)$.
    To prove \eqref{eq:energy-dissipation-identity}, we pick $t\in (0,T)$ and compute, using \eqref{eq:seminormKs} and equation \eqref{eq:PDE-GF}:
    \begin{align*}
        \frac{\mathscr{E}^{\nu}_{s}(\mu_{t+h})-\mathscr{E}^{\nu}_{s}(\mu_{t})}{h}&=\left\langle K_{s}*(\mu_{t}-\nu)+\frac{1}{2}K_{s}*(\mu_{t+h}-\mu_{t}),\frac{\mu_{t+h}-\mu_{t}}{h}\right\rangle\\
        &= \fint_{t}^{t+h}\left\langle K_{s}*(\mu_{t}-\nu)+\frac{1}{2}K_{s}*(\mu_{t+h}-\mu_{t}),\partial_{r}\mu_{r}\right\rangle dr\\
        &= \fint_{t}^{t+h}\left\langle K_{s}*(\mu_{t}-\nu)+\frac{1}{2}K_{s}*(\mu_{t+h}-\mu_{t}),\diver(\mu_{r}\nabla K_{s}*(\mu_{r}-\nu))\right\rangle dr\\
        &= -\fint_{t}^{t+h}\left\langle \nabla K_{s}*(\mu_{t}-\nu)+\frac{1}{2}\nabla K_{s}*(\mu_{t+h}-\mu_{t}),\mu_{r}\nabla K_{s}*(\mu_{r}-\nu)\right\rangle dr.
    \end{align*}
    Now, by \Cref{lem:W2-Linfty}, since $\mu$ is weakly-$*$ continuous in time, $\nabla K_{s}*(\mu_{r}-\nu)\rightarrow \nabla K_{s}*(\mu_{t}-\nu)$ uniformly as $r\to t$ and $\nabla K_{s}*(\mu_{t+h}-\mu_{t})\rightarrow 0$ uniformly as $h\to 0$. Therefore, we find
    \begin{equation*}
         \frac{\mathscr{E}^{\nu}_{s}(\mu_{t+h})-\mathscr{E}^{\nu}_{s}(\mu_{t})}{h}=-\int_{\T^{d}}|\nabla K_{s}*(\mu_{t}-\nu)|^{2}d\mu_{t}+o(h)\qquad \text{as $h\to 0$},
    \end{equation*}
    from which \eqref{eq:energy-dissipation-identity} follows. 
\end{proof}

\subsection{Local well-posedness in the weak class} \label{subsec:local-well-posedness-weak-class}
In this section we prove existence and uniqueness of (maximal) solutions to \eqref{eq:PDE-GF} in the class $\mathscr{X}_{s}(\T^{d})$. Specifically, in \Cref{subsec:uniqueness-riesz} we establish a quantitative stability estimate for solutions in $\mathscr{X}_{s}(\T^{d})$, which in particular yields uniqueness. Existence is then proved in \Cref{subsec:existence-riesz}. Finally, \Cref{subsec:maximal-sol-riesz} is devoted to the construction of maximal solutions and to the proof of the continuation criterion.
\subsubsection{Uniqueness and stability}\label{subsec:uniqueness-riesz}

In the following proposition, we show that solutions are uniformly stable in Wasserstein distance on bounded intervals of time with respect to variations of the data. The precise detachment rate of the stability inequality, exponential or double-exponential according to the choice of $s$, is described in \Cref{rmk:rate-stability}.     

\begin{prop}(Uniqueness and stability)\label{prop:stability-yudovich}
Let $s\ge 1$ and let $\mu^{1}, \mu^{2}\in  L^{\infty}([0,T);\mathscr{X}_{s}(\T^{d}))$ be two solutions of \eqref{eq:PDE-GF} with initial and target measures $\bar\mu^{1}, \bar\mu^{2},\nu^{1},\nu^{2}\in \mathscr{P}\cap \mathscr{X}_{s}(\T^{d})$, respectively. Let
\begin{equation*}
    M:= \max \left\{\sup_{t\in [0,T)}\lVert \mu_{t}^{1}\rVert_{\mathscr{X}_{s}},\sup_{t\in [0,T)}\lVert \mu_{t}^{2}\rVert_{\mathscr{X}_{s}}, \lVert \nu^{1}\rVert_{\mathscr{X}_{s}}, \lVert \nu^{2}\rVert_{\mathscr{X}_{s}}\right\}.
\end{equation*}
Then, the following facts hold:
\begin{itemize}
    \item [i)] (\textit{Uniqueness}). If $\bar\mu^{1}=\bar\mu^{2}$ and $\nu^{1}=\nu^{2}$, then $\mu^{1}_{t}=\mu^{2}_{t}$ for all $t\in [0,T)$.
    \item [ii)] (\textit{Stability}). Otherwise, the following bound holds:
\begin{equation*}
    W_{2}(\mu_{t}^{1},\mu_{t}^{2})\le \Omega(t)\qquad \forall t\in [0,T),
\end{equation*}
where $\Omega=\Omega\left(d,s,M,W_{2}(\bar\mu^{1},\bar\mu^{2}),W_{2}(\nu^{1},\nu^{2})):[0,\infty)\to [0,\infty\right)$ is a continuous increasing function  such that $\Omega(0)= W_{2}(\bar\mu^{1},\bar\mu^{2})$. Moreover, $\Omega$ converges uniformly to zero in bounded intervals as $W_{2}(\bar\mu^{1},\bar\mu^{2})$ and $W_{2}(\nu^{1},\nu^{2})$ approach zero. 
\end{itemize}
 \end{prop}

\begin{proof}
Let $v_{t}^{1}=-\nabla K_{s}*(\mu_{t}^{1}-\nu^1)$ and $v_{t}^{2}=-\nabla K_{s}*(\mu_{t}^{2}-\nu^{2})$ be the vector fields generated by the two solutions $\mu^{1}$ and $\mu^{2}$, respectively. By \Cref{lem:modulus-continuity-vector-field},
\begin{align}\label{eq:mod-cont-vect-field-proof-stability}
    |v_{t}^{i}(x)-v_{t}^{i}(y)|\lesssim_{d,s} M\omega_{s}(|x-y|) \qquad \forall x,y\in \T^{d},\quad \forall t\in [0,T),
\end{align}
where the modulus of continuity $\omega_{s}$ from \eqref{eq:def:mod-continuity-vect-field} is Lipschitz or log-Lipschitz according to the choice of $s$. Being in particular $\omega_{s}$ of Osgood-type\footnote{A modulus of continuity $\omega:[0,\infty)\to [0,\infty)$ is said to be of Osgood-type if $\int_{0}^{1}\frac{1}{\omega(r)}dr=\infty.$},
by the results in \cite{ambrosio2008uniqueness}, the two solutions are necessarily Lagrangian, that is, the flow maps $X^{1},X^{2}:[0,T)\times \T^{d}\to \T^{d}$ associated to the vector fields $v^{1},v^{2}$ are well-defined and provide the representation
$$\mu_{t}^{i}=(X_{t}^{i})_{\#}\bar\mu^{i}\qquad \forall t\in [0,T), \quad  i=1,2.$$

Let $\bar \gamma\in \mathscr{P}(\T^{d}\times \T^{d})$ be an optimal coupling for the optimal transport problem with quadratic cost between the initial measures $\bar\mu^{1}$ and  $\bar\mu^{2}$. We consider the coupling $\gamma_{t}:= (X_{t}^{1}, X_{t}^{2})_{\#}\bar\gamma$ between $\mu^{1}_{t}$ and $\mu^{t}_{2}$. We have
$$W_{2}(\mu_{t}^{1},\mu_{t}^{2})\le \left(\int_{\T^{d}}\int_{\T^{d}}|x-y|^{2}d\gamma_{t}(x,y)\right)^{1/2}= \left(\int_{\T^{d}}\int_{\T^{d}}|X_{t}^{1}(x)-X_{t}^{2}(y)|^{2}d\bar\gamma(x, y)\right)^{1/2}=:\rho(t).$$
We will derive a suitable differential inequality for the quantity $\rho(t)$ in the time interval $[0,T)$, which will give the desired control of $W_{2}(\mu_{t}^{1},\mu_{t}^{2})$ after integration. By Cauchy--Schwarz and triangle inequality in $L^{2}(\T^{d}\times \T^{d}, \bar\gamma;\R^{d})$, we get
\begin{align*}
    \rho'(t)&\le \left(\int_{\T^{d}}\int_{\T^{d}}|v_{t}^{1}(X_{t}^{1}(x))-v_{t}^{2}(X_{t}^{2}(y))|^{2}d\bar\gamma(x, y)\right)^{1/2} \\
    &\le \underbrace{\left(\int_{\T^{d}}\int_{\T^{d}}|v_{t}^{1}(X_{t}^{1}(x))-v_{t}^{2}(X_{t}^{1}(x))|^{2}d\bar\gamma(x, y)\right)^{1/2}}_{(I)}+ \underbrace{\left(\int_{\T^{d}}\int_{\T^{d}}|v_{t}^{2}(X_{t}^{1}(x))-v_{t}^{2}(X_{t}^{2}(y))|^{2}d\bar\gamma(x, y)\right)^{1/2}}_{(II)}.
\end{align*}
In order to bound $(I)$, we first use the marginal condition on $\bar \gamma$, change of variables, and the triangle inequality in $L^{2}(\T^{d}, \mu_{t};\R^{d})$ to get
\begin{align*}
    (I)&= \left(\int_{\T^{d}}|v_{t}^{1}-v_{t}^{2}|^{2}d\mu_{t}^{1}\right)^{1/2}
    \le \left(\int_{\T^{d}}|\nabla K_{s}*(\mu_{t}^{1}-\mu_{t}^{2})|^{2}d\mu_{t}^{1}\right)^{1/2}+\left(\int_{\T^{d}}|\nabla K_{s}*(\nu^{1}-\nu^{2})|^{2}d\mu_{t}^{1}\right)^{1/2}.
\end{align*}
Now, in the case $s\in \left[1,d/2+1\right)$ we use H\"{o}lder inequality with exponents $p=d/(2s-2)\in (1,\infty]$ and $q=p'=d/(d-2s+2)\in [1,\infty)$ along with the bound from \Cref{lem:W2-dual} and get
\begin{equation}\label{eq:bound-(I)-s-small-stability}
    \begin{aligned}
    (I)&\le \lVert \mu_{t}^{1}\rVert_{L^{p}}^{1/2}\lVert \nabla K_{s}*(\mu_{t}^{1}-\mu_{t}^{2})\rVert_{L^{2q}}+\lVert \mu_{t}^{1}\rVert_{L^{p}}^{1/2}\lVert \nabla K_{s}*(\nu^{1}-\nu^{2})\rVert_{L^{2q}}\\
    &\lesssim_{d,s} \lVert \mu_{t}^{1}\rVert_{L^{p}}^{1/2}\max\{\lVert \mu_{t}^{1}\rVert_{L^{p}}, \lVert \mu_{t}^{2}\rVert_{L^{p}}\}^{1/2}W_{2}(\mu_{t}^{1},\mu_{t}^{2})+\lVert \mu_{t}^{1}\rVert_{L^{p}}^{1/2}\max\{\lVert \nu^{1}\rVert_{L^{p}}, \lVert \nu^{2}\rVert_{L^{p}}\}^{1/2}W_{2}(\nu^{1},\nu^{2})\\
    &\lesssim MW_{2}(\mu_{t}^{1},\mu_{t}^{2})+MW_{2}(\nu^{1},\nu^{2}).
\end{aligned}
\end{equation}
In the case $s\ge d/2+1$ we use the uniform bound from \Cref{lem:W2-Linfty} together with the unit mass condition on $\mu_{t}^{1}$, and we obtain
\begin{equation}\label{eq:bound-(I)-s-big-stability}
    (I)\lesssim_{d,s}\begin{cases}
        W_{2}(\mu_{t}^{1},\mu_{t}^{2})+W_{2}(\nu^{1},\nu^{2})\quad &\text{if $s> \frac{d}{2}+1$},\\
         W_{2}(\mu_{t}^{1},\mu_{t}^{2})\log\left(2+W_{2}(\mu_{t}^{1},\mu_{t}^{2})^{-1}\right)+W_{2}(\nu^{1},\nu^{2})\log\left(2+W_{2}(\nu^{1},\nu^{2})^{-1}\right)\quad &\text{if $s=\frac{d}{2}+1$}.
    \end{cases}
\end{equation}
To bound $(II)$, instead, we use the modulus of continuity of the vector field $v^{2}_{t}$ from \Cref{lem:modulus-continuity-vector-field} and Jensen's inequality:
\begin{align}\label{eq:bound-(II)-stability}
    (II)\lesssim_{d,s}M\left(\int_{\T^{d}}\int_{\T^{d}}\omega_{s}^{2}(|X_{t}^{1}(x)-X_{t}^{2}(y)|)d\bar\gamma(x,y)\right)^{1/2}
    \le \begin{cases}
        M\rho(t)\quad &\text{if $s\not\in \{1,d/2+1\}$},\\
        M\rho(t)\log\left(2+\rho(t)^{-1}\right)\quad &\text{if $s\in \{1,d/2+1\}$}.
    \end{cases}
\end{align}
Putting together \eqref{eq:bound-(I)-s-small-stability}, \eqref{eq:bound-(I)-s-big-stability} and \eqref{eq:bound-(II)-stability}, and recalling that $W_{2}(\mu_{t}^{1},\mu_{t}^{2})\le \rho(t)$, we finally obtain
\begin{equation}\label{eq:differential-inequality-final-stability}
    \rho'(t)\lesssim_{d,s} \begin{cases}
        M\rho(t)+MW_{2}(\nu^{1},\nu^{2})\quad &\text{if $s\notin \{1,d/2+1\}$},\\
        M\rho(t)\log\left(2+\rho(t)^{-1}\right)+MW_{2}(\nu^{1},\nu^{2})\log\left(2+W_{2}(\nu^{1},\nu^{2})^{-1}\right)\quad &\text{if $s\in \{1,d/2+1\}$}.
    \end{cases}
\end{equation}
The desired conclusion is obtained by integrating this differential inequality  (see \Cref{rmk:rate-stability} for the precise expression of the modulus $\Omega$). 
\end{proof}
\begin{rmk}\label{rmk:rate-stability}
    An explicit expression for the function $\Omega$ from \Cref{prop:stability-yudovich} is obtained by integrating \eqref{eq:differential-inequality-final-stability}. When $s\notin \{1,d/2+1\}$, by Gr\"onwall's inequality, we  get
    \begin{equation*}
        W_{2}(\mu_{t}^{1},\mu_{t}^{2})\le \rho(t)\le 
            W_{2}(\bar\mu^{1},\bar\mu^{2})e^{c(d,s)Mt}+W_{2}(\nu^{1},\nu^{2})\left(e^{c(d,s)Mt}-1\right).
    \end{equation*}
    Notice that the closeness between solutions is lost at most exponentially fast in this case, consistently to the fact that the vector field is uniformly Lipschitz. In the critical cases $s\in \{1,d/2+1\}$, integrating \eqref{eq:differential-inequality-final-stability} exactly is more complicated. However, we can analyze the behavior of $\Omega$ in bounded intervals when $W_{2}(\bar\mu^{1},\bar\mu^{2})$ and $W_{2}(\nu^{1}, \nu^{2})$ are small, noting that, for $\rho(t)+W_{2}(\nu^{1},\nu^{2})\le 1/2$,   
    \begin{align*}
    \left(\rho(t)+W_{2}(\nu^{1},\nu^{2})\right)'&\lesssim_{d,s}
        M\rho(t)\log\left(2+\rho(t)^{-1}\right)+MW_{2}(\nu^{1},\nu^{2})\log\left(2+W_{2}(\nu^{1},\nu^{2})^{-1}\right)\\
        &\lesssim M(\rho(t)+W_{2}(\nu^{1},\nu^{2}))\log\left((\rho(t)+W_{2}(\nu^{1},\nu^{2}))^{-1}\right).
    \end{align*}
    Therefore, 
    \begin{equation*}
        W_{2}(\mu_{t}^{1}, \mu_{t}^{2})\le \rho(t)\le \left(W_{2}(\bar\mu^{1},\bar\mu^{2})+W_{2}(\nu^{1},\nu^{2})\right)^{\exp\left(-c(d,s)Mt\right)}-W_{2}(\nu^{1},\nu^{2})
    \end{equation*}
    for all times $t\ge 0$ for which the right-hand side is bounded by $1/2$. In this case, the detachment is at most double-exponential, in accordance with the log-Lipschitz regularity of the vector field. \fr
\end{rmk}

\subsubsection{Existence of solutions}\label{subsec:existence-riesz}
 Next, we construct solutions of \eqref{eq:PDE-GF} in the space $\mathscr{X}_{s}(\T^{d})$ for short intervals of time. The proof is done by means of a Picard iteration at the Lagrangian flow level.
\begin{prop}[Existence]\label{prop:existence-yudovich}
    Let $s\ge 1$ and $\bar{\mu}, \nu \in \mathscr{P}\cap \mathscr{X}_{s}(\T^{d})$. Then, there exists $T>0$ and a solution $\mu \in L^{\infty}([0,T);\mathscr{X}_{s}(\T^{d}))$ of equation \eqref{eq:PDE-GF}.
\end{prop}
\begin{proof} 
Let $T>0$ be a positive number to be chosen sufficiently small later. We construct recursively a sequence of approximate solutions $\mu^{n}\in L^{\infty}([0,T);\mathscr{X}_{s}(\T^{d}))$ as follows. We set
\begin{equation*}
    \mu^{0}_{t}=\bar\mu\qquad \forall t\in [0,T).
\end{equation*}
Then, for every $n\ge 0$, calling $v^{n}_{t}=-\nabla K_{s}*(\mu_{t}^{n}-\nu)$ and $X^{n}$, respectively, the vector field and the flow map associated to the $n$-th approximate solution $\mu^{n}$, we define
\begin{equation*}
    \mu^{n+1}_{t}:=(X^{n}_{t})_{\#}\bar\mu\qquad \forall t\in [0,T).
\end{equation*}
We will prove that up to choosing $T$ small enough, $\{\mu^{n}\}_{n\ge 0}$ is uniformly bounded in $L^{\infty}([0,T);\mathscr{X}_{s}(\T^{d}))$, pre-compact in $C_{w^{*}}([0,T);\mathscr{P}(\T^{d}))$, and that any weak limit of the sequence is a solution of \eqref{eq:PDE-GF}.

\smallskip
\noindent \textbf{Step 1:}
We first show by induction that
\begin{equation}\label{eq:uniform-bound-existence-wellposedness}
    \lVert \mu^{n}\rVert_{L^{\infty}([0,T); \mathscr{X}_{s})}\le 2 \lVert \bar\mu\rVert_{\mathscr{X}_{s}}\qquad \forall n\ge 0,
\end{equation}
if $T$ is sufficiently small. Observe that this is automatically satisfied when $s\ge d/2+1$, because the total mass is conserved by the continuity equation. Let us address the case $s\in \left[1,d/2+1\right)$. The inequality \eqref{eq:uniform-bound-existence-wellposedness} is trivial for $n=0$. Suppose that it holds for some $n\ge 0$, and let us show it holds for $n+1$. 

If $s=1$, we have $\diver v_{t}^{n}= \mu_{t}^{n}-\nu$, therefore the following bound holds for all $t\in [0,T)$:
\begin{equation*}
    \lVert \mu_{t}^{n+1}\rVert_{L^{\infty}}\le \lVert \bar\mu \rVert_{L^{\infty}}\exp \left(\int_{0}^{t}\lVert \diver v_{r}^{n}\rVert_{L^{\infty}}dr\right)\le \lVert \bar\mu \rVert_{L^{\infty}}\exp \left(T(2\lVert \bar\mu\rVert_{L^{\infty}}+\lVert \nu\rVert_{L^{\infty}})\right)\le 2\lVert \bar\mu \rVert_{L^{\infty}},
\end{equation*}
if $T$ is sufficiently small, where we used the induction hypothesis. In the case $s\in (1,d/2+1)$, instead, by \Cref{lem:modulus-continuity-vector-field}, $v_{t}^{n}$ has a Lipschitz modulus of continuity, which rewrites as
\begin{equation*}
    \lVert \nabla v_{t}^{n}\rVert_{L^{\infty}}\lesssim_{d,s}\lVert \mu_{t}^{n}-\nu \rVert_{L^{p,1}}\qquad \forall t\in [0,T).
\end{equation*}
In particular, for all $t\in [0,T)$, the Lipschitz constants of the flow map $X_{t}^{n}$ and its inverse $(X_{t}^{n})^{-1}$ can be bounded by
\begin{equation*}
    \max\{\lVert \nabla (X_{t}^{n})^{-1}\rVert_{L^{\infty}},\lVert \nabla X_{t}^{n}\rVert_{L^{\infty}}\} \le \exp \left(\int_{0}^{t}\lVert \nabla v_{r}^{n}\rVert_{L^{\infty}}dr\right)\le \exp\left(C(d,s)T(\lVert \bar\mu \rVert_{L^{p,1}}+\lVert \nu\rVert_{L^{p,1}})\right)=:L,
\end{equation*}
Consequently, from the representation formula $\mu^{n+1}_{t}= (\bar\mu \det \nabla X^n_{t})\circ (X^n_{t})^{-1}$,  for all $t\in [0,T)$ we derive
\begin{align*}
    \lVert \mu_{t}^{n+1}\rVert_{L^{p,1}}&=\left\lVert \left(\bar\mu \det \nabla X^n_{t}\right)\circ (X_{t}^{n})^{-1}\right\rVert_{L^{p,1}} \le L^{d}\lVert \bar\mu \circ \left(X_{t}^{n}\right)^{-1}\rVert_{L^{p,1}}\le L^{d+1+1/p}\lVert \bar\mu \rVert_{L^{p,1}}\le 2\lVert \bar\mu \rVert_{L^{p,1}},
\end{align*}
provided that $T$ is chosen sufficiently small, where we used the definition of $L^{p,1}$-quasi-norm \eqref{eqn:Lorentz} together with the fact that 
\begin{align*}
    \left|\{|\bar\mu\circ (X_{t}^{n})^{-1}|>\lambda\}\right|=\left|(X_{t}^{n})^{-1}\left(\{|\bar\mu|>\lambda\}\right)\right|\le L|\{|\bar\mu|>\lambda\}|\qquad \forall \lambda>0,
\end{align*}
which follows from the Lipschitz bound on $(X_{t}^{n})^{-1}$.

\smallskip
\noindent \textbf{Step 2:} Next, by making $T$ even smaller if necessary, we prove that
\begin{equation}\label{eq:convergence-subsequent-approximate-solutions}
    \lim_{n\to \infty }\,\sup_{k\ge n}\, \sup_{t\in [0,T)}\,W_{2}(\mu_{t}^{k},\mu_{t}^{k+1})=0.
\end{equation}
Let $M:= \max\{\lVert\bar\mu\rVert_{\mathscr{X}_{s}}, \lVert \nu\rVert_{\mathscr{X}_{s}}\}$. Then, defining the quantity
\begin{equation*}
    \rho^{n}(t):=\left(\int_{\T^{d}}|X_{t}^{n+1}-X_{t}^{n}|^{2}d\bar \mu\right)^{1/2}\ge W_{2}(\mu_{t}^{n}, \mu_{t}^{n+1}),
\end{equation*}
arguing exactly as in the proof of \Cref{prop:stability-yudovich}, and using the uniform bound from \eqref{eq:uniform-bound-existence-wellposedness}, we find
\begin{equation*}
    \frac{d}{dt}\rho^{n}(t)\lesssim_{d,s} \begin{cases}
        M\rho^{n}(t)+M\rho^{n-1}(t)\quad &\text{if $s\notin \{1,d/2+1\}$},\\
        M\rho^{n}(t)\log\left(2+\frac{1}{\rho^{n}(t)}\right)+M\rho^{n-1}(t)\log\left(2+\frac{1}{\rho^{n-1}(t)}\right)\quad &\text{if $s\in \{1,d/2+1\}$}.
    \end{cases}
\end{equation*}
As a consequence, setting
\begin{equation*}
    \tilde{\rho}^{n}(t):= \sup_{k\ge n} \rho^{k}(t),\qquad \omega_{s}(r):= \begin{cases}
        r\qquad &\text{if $s\not\in \{1,d/2+1\}$},\\
        r\log(2+r^{-1})\qquad &\text{if $s\in \{1,d/2+1\}$},
    \end{cases}
\end{equation*}
we derive
\begin{equation}\label{eq:interative-estimate-MarchioroPulvirenti}
    \tilde{\rho}^{n+1}(t)\le C\int_{0}^{t}\omega_{s}(\tilde{\rho}^{n}(r))dr,\quad C=C(d,s,M)\qquad \forall t\in [0,T),\quad \forall n\ge 0.
\end{equation}
From \eqref{eq:interative-estimate-MarchioroPulvirenti} we will conclude that $\sup_{t\in [0,T)}\tilde{\rho}^{n}(t)\to 0$ as $n\to \infty$, as soon as $T$ is chosen sufficiently small. First, observe that there exists a constant $K>0$ such that
\begin{equation}\label{eq:bound-loglip-linear-MarchioroPulvirenti}
    \omega_{s}(r)\lesssim r\log(2+r^{-1})\le \eps +L_{\eps}r,\qquad L_{\eps}:=K\log(\eps^{-1}),\qquad \forall \eps \in (0,1/2),\quad \forall r\in(0,\infty).
\end{equation}
Given $n\ge 1$ and $\eps \in (0,1/2)$, we can use iteratively \eqref{eq:interative-estimate-MarchioroPulvirenti} combined with \eqref{eq:bound-loglip-linear-MarchioroPulvirenti}, to get, for every $t\in [0,T)$,
\begin{align*}
    \tilde{\rho}^{n}(t)&\le C\int_{0}^{t}\omega_{s}(\tilde{\rho}^{n}(r))dr \le C\int_{0}^{T}(\eps+L_{\eps}\tilde{\rho}^{n}(r))dr\\
    &\le \eps C T+ C^{2}L_{\eps}\int_{0}^{T}\tilde{\rho}^{n-1}(r)dr\\
    &\le \dots \le \eps C T+\eps \frac{C^{2}L_{\eps}T^{2}}{2}+\dots +\eps \frac{C^{n}L_{\eps}^{n-1}T^{n}}{n!}+C^{n}L_{\eps}^{n}\int_{0}^{T}\int_{0}^{t_1}\dots \int_{0}^{t_{n-1}}\tilde{\rho}^{0}(t_{n})dt_{n}dt_{n-1}\dots d_{t_{1}}\\
    &\le \eps CT\sum_{k=0}^{n-1}\frac{(CL_{\eps}T)^{k}}{k!}+\bar{C}\frac{(CL_{\eps}T)^{n}}{n!}\le \eps CT \exp(CKT\log(\eps^{-1}))+\bar{C}\frac{(CKT\log(\eps^{-1}))^{n}}{n!},
\end{align*}
where in the penultimate step we used the uniform bound $\tilde{\rho}^{0}\le \bar C=\bar C(d)$ coming from the boundedness of $\T^{d}$. Choosing $\eps=e^{-n}$ we find
\begin{equation*}
    \sup_{t\in[0,T)}\tilde{\rho}(t) \le CT\exp\left(-n(1-CKT)\right)+\bar{C}\frac{(CKTn)^{n}}{n!}.
\end{equation*}
Then, exploiting the Stirling's inequality $\log(n^{n}/n!) \lesssim n$, we can find $T$ sufficiently small, independent of $n$, such that 
\begin{equation*}
    \sup_{t\in[0,T)}\tilde{\rho}(t) \le (C+\bar{C})e^{-n/2}\qquad \forall n\ge 0,
\end{equation*}
which concludes the proof of this step. 

\smallskip
\noindent \textbf{Step 3:} In this final step we show that $\{\mu^{n}\}_{n\ge 0}$ converges to a solution $\mu$ of \eqref{eq:PDE-GF} in $C_{w^{*}}([0,T);\mathscr{P}(\T^{d}))$, up to subsequences. First, thanks to Step 1, \Cref{lem:modulus-continuity-vector-field}, and \Cref{lem:W2-Linfty}, the vector fields $v^{n}_{t}$ are equi-bounded and equi-continuous with respect to the same log-Lipschitz modulus of continuity. Therefore, the flow maps $\{X^{n}\}_{n\ge 0}$ are equi-continuous in $[0,T)\times\T^{d}$, and by Arzelà-Ascoli, there exist a subsequence $n_{k}\to \infty$ and a continuous map $X:[0,T)\times \T^{d}\to \T^{d}$ such that
\begin{equation*}
   \lim_{k\to \infty}\lVert X^{n_{k}}-X\rVert_{C([0,T)\times \T^{d})}=0.
\end{equation*}
Defining $\mu_{t}:= (X_{t})_{\#}\bar\mu$ we deduce that
\begin{equation}\label{eq:weak-convergence-approximate-solutions-existence}
    \mu_{t}^{n_{k}}\stackrel{*}{\rightharpoonup} \mu_{t}\qquad \forall t\in [0,T),\qquad \mu \in C_{w^{*}}([0,T);\mathscr{P}(\T^{d}))\cap L^{\infty}([0,T);\mathscr{X}_{s}(\T^{d})).
\end{equation}
Then, from Step 2 we also infer
\begin{equation*}
    \mu_{t}^{n_{k}-1}\stackrel{*}{\rightharpoonup} \mu_{t}\qquad \forall t\in [0,T).
\end{equation*}
Finally, setting $v_{t}:=-\nabla K_{s}*(\mu_{t}-\nu)$, \Cref{lem:W2-Linfty} implies that 
\begin{equation}\label{eq:strong-convergence-vector-fields-existence}
     \lVert v_{t}^{n_{k}-1}-v_{t}\rVert_{C(\T^{d})}\to 0\qquad \forall t\in [0,T).
\end{equation}
Taking into account \eqref{eq:weak-convergence-approximate-solutions-existence}, \eqref{eq:strong-convergence-vector-fields-existence}, and the uniform boundedness of $v^{n}, v$, we may use the dominated convergence theorem to pass to the limit in both sides of the distributional formulation
\begin{equation*}
    \int_{\T^{d}}\varphi \mu_{t}^{n_{k}}=\int_{\T^{d}}\varphi \bar\mu+\int_{0}^{t}\int_{\T^{d}}\nabla \varphi \cdot v_{r}^{n_{k}-1}\mu_{r}^{n_{k}}dr\qquad \forall t\in [0,T),\quad \forall \varphi \in C^{\infty}(\T^{d}).
\end{equation*}
This shows that $\mu$ is a solution of \eqref{eq:PDE-GF} and concludes the proof.  
\end{proof}

\subsubsection{Maximal solutions and the continuation criterion}\label{subsec:maximal-sol-riesz}
In the following proposition we show that the local solution constructed above can be extended up to some maximal time of existence, and we provide a continuation criterion.
\begin{prop}\label{prop:maximal-solutions}
    Let $s\ge 1$ and $\bar\mu, \nu \in \mathscr{P}\cap\mathscr{X}_{s}(\T^{d})$. Then, there exists a unique maximal solution $u\in L^{\infty}_{\loc}([0,T);\mathscr{X}_{s}(\T^{d}))$ of \eqref{eq:PDE-GF} according to \Cref{def:maximal-solutions}. Moreover, if $s\ge d/2+1$, we have $T=\infty$. If, instead, $s\in \left[1,d/2+1\right)$, then $T$ is finite if and only if $\limsup_{t\to T^{-}}\lVert \mu_{t}\rVert_{L^{p}}=+\infty$, where $p=d/(2s-2)\in (1,\infty]$.
\end{prop}
\begin{proof}
    Consider the family $\mathscr{S}_{s}(\bar\mu, \nu)$ of all solutions of \eqref{eq:PDE-GF} according to \Cref{def:solution-active-scalar-equation} and define $T$ as the supremum of the existence times among all elements in $\mathscr{S}_{s}(\bar\mu, \nu)$. By \Cref{prop:stability-yudovich}, if $\mu^{1}\in L^{\infty}_{\loc}([0,T^{1}))$ and $\mu^{2}\in L^{\infty}_{\loc}([0,T^{2}))$ are two solutions of \eqref{eq:PDE-GF}, then $\mu^{1}_{t}=\mu^{2}_{t}$ for all $t\in [0,T^{1}\wedge T^{2})$. Therefore, one can define the curve $\mu:[0,T)\to \mathscr{P}(\T^{d})$ such that, for every $t\in [0,T)$, $\mu_{t}= \hat{\mu}_{t}$, where $\hat{\mu}\in \mathscr{S}_{s}(\bar\mu, \nu)$ is any solution up to some time $\hat{T}>t$ (which exists by the definition of $T$). Notice that $\mu \in L^{\infty}_{\loc}([0,T);\mathscr{X}_{s}(\T^{d}))$ solves \eqref{eq:PDE-GF}, because it coincides locally in time with elements from $\mathscr{S}_{s}(\bar\mu, \nu)$. Moreover, by construction, $\mu$ is a maximal solution according to \Cref{def:maximal-solutions}. Uniqueness of maximal solutions now follows directly from \Cref{prop:stability-yudovich}. 
    
    To prove the last part of the statement, we first observe the following: if $T<\infty$ and 
    $[0,T)\ni t\mapsto \mu_{t} \in (\mathscr{P}(\T^{d}), W_2)$ is Lipschitz continuous, then there exists a limit measure $\mu_{*}\in \mathscr{P}\cap \mathscr{X}_{s}(\T^d)$ such that $\mu_{t}\stackrel{*}{\rightharpoonup}\mu_{*}$ as $t\to T^{-}$. Applying \Cref{prop:existence-yudovich} to the initial condition $\mu_{*}$ starting at time $T$, we may find a solution $\tilde{\mu}$ in $[T, T+\eps)$, for some $\eps>0$. Gluing $\mu$ and $\tilde{\mu}$ provides a non-trivial extension of $\mu$, contradicting its maximality.  
    On the other hand, by \Cref{prop:gradient-flow-structure} we know that $\mu$ satisfies the energy dissipation identity \eqref{eq:energy-dissipation-identity}. In particular, thanks to \Cref{lem:W2-Linfty}, for every $t\in (0,T)$, the metric derivative $|\mu'_{t}|$ can be bounded by 
    \begin{align*}
        |\mu'_{t}|&= \left(\int_{\T^{d}}|\nabla K_{s}*(\mu_{t}-\nu)|^{2}d\mu_{t}\right)^{1/2}\\
        &\le \lVert \nabla K_{s}*(\mu_{t}-\nu)\rVert_{L^{\infty}}
        \lesssim_{d,s}\begin{cases}
            \lVert \mu\rVert_{L^{p}}+\lVert \nu\rVert_{L^{p}}\quad \text{for \, $p=\frac{d}{2s-2}$}\qquad &\text{if $s\in [1,\frac{d}{2}+1)$},\\
            1\qquad &\text{if $s\ge \frac{d}{2}+1$}.
        \end{cases}
    \end{align*}
    As a consequence, when $s\ge d/2+1$, we always have $\mu \in \Lip([0,T);\left( \mathscr{P}(\T^{d}), W_2\right))$, while for $s\in [1,d/2+1)$ the same holds provided that $\mu \in L^{\infty}([0,T);L^{p})$. This concludes the proof.
\end{proof}

\subsection{Propagation of regularity}\label{subsec:propagation-regularity}
In this section we prove that H\"{o}lder and Sobolev regularity are propagated from the data to solutions of \eqref{eq:PDE-GF} up to their maximal time of existence. The H\"{o}lder case is considered in \Cref{subsec:propagation-holder}. We deduce in particular that smooth data give rise to smooth solutions of \eqref{eq:PDE-GF}. This, combined with a priori estimates from \Cref{lem:energy-estimates}, will allow us to conclude the propagation of Sobolev regularity in \Cref{subsec:propagation-sobolev}.

\subsubsection{Propagation of H\"{o}lder regularity}\label{subsec:propagation-holder}

In the following proposition we prove propagation of H\"{o}lder regularity. Similar strategies can be adopted to prove that any regularity which is \lq\lq better than $\mathscr{X}_{s}(\T^{d})$'' is propagated up to the maximal existence time of $\mathscr{X}_{s}(\T^{d})$-solutions (see \Cref{rmk:propagation-regularity}). 

\begin{prop}\label{prop:propagation-regularity}
    Let $s\ge 1$, $\bar\mu, \nu \in \mathscr{P}\cap \mathscr{X}_{s}(\T^{d})$ and $\mu\in L^{\infty}_{\loc}([0,T); \mathscr{X}_{s}(\T^{d}))$ be the unique maximal solution of \eqref{eq:PDE-GF} given by \Cref{prop:maximal-solutions}. Let $k\in \N\cup \{0\}$ and $\alpha\in (0,1)$. If $\bar\mu, \nu \in C^{k,\alpha}(\T^{d})$, then $\mu \in L^{\infty}_{\loc}([0,T);C^{k,\alpha}(\T^{d}))$. 
\end{prop}
\begin{proof}
    Clearly, it suffices to prove that $\mu \in L^{\infty}([0,\tau);C^{k,\alpha}(\T^{d}))$ for every given $\tau \in (0,T)$. Let $v_{t}:= -\nabla K_{s}*(\mu_{t}-\nu)$ and $X$ be, respectively, the vector field and the flow map associated to the solution $\mu$. We divide the proof in some steps.

    \smallskip
    \noindent \textbf{Step 1:} We first prove that $\mu \in L^{\infty}([0,\tau'); C^{0,\alpha})$ for some small $0<\tau'\le \tau$. We argue as in the proof of \Cref{prop:existence-yudovich} by building a sequence of approximate solutions $\mu^{n}\in L^{\infty}([0,\tau');C^{0,\alpha})$ such that
    $$\mu^{0}_{t}=\bar\mu,\qquad \mu^{n+1}_{t}= (X_{t}^{n})_{\#}\bar \mu\qquad \forall t\in [0,\tau'),\quad \forall n\ge 0.$$
    We claim that there exists $\tau' \in (0,\tau]$ such that 
    \begin{equation}\label{eq:claim-prop-holder-reg}
        \lVert \mu^{n}\rVert_{L^{\infty}([0,\tau');C^{0,\alpha})}\le 2\lVert \bar\mu\rVert_{C^{0,\alpha}}\qquad \forall n\ge 0. 
    \end{equation}
    Once \eqref{eq:claim-prop-holder-reg} is proved, by sending $n\to \infty$ we deduce $\mu \in L^{\infty}([0,\tau'); C^{0,\alpha})$, as desired. 
    Note that \eqref{eq:claim-prop-holder-reg} is trivial when $n=0$. Suppose it holds for some $n\ge 0$ and let us prove it holds for $n+1$. By the continuity of the operator $\nabla (-\Delta)^{-s}:C^{0,\alpha}\to C^{1,\alpha}$ we find 
    \begin{equation*}
        \lVert v^{n}_{t}\rVert_{C^{1,\alpha}}\le C(d,s,\alpha)\lVert \mu_{t}-\nu\rVert_{C^{0,\alpha}}\le C(d,s,\alpha)\left(\lVert \nu\rVert_{C^{0,\alpha}}+2\lVert \bar\mu\rVert_{C^{0,\alpha}}\right)=:L\qquad \forall t\in [0,\tau').
    \end{equation*}
    Therefore, taking norms in the identity $\frac{d}{dt}\nabla X^{n}_{t}=\nabla v^{n}_{t}\nabla X^{n}_{t}$ and integrating in time, we obtain 
    \begin{equation*}
      \max\{ \lVert \nabla X^{n}_{t}\rVert_{C^{0,\alpha}},\lVert \nabla (X^{n}_{t})^{-1}\rVert_{C^{0,\alpha}}\}\le \exp\left(\int_{0}^{t}\lVert \nabla v_{r}\rVert_{C^{0,\alpha}}dr\right)\le \exp\left(L\tau'\right)\qquad \forall t \in [0,\tau').
    \end{equation*}
    In particular, exploiting the representation formula $\mu^{n+1}_{t}=\left(\bar\mu \det \nabla X_{t}^{n}\right)\circ (X_{t}^{n})^{-1}$, we conclude
    \begin{equation*}
        \lVert \mu^{n+1}_{t}\rVert_{C^{0,\alpha}}\le \left(1+\lVert \nabla (X_{t}^{n})^{-1}\rVert_{C^{0,\alpha}}^{\alpha}\right)\lVert \nabla X_{t}^{n}\rVert_{C^{0,\alpha}}^{d}\lVert \bar\mu \rVert_{C^{0,\alpha}}\le \exp\left((\alpha+d)L\tau'\right)\lVert \bar\mu \rVert_{C^{0,\alpha}}\le 2\lVert \bar\mu \rVert_{C^{0,\alpha}},
    \end{equation*}
    provided that we choose $\tau'>0$ sufficiently small. 
    
    \smallskip
    \noindent \textbf{Step 2:} Let $\tau''\in [\tau',\tau]$ be the maximal time for which the solution lives in $C^{0,\alpha}$, i.e.
    \begin{equation*}
        \tau'':= \sup \{t\in [\tau',\tau]: \mu \in L^{\infty}([0,t);C^{0,\alpha})\}.
    \end{equation*}
    In this step, we show that $\tau''=\tau$. Arguing by continuation as in \Cref{prop:maximal-solutions}, we may reduce to prove $\mu \in L^{\infty}([0,\tau'');C^{0,\alpha})$. Let 
    \begin{equation*}
        M:= \lVert \nu\rVert_{\mathscr{X}_{s}}+\sup_{t\in [0,\tau)}\lVert \mu_{t}\rVert_{\mathscr{X}_{s}}.
    \end{equation*}
    Since $\mu \in L_{\loc}^{\infty}([0,\tau'');C^{0,\alpha})$, we have $v\in L^{\infty}_{\loc}([0,\tau'');C^{1,\alpha})$. Let $\tilde{\mu}_{t}:= \mu_{t}\circ X_{t}$. Since $\mu$ solves \eqref{eq:PDE-GF} and the vector field is regular, we find
    \begin{equation}\label{eq:evolution-characteristics-propagation-regularity}
        \frac{d}{dt}\tilde{\mu}_{t}=-\tilde{\mu}_{t}\diver v_{t}\circ X_{t}\qquad \forall t\in (0,\tau'').
    \end{equation}
    We first show that there exists some finite constant $K>0$ such that
    \begin{equation}\label{eq:uniform-bound-characteristics-Linfty-propagation}
       \lVert \mu_{t}\rVert_{L^{\infty}}= \lVert \tilde{\mu}_{t}\rVert_{L^{\infty}}\le K\qquad \forall t\in [0,\tau'').
    \end{equation}
    When $s=1$ we already know by assumption that $\lVert\mu_{t}\rVert_{L^{\infty}}$ (thus $\lVert\tilde{\mu}_{t}\rVert_{L^{\infty}}$) is uniformly bounded in $[0,\tau)\supseteq[0,\tau'')$. When $s\in \left(1,d/2+1\right]$, by \eqref{eq:evolution-characteristics-propagation-regularity} and \Cref{lem:log-interp-BKM-criterion}, setting $p=d/(2s-2)\in [1,\infty)$ we obtain
    \begin{align*}
        \frac{d}{dt}\lVert \tilde{\mu}_{t}\rVert_{L^{\infty}}&\le \lVert \tilde{\mu}_{t}\rVert_{L^{\infty}}\lVert \nabla^{2}K_{s}*(\mu_{t}-\nu)\rVert_{L^{\infty}}\\
        &\lesssim_{d,s} \lVert \tilde{\mu}_{t}\rVert_{L^{\infty}}\lVert \mu_{t}-\nu\rVert_{L^{p}}\log\left(1+\frac{\lVert \mu_{t}-\nu\rVert_{L^{\infty}}}{\lVert \mu_{t}-\nu\rVert_{L^{p}}}\right)\lesssim_{d,s,M} \lVert \tilde{\mu}_{t}\rVert_{L^{\infty}}\log\left(2+\lVert \nu\rVert_{L^{\infty}}+\lVert \tilde{\mu}\rVert_{L^{\infty}}\right).
    \end{align*}
    Hence $\lVert \tilde{\mu}\rVert_{L^{\infty}}$ grows at most as a double-exponential in time, and is uniformly bounded on $[0,\tau'')$. Finally, the same holds when $s>d/2+1$ because of the uniform bound on $\nabla ^{2}K_{s}*(\mu_{t}-\nu)$ from \Cref{lem:modulus-continuity-vector-field} and Gr\"onwall's inequality applied to \eqref{eq:evolution-characteristics-propagation-regularity}. 

    To prove $\mu \in L^{\infty}([0,\tau'');C^{0,\alpha})$, we distinguish the two cases $s>1$, and $s=1$. If $s>1$, since $\nabla v_{t}=\nabla^{2}K_{s}*(\mu_{t}-\nu)$, Schauder estimates give 
    \begin{equation}\label{eq:calderon-zygmund-propagation-holder}
        \begin{gathered}
        \lVert \nabla v_{t}\rVert_{L^{\infty}}\lesssim_{d,s} \lVert \mu_{t}\rVert_{L^{\infty}}+\lVert \nu\rVert_{L^{\infty}},\\
        \lVert \nabla v_{t}\rVert_{C^{0,\alpha}}\lesssim_{d,s,\alpha} \lVert \mu_{t}\rVert_{C^{0,\alpha}}+\lVert \nu\rVert_{C^{0,\alpha}}.
    \end{gathered}\qquad \forall t\in [0,\tau'').
    \end{equation}
    Therefore, in view of \eqref{eq:uniform-bound-characteristics-Linfty-propagation} we find 
    \begin{equation}\label{eq:Lip-bound-vector-field-propagation-holder-s>1}
        \lVert \nabla v_{t}\rVert_{L^{\infty}},\, \Lip(X_{t}),\, \Lip (X_{t}^{-1})\le \tilde{K}\qquad \lVert \mu_{t}\rVert_{C^{0,\alpha}}\le \lVert \tilde{\mu}_{t}\rVert_{C^{0,\alpha}}(1+\tilde{K}^{\alpha})\qquad  \forall t\in [0,\tau''),
    \end{equation}
    for some finite constant $\tilde{K}>0$.
    By \eqref{eq:evolution-characteristics-propagation-regularity}, using \eqref{eq:uniform-bound-characteristics-Linfty-propagation}, \eqref{eq:calderon-zygmund-propagation-holder} and \eqref{eq:Lip-bound-vector-field-propagation-holder-s>1}, for every $t\in (0,\tau'')$ we find
    \begin{align*}
        \frac{d}{dt}\lVert \tilde{\mu}_{t}\rVert_{C^{0,\alpha}}&\le \lVert \tilde\mu_{t}\nabla v_{t}\circ X_{t}\rVert_{C^{0,\alpha}}\\
        &\le \lVert \tilde{\mu}_{t}\rVert_{C^{0,\alpha}}\lVert\nabla v_{t}\rVert_{L^{\infty}}+\lVert \tilde{\mu}_{t}\rVert_{L^{\infty}}\left(\lVert \nabla v_{t}\rVert_{L^{\infty}}+\lVert \nabla v_{t}\rVert_{C^{0,\alpha}}\Lip(X_{t})^{\alpha}\right)\\
        & \lesssim_{d,s,\alpha,K,\tilde K}  \lVert \tilde{\mu}_{t}\rVert_{C^{0,\alpha}}+1.
    \end{align*}
    Hence, using Gr\"onwall's inequality, we deduce $\tilde{\mu}\in L^{\infty}([0,\tau'');C^{0,\alpha})$, and by \eqref{eq:Lip-bound-vector-field-propagation-holder-s>1}, $\mu\in L^{\infty}([0,\tau'');C^{0,\alpha})$. 

    In the case $s=1$ we have $-\diver v_{t}=\mu_{t}-\nu$, therefore, by \eqref{eq:evolution-characteristics-propagation-regularity}, for all $t\in (0,\tau'')$,
    \begin{equation*}
        \frac{d}{dt}\lVert \tilde \mu_{t}\rVert_{C^{0,\alpha}}\le \lVert \tilde{\mu}_{t}(\tilde{\mu}_{t}-\nu\circ X_{t})\rVert_{C^{0,\alpha}}\lesssim_{M}\lVert \tilde{\mu}_{t}\rVert_{C^{0,\alpha}}+ \lVert \nu\rVert_{C^{0,\alpha}}\Lip(X_{t})^{\alpha}.
    \end{equation*}
    By Gr\"onwall's inequality this yields, for all $t\in [0,\tau'')$,
    \begin{equation}\label{eq:Calpha-bound-propagation-holder-s=1}
       \lVert \mu_{t}\rVert_{C^{0,\alpha}}\le  \Lip(X_{t}^{-1})^{\alpha}\lVert \tilde{\mu}_{t}\rVert_{C^{0,\alpha}}\lesssim_{M,\tau}\Lip(X_{t}^{-1})^{\alpha}\left(1+\lVert \nu\rVert_{C^{0,\alpha}}\int_{0}^{t}\Lip(X_{r})^{\alpha}dr\right).
    \end{equation}
    On the other hand, we know that
    \begin{equation}\label{eq:Lip-bound-flow-propagation-holder-s=1}
        \Lip(X_{t})^{\alpha}, \Lip(X_{t}^{-1})^{\alpha}\le \exp \left(\alpha\int_{0}^{t}\lVert \nabla v_{r}\rVert_{L^{\infty}}dr\right).
    \end{equation}
    In particular, applying \Cref{lem:log-interp-BKM-criterion} along with \eqref{eq:Calpha-bound-propagation-holder-s=1} and \eqref{eq:Lip-bound-flow-propagation-holder-s=1}, for all $t\in [0,\tau'')$ we find
    \begin{align*}
        \lVert \nabla v_{t}\rVert_{L^{\infty}}&\lesssim_{d,\alpha} \lVert \mu_{t}-\nu\rVert_{L^{\infty}}\left(1+\log\left(\frac{\lVert \mu_{t}-\nu\rVert_{C^{0,\alpha}}}{\lVert \mu_{t}-\nu\rVert_{L^{\infty}}}\right)\right)\\
        &\lesssim_{M}1+\log\left(2+\lVert \nu\rVert_{C^{0,\alpha}} +\lVert \mu_{t}\rVert_{C^{0,\alpha}}\right)\lesssim_{M, \alpha, \tau, \lVert \nu\rVert_{C^{0,\alpha}}} 1+ \int_{0}^{t}\lVert \nabla v_{r}\rVert_{L^{\infty}}dr.
    \end{align*} 
    By one more application of Gr\"onwall's inequality this gives $v \in L^{\infty}([0,\tau'');C^{0,1})$, from which we eventually deduce that $\mu \in L^{\infty}([0,\tau'');C^{0,\alpha})$ thanks to \eqref{eq:Calpha-bound-propagation-holder-s=1} and \eqref{eq:Lip-bound-flow-propagation-holder-s=1}.

    \smallskip
    \noindent \textbf{Step 3:} Finally, by a bootstrap argument we show that $\mu \in L^{\infty}([0,\tau);C^{k,\alpha})$.
    Suppose we know $\mu \in L^{\infty}([0,\tau);C^{k',\alpha})$ for some $0\le k'<k$. Then, by the continuity of $\nabla (-\Delta)^{-s}:C^{k',\alpha}\to C^{k'+1,\alpha}$ we deduce that $v \in L^{\infty}([0,\tau);C^{k'+1,\alpha})$ and consequently $X, X^{-1} \in L^{\infty}([0,\tau);C^{k'+1,\alpha})$. The function $\tilde{\mu}_{t}=\mu_{t}\circ X_{t}$ solves $\partial_{t}\tilde\mu_{t}=-\tilde{\mu}_{t}[(-\Delta)^{1-s}(\mu_{t}-\nu)]\circ X_{t}$. Therefore, by composition and product rules in H\"{o}lder spaces along with uniform bounds of $\tilde{\mu}$ in $C^{k',\alpha}$ and of $X_{t}$ in $C^{k'+1,\alpha}$, we get
    \begin{equation*}
        \frac{d}{dt}\lVert \tilde{\mu}_{t}\rVert_{C^{k'+1,\alpha}}\le \lVert \tilde{\mu}_{t}[(-\Delta)^{1-s}(\mu_{t}-\nu)]\circ X_{t}\rVert_{C^{k'+1,\alpha}}\lesssim \lVert \tilde{\mu}_{t}\rVert_{C^{k'+1,\alpha}}+\lVert \nu\rVert_{C^{k'+1,\alpha}}.
    \end{equation*}
    From Gronwall's inequality we deduce $\tilde{\mu}\in L^{\infty}([0,\tau);C^{k'+1,\alpha})$, and composing with the inverse flow that $\mu \in L^{\infty}([0,\tau);C^{k'+1,\alpha})$. This concludes the bootstrap, and with it the proof.  
\end{proof}

\begin{rmk}[Propagation of general regularities]\label{rmk:propagation-regularity}
The main point behind the proof of \Cref{prop:propagation-regularity} is that $\mu_{t}$ generates a Lipschitz vector field $v_{t}$ as soon as it is slightly more regular than $\mathscr{X}_{s}(\T^{d})$. This is the key condition required to propagate (H\"{o}lder) regularity thanks to Gr\"onwall's inequality. 
The same arguments apply to propagation of other types of regularity. For instance, one could use a similar strategy to prove that if $s\in \left(1,d/2+1\right]$, $p=d/(2s-2)$, $r>p$, and $\bar\mu, \nu \in L^{r}(\T^{d})$, then $\mu \in L^{\infty}_{\loc}([0,T);L^{r})$, where $T$ is the maximal existence time of $\mathscr{X}_{s}$-solutions. The same would hold for all $r\ge 1$ when $s>d/2+1$. Moreover, with the same technique we can propagate $C^{k,1}$-regularity, for any $k\in \N\cup \{0\}$, provided that $s>1$.  \fr
\end{rmk}
\begin{rmk}[Regularity in time]\label{rmk:regularity-in-time} For a maximal solution $\mu\in L^{\infty}([0,T);\mathscr{X}_{s}(\T^{d}))$ of \eqref{eq:PDE-GF}, the space regularity (uniform in time) provided by \Cref{prop:propagation-regularity} can be used to prove regularity in time. For example, denoting by $v$ and $X$ the velocity field and the flow map generated by $\mu$, respectively, the following holds: 
\begin{equation*}
    \bar\mu, \nu \in C^{\infty}\implies \mu \in C^{\infty}([0,T)\times \T^{d}),\, v\in C^{\infty}([0,T)\times \T^{d};\R^{d}),\, X\in C^{\infty}([0,T)\times \T^{d};\T^{d}).
\end{equation*}
In fact, in this case \Cref{prop:propagation-regularity} gives $\mu \in L^{\infty}_{\loc}([0,T);C^{k})$ for all $k\in \N$, where $T>0$ is the maximal existence time of the solution $\mu$. This in turn implies that $v\in L^{\infty}_{\loc}([0,T);C^{k})$ for all $k\in \N$. From here, differentiating $\frac{d}{dt}X=v\circ X$ in time and space, and arguing by induction on the order of derivation, one deduces that $X,X^{-1}\in C^{\infty}([0,T)\times \T^{d};\T^{d})$, and finally, from the representation formula $\mu_{t}=\left(\bar\mu \det \nabla X_{t}\right)\circ X_{t}^{-1}$, that $\mu \in C^{\infty}([0,T)\times \T^{d})$ and $v\in C^{\infty}([0,T)\times \T^{d};\R^{d})$. \fr
\end{rmk}

\subsubsection{Propagation of Sobolev regularity}\label{subsec:propagation-sobolev}

 In the next lemma, we derive energy estimates in homogeneous Sobolev spaces for solutions of \eqref{eq:PDE-GF}. As an application, we establish the propagation of Sobolev regularity. These estimates will play a fundamental role in the proof of our smooth convergence results in \Cref{thm:convergence-s=1} and \Cref{thm:convergence-s>1}.

 \begin{lem}\label{lem:energy-estimates}
     Let $s\ge 1$, $\gamma>d/2$, and $\bar\mu, \nu \in \mathscr{P}\cap C^{\infty}(\T^{d})$. Let $\mu\in C^{\infty}([0,T)\times \T^{d})$ be the maximal solution of \eqref{eq:PDE-GF} given by \Cref{prop:maximal-solutions} and \Cref{rmk:regularity-in-time}. Then, the following energy estimates hold for some $C=C(d,s,\gamma)>0$ and all $t\in (0,T)$:
     \begin{align}
        \frac{d}{dt}\lVert \mu_{t}\rVert_{\dot H^{\gamma}}^{2}&\le C\left(\lVert \nabla^{2}(-\Delta)^{-s}(\mu_{t}-\nu)\rVert_{L^{\infty}}+\lVert \nu\rVert_{H^{\gamma}}\right)\left(\lVert \mu_{t}\rVert_{\dot H^{\gamma}}^{2}+\lVert \nu\rVert_{\dot H^{\gamma}}^{2}\right),\label{eq:energy-estimate-propagation}\\[5pt]
        \frac{d}{dt}\lVert \mu_{t}-\nu\rVert_{\dot H^{\gamma}}^{2}&\le C\lVert \nabla^{2}(-\Delta)^{-s}(\mu_{t}-\nu)\rVert_{L^{\infty}}\lVert \mu_{t}-\nu\rVert_{\dot H^{\gamma}}^{2}\label{eq:energy-estimate}\\
        &\quad\,-2\big(\min_{\T^{d}}\nu\big)\lVert \mu_{t}-\nu\rVert_{\dot H^{\gamma-s+1}}^{2}+ C\lVert \nu\rVert_{\dot H^{\gamma+s}}\lVert \mu_{t}-\nu\rVert_{\dot H^{\gamma-s}}\lVert \mu_{t}-\nu\rVert_{\dot H^{\gamma-s+1}}.\notag
    \end{align}
 \end{lem}
 \begin{proof}
    In the following, we use the notation $D^{\beta}:=(-\Delta)^{\beta/2}$ for all $\beta\in \R$, and we consider the zero-mean perturbation $\sigma_{t}:= \mu_{t}-\nu$,
     which solves the equation
    \begin{equation*}
        \partial_{t}\sigma_{t}=\diver(\nu \nabla (-\Delta)^{-s}\sigma_{t})+\diver(\sigma_{t}\nabla (-\Delta)^{-s}\sigma_{t})\qquad \forall t\in (0,T).
    \end{equation*}

    \smallskip
    \noindent \textbf{Step 1:} We first prove \eqref{eq:energy-estimate-propagation}. Using equation \eqref{eq:PDE-GF}, integrating by parts, and distributing fractional derivatives, we get three terms: 
    \begin{align*}
        \frac{d}{dt}\frac{\lVert \mu_{t}\rVert_{\dot H^{\gamma}}^{2}}{2}&=\int_{\T^{d}}D^{\gamma}\diver \left(\mu_{t}\nabla D^{-2s}\sigma_{t}\right)D^{\gamma}\mu_{t}\\
        &=-\int_{T^{d}}D^{\gamma}\mu_{t}\nabla D^{-2s}\sigma_{t} \cdot \nabla D^{\gamma}\mu_{t}\\
        &\quad\,+ \int_{\T^{d}}\big[D(D^{\gamma}\mu_{t} \nabla D^{-2s}\sigma_{t})-D^{\gamma+1}\mu_{t}\nabla D^{-2s}\sigma_{t}\big]\cdot \nabla D^{\gamma-1}\mu_{t}\\
        &\quad\, -\int_{\T^{d}}\big[D^{\gamma+1}(\mu_{t}\nabla D^{-2s}\sigma_{t})-D^{\gamma+1}\mu_{t}\nabla D^{-2s}\sigma_{t}\big]\cdot \nabla D^{\gamma-1}\mu_{t}=:I+II+III.
    \end{align*}
    For $I$, we use the identity $2D^{\gamma}\mu_{t}\nabla D^{\gamma}\mu_{t}=\nabla (D^{\gamma}\mu_{t})^{2}$, and integration by parts:
    \begin{equation}\label{eq:ee-riesz-propagation-I}
        I= \frac{1}{2}\int_{\T^{d}}\Delta D^{-2s}\sigma_{t}(D^{\gamma}\mu_{t})^{2} \le C \lVert \nabla^{2}D^{-2s}\sigma_{t}\rVert_{L^{\infty}}\lVert \mu_{t}\rVert_{\dot H^{\gamma}}^{2}.
    \end{equation}
    For $II$, we use Cauchy--Schwarz inequality and \eqref{eq:kato-ponce-gamma1}:
    \begin{equation}\label{eq:ee-riesz-propagation-II}
    \begin{aligned}
        II&\le \lVert D(D^{\gamma}\mu_{t} \nabla D^{-2s}\sigma_{t})-D^{\gamma+1}\mu_{t}\nabla D^{-2s}\sigma_{t} \rVert_{L^{2}}\lVert \nabla D^{\gamma-1}\mu_{t}\rVert_{L^{2}}\\
        &\lesssim_{d,\gamma}\lVert \nabla^{2}D^{-2s}\sigma_{t}\rVert_{L^{\infty}}\lVert D^{\gamma}\mu_{t}\rVert_{L^{2}}\lVert \nabla D^{\gamma-1}\mu_{t}\rVert_{L^{2}}\lesssim_{d,\gamma} \lVert \nabla^{2}D^{-2s}\sigma_{t}\rVert_{L^{\infty}}\lVert \mu_{t}\rVert_{\dot H^{\gamma}}^{2}.
    \end{aligned}
    \end{equation}
    We then write $\mu_{t}=\sigma_{t}+\nu$, and further divide $III$ into two terms: 
    \begin{align*}
        III&= -\int_{\T^{d}}\big[D^{\gamma+1}(\sigma_{t}\nabla D^{-2s}\sigma_{t})-D^{\gamma+1}\sigma_{t}\nabla D^{-2s}\sigma_{t}\big]\cdot \nabla D^{\gamma-1}\mu_{t} \\
        &\quad\, -\int_{\T^{d}}\big[D^{\gamma+1}(\nu\nabla D^{-2s}\sigma_{t})-D^{\gamma+1}\nu\nabla D^{-2s}\sigma_{t}\big]\cdot \nabla D^{\gamma-1}\mu_{t}=:III_{1}+III_{2}.
    \end{align*}
    To bound $III_{1}$, we use Cauchy--Schwarz inequality, and then apply \eqref{eq:kato-ponce-non-linear-term} to the function $\phi_{t}=D^{-2s}\sigma_{t}$:
    \begin{equation}\label{eq:ee-riesz-propagation-III1}
        \begin{aligned}
        III_{1}&\le \lVert D^{\gamma+1}(D^{2s}\phi_{t}\nabla \phi_{t})-D^{\gamma+1+2s}\phi_{t}\nabla \phi_{t}\rVert_{L^{2}}\lVert \nabla D^{\gamma-1}\mu_{t}\rVert_{L^{2}}\\
        &\lesssim_{d,s,\gamma}\lVert \nabla^{2}\phi_{t}\rVert_{L^{\infty}}\lVert D^{\gamma+2s}\phi_{t}\rVert_{L^{2}}\lVert \mu_{t}\rVert_{\dot H^{\gamma}}=\lVert \nabla^{2}D^{-2s}\sigma_{t}\rVert_{L^{\infty}}\lVert \sigma_{t}\rVert_{\dot H^{\gamma}}\lVert \mu_{t}\rVert_{\dot H^{\gamma}}.
    \end{aligned}
    \end{equation}
    To bound $III_{2}$, we use Cauchy--Schwarz inequality, and \eqref{eq:kato-ponce} along with the Sobolev embedding $H^{\gamma}\hookrightarrow L^{\infty}$:
    \begin{equation}\label{eq:ee-riesz-propagation-III2}
    \begin{aligned}
        III_{2}&\le \lVert D^{\gamma+1}(\nu\nabla D^{-2s}\sigma_{t})-D^{\gamma+1}\nu\nabla D^{-2s}\sigma_{t}\rVert_{L^{2}}\lVert \nabla D^{\gamma-1}\mu_{t}\rVert_{L^{2}}\\
        &\lesssim_{d,s,\gamma} \left(\lVert D^{\gamma}\nu\rVert_{L^{2}}\lVert \nabla^{2}D^{-2s}\sigma_{t}\rVert_{L^{\infty}}+\lVert \nu\rVert_{L^{\infty}}\lVert \nabla D^{\gamma+1-2s}\sigma_{t}\rVert_{L^{2}}\right)\lVert \mu_{t}\rVert_{\dot H^{\gamma}}\\
        &\lesssim_{d}\lVert \nu\rVert_{\dot H^{\gamma}}\lVert \nabla^{2}D^{-2s}\sigma_{t}\rVert_{L^{\infty}}\lVert \mu_{t}\rVert_{\dot H^{\gamma}}+\lVert \nu\rVert_{H^{\gamma}}\lVert \sigma_{t}\rVert_{\dot H^{\gamma}}\lVert \mu_{t}\rVert_{\dot H^{\gamma}}.
    \end{aligned}
    \end{equation}
    Gathering \eqref{eq:ee-riesz-propagation-I}, \eqref{eq:ee-riesz-propagation-II}, \eqref{eq:ee-riesz-propagation-III1}, and \eqref{eq:ee-riesz-propagation-III2}, we obtain \eqref{eq:energy-estimate-propagation}.
    
    \smallskip
    \noindent \textbf{Step 2:} Next, we prove \eqref{eq:energy-estimate}.
     Differentiating $\lVert \sigma_{t}\rVert_{\dot H^{\gamma}}^{2}$ in time we get the sum of two terms:
    \begin{equation*}
        \frac{d}{dt}\frac{\lVert \sigma_{t}\rVert_{\dot H^{\gamma}}^{2}}{2}=\underbrace{\int_{\T^{d}}D^{\gamma}\diver(\nu \nabla D^{-2s}\sigma_{t})D^{\gamma}\sigma_{t}}_{({\rm Lin})}+\underbrace{\int_{\T^{d}}D^{\gamma}\diver(\sigma_{t} \nabla D^{-2s}\sigma_{t})D^{\gamma}\sigma_{t}}_{({\rm NonLin})}.
    \end{equation*}
    Let us first treat the term $({\rm Lin})$, arising from the linearized equation (see \Cref{subsec:idea-proof}). Integrating by parts, and then distributing fractional derivatives between the two factors in the integrand we get
    \begin{align*}
        ({\rm Lin})&=-\int_{\T^{d}}D^{\gamma}\left(\nu\nabla D^{-2s}\sigma_{t}\right) \cdot \nabla D^{\gamma}\sigma_{t}\\
        &= -\int_{\T^{d}}D^{\gamma+s}\left(\nu\nabla D^{-2s}\sigma_{t}\right)\cdot \nabla D^{\gamma-s}\sigma_{t}\\
        &= \underbrace{-\int_{\T^{d}}\nu |\nabla D^{\gamma-s}\sigma_{t}|^{2}}_{({\rm Lin})_{{\rm main}}}+\underbrace{(-1)\int_{\T^{d}}\Big[D^{\gamma+s}\left(\nu\nabla D^{-2s}\sigma_{t}\right)-\nu D^{\gamma+s}\nabla D^{-2s}\sigma_{t}\Big]\cdot \nabla D^{\gamma-s}\sigma_{t}}_{({\rm Lin})_{{\rm err}}}.
    \end{align*}
    Now, for the main term $({\rm Lin})_{{\rm main}}$, we simply bound $\nu$ below with its minimum value, and get
    \begin{equation}\label{eq:bound-Lin-Main-eenergy-estimate}
        ({\rm Lin})_{{\rm main}}\le -\big(\min_{\T^{d}}\nu\big)\int_{\T^{d}}|\nabla D^{\gamma-s}\sigma_{t}|^{2}=-\big(\min_{\T^{d}}\nu\big)\lVert \sigma_{t}\rVert_{\dot H^{\gamma-s+1}}^{2}.
    \end{equation}
    For the error term $({\rm Lin})_{{\rm err}}$, using the Cauchy--Schwarz inequality and the Kato--Ponce commutator estimate \eqref{eq:kato-ponce}, we derive
    \begin{equation}\label{eq:bound-Lin-err-eenergy-estimate}
        \begin{aligned}
            ({\rm Lin})_{{\rm err}}&\le \left\lVert D^{\gamma+s}\left(\nu\nabla D^{-2s}\sigma_{t}\right)-\nu D^{\gamma+s}\nabla D^{-2s}\sigma_{t}\right\rVert_{L^{2}} \lVert \nabla D^{\gamma-s}\sigma_{t}\rVert_{L^{2}}\\
        &\lesssim_{d,s}\left(\lVert \nabla \nu\rVert_{L^{\infty}}\lVert D^{\gamma+s-1}\nabla D^{-2s}\sigma_{t}\rVert_{L^{2}}+\lVert D^{\gamma+s}\nu\rVert_{L^{2}}\lVert \nabla D^{-2s}\sigma_{t}\rVert_{L^{\infty}}\right)\lVert \sigma_{t}\rVert_{\dot H^{\gamma-s+1}}\\
        &\lesssim_{d,s,\gamma}\lVert \nu\rVert_{\dot H^{\gamma+s}}\lVert \sigma_{t}\rVert_{\dot H^{\gamma-s}}\lVert \sigma_{t}\rVert_{\dot H^{\gamma-s+1}},
        \end{aligned}
    \end{equation}
    where in the last step we also used the embedding $\dot H^{\gamma+s-1}\hookrightarrow C^{0}$.

    The term $({\rm NonLin})$ can be treated similarly to Step 1. After splitting
    \begin{align*}
        ({\rm NonLin})&=-\int_{T^{d}}D^{\gamma}\sigma_{t}\nabla D^{-2s}\sigma_{t} \cdot \nabla D^{\gamma}\sigma_{t}\\
        &\quad\,+ \int_{\T^{d}}\big[D(D^{\gamma}\sigma_{t} \nabla D^{-2s}\sigma_{t})-D^{\gamma+1}\sigma_{t}\nabla D^{-2s}\sigma_{t}\big]\cdot \nabla D^{\gamma-1}\sigma_{t}\\
        &\quad\, -\int_{\T^{d}}\big[D^{\gamma+1}(\sigma_{t}\nabla D^{-2s}\sigma_{t})-D^{\gamma+1}\sigma_{t}\nabla D^{-2s}\sigma_{t}\big]\cdot \nabla D^{\gamma-1}\sigma_{t}\\
        &=:({\rm NonLin})_{1}+({\rm NonLin})_{2}+({\rm NonLin})_{3},
    \end{align*}
    we may bound $({\rm NonLin})_{1}$, $({\rm NonLin})_{2}$, and $({\rm NonLin})_{3}$ as $I, II$, and $III_{1}$ from Step 1, respectively, and obtain
    \begin{equation}\label{eq:bound-NonLin-energy-estimate}
        ({\rm NonLin})\lesssim_{d,s,\gamma}\lVert \nabla^{2}D^{-2s}\sigma_{t}\rVert_{L^{\infty}}\lVert \sigma_{t}\rVert_{\dot H^{\gamma}}^{2}.
    \end{equation}
    Gathering \eqref{eq:bound-Lin-Main-eenergy-estimate}, \eqref{eq:bound-Lin-err-eenergy-estimate}, and \eqref{eq:bound-NonLin-energy-estimate}, we get \eqref{eq:energy-estimate}.
 \end{proof}
 \begin{prop}\label{prop:propagation-sobolev}
     Let $s\ge 1$, $\gamma>d/2$, $\bar\mu, \nu\in \mathscr{P}\cap H^{\gamma}(\T^{d})$, and let $\mu\in L^{\infty}_{\loc}([0,T);\mathscr{X}_{s}(\T^{d}))$ be the unique maximal solution of \eqref{eq:PDE-GF} given by \Cref{prop:maximal-solutions}. Then, $\mu\in L^{\infty}_{\loc}([0,T);H^{\gamma}(\T^{d}))$.
 \end{prop}
 \begin{proof}
 We proceed by approximation with smooth solutions, using the a priori energy estimate \eqref{eq:energy-estimate-propagation} from \Cref{lem:energy-estimates}.
     Let $\bar\mu^{n}, \nu^{n}\in \mathscr{P}\cap C^{\infty}(\T^{d})$ be such that 
    \begin{equation}\label{eq:convergence-initial-and-target-sobolev}
        \lVert \bar\mu^{n}-\bar\mu\rVert_{H^{\gamma}}\rightarrow 0,\quad \lVert \nu^{n}-\nu\rVert_{H^{\gamma}}\rightarrow 0\qquad \text{as $n\to \infty$},
    \end{equation}
    and let $\mu^{n}$ be the maximal solutions of \eqref{eq:PDE-GF} with initial and target measures $\bar\mu^{n}$ and $\nu^{n}$, respectively, and maximal time of existence $T^{n}$. We divide the proof into three steps. 

    \smallskip
    \noindent \textbf{Step 1:} In this step we derive a uniform bound from below for $T^{n}$, and uniform Sobolev bounds for $\mu^{n}_{t}$.
    By the continuity of $\nabla^{2}(-\Delta)^{-s}:\dot H^{\gamma}\to C^{0}$, from \eqref{eq:energy-estimate-propagation} we deduce  
    \begin{equation}\label{eq:simplified-energy-estimate-differential-form}
    \begin{aligned}
        \frac{d}{dt}\left(\lVert \mu_{t}^{n}\rVert_{\dot H^{\gamma}}^{2}+\lVert \nu^{n}\rVert_{H^{\gamma}}^{2}\right)&\le C\left(\lVert \mu_{t}^{n}\rVert_{\dot H^{\gamma}}^{2}+\lVert \nu^{n}\rVert_{H^{\gamma}}^{2}\right)^{\frac{3}{2}}
    \end{aligned}\qquad \forall t\in (0,T^{n}).
    \end{equation}
    Integrating this differential inequality, and taking into account the continuation criterion from \Cref{prop:maximal-solutions}, we deduce the following bounds:
    \begin{gather}
        T^{n}\ge T_{\min}^{n}:= C^{-1}\left(\lVert \bar\mu^{n}\rVert_{\dot H^{\gamma}}^{2}+\lVert \nu^{n}\rVert_{H^{\gamma}}^{2}\right)^{-1/2},\label{eq:minimal-time-apriori}\\ 
        \lVert \mu^{n}_{t}\rVert_{\dot H^{\gamma}}\le \frac{\left(\lVert \bar\mu^{n}\rVert_{\dot H^{\gamma}}^{2}+\lVert \nu^{n}\rVert_{H^{\gamma}}^{2}\right)^{1/2}}{1-C\left(\lVert \bar\mu^{n}\rVert_{\dot H^{\gamma}}^{2}+\lVert \nu^{n}\rVert_{H^{\gamma}}^{2}\right)^{1/2}t}\qquad \forall t\in [0,T_{\min}^{n}).\label{eq:uniform-bound-Hgamma-apriori}
    \end{gather}

    \smallskip
    \noindent \textbf{Step 2:} Next, we prove that $\mu \in L^{\infty}([0,\tau);H^{\gamma}(\T^{d}))$ for some small $\tau\in (0,T)$. 
    By \eqref{eq:convergence-initial-and-target-sobolev} and \eqref{eq:minimal-time-apriori} we get 
    \begin{equation*}
        \liminf_{n\to \infty} T^{n}\ge \liminf_{n\to \infty}T^{n}_{\min}=C^{-1}\left(\lVert \bar\mu\rVert_{\dot H^{\gamma}}^{2}+\lVert \nu\rVert_{H^{\gamma}}^{2}\right)^{-1/2}=: T_{\min}>0.
    \end{equation*}
    Defining $\tau:= (T_{\min}/2)\wedge T$, the stability result from \Cref{prop:stability-yudovich} ensures the weak-$*$ convergence of measures $\mu_{t}^{n}\stackrel{*}{\rightharpoonup}\mu_{t}$ as $n\to \infty$ for all $t\in [0,\tau)$.
   Therefore, since by \eqref{eq:uniform-bound-Hgamma-apriori} the approximating solutions $\mu^{n}$ are equi-bounded in $L^{\infty}\left([0,\tau); H^{\gamma}(\T^{d})\right)$, we get $\mu\in L^{\infty}([0,\tau);H^{\gamma}(\T^{d}))$ and 
    \begin{equation}\label{eq:convergence-approximation-propagation-Sobolev}
        \mu^{n}_{t} \rightharpoonup \mu_{t}\quad \text{in $H^{\gamma}(\T^{d})$},\qquad \mu^{n}_{t}-\nu^{n}\to \mu_{t}-\nu \quad \text{in $\dot H^{\gamma'}(\T^{d})$\quad $\forall \gamma'<\gamma$}\qquad \forall t\in [0,\tau).
    \end{equation}
    By \eqref{eq:convergence-approximation-propagation-Sobolev} and the continuity of $\nabla^{2}(-\Delta)^{-s}:\dot H^{\gamma'}\to C^{0}$, for $\gamma' \in \left(\frac{d}{2},\gamma\right)$, we may pass to the limit in the integral version of \eqref{eq:energy-estimate-propagation} and get, for all $t\in [0,\tau)$,
    \begin{equation}\label{eq:simplified-energy-estimate-integral-form}
        \left(\lVert \mu_{t}\rVert_{\dot H^{\gamma}}^{2}+\lVert \nu\rVert_{H^{\gamma}}^{2}\right)^{1/2}\le \left(\lVert \bar\mu\rVert_{\dot H^{\gamma}}^{2}+\lVert \nu\rVert_{H^{\gamma}}^{2}\right)^{1/2}\exp \left(C\lVert \nu\rVert_{ H^{\gamma}}t+\int_{0}^{t}C\lVert \nabla^{2}(-\Delta)^{-s}(\mu_{r}-\nu)\rVert_{L^{\infty}}dr\right).
    \end{equation}

\smallskip
\noindent \textbf{Step 3:} Let $\tau'\in [\tau,T]$ be the largest time for which $\mu \in L^{\infty}_{\loc}([0,\tau');H^{\gamma}(\T^{d}))$ and \eqref{eq:simplified-energy-estimate-integral-form} holds in $[0,\tau')$. In this final step we show that $\tau'=T$, thus concluding the proof. 
    
    We may assume that $\tau'<\infty$, otherwise there is nothing to prove. We must necessarily have 
    \begin{equation}\label{eq:Hgamma-must-explode-maximal-time}
        \limsup_{t\to (\tau')^{-}}\,\lVert \mu_{t}\rVert_{H^{\gamma}}=\infty.
    \end{equation}
    In fact, if we had $\mu \in L^{\infty}([0,\tau');H^{\gamma})$, arguing as in the proof of \Cref{prop:maximal-solutions}, we would find that $\mu_{t}$ has a weak limit in $H^{\gamma}$ as $t\to (\tau')^{-}$, and repeating Step 2 starting from time $\tau'$, we could continue the solution in $H^{\gamma}$ and get \eqref{eq:simplified-energy-estimate-integral-form} past the maximal time $\tau'$.
    
    Suppose by contradiction that $\tau'<T$. By the definition of $T$, there exists some constant $K>0$ such that 
    \begin{equation}\label{eq:unif-bound-Xs-propagation-sobolev-regularity}
        \lVert \mu_{t}\rVert_{\mathscr{X}_{s}}\le K\qquad \forall t\in [0,\tau').
    \end{equation}
    As a consequence of \Cref{lem:log-interp-BKM-criterion} and \eqref{eq:unif-bound-Xs-propagation-sobolev-regularity}, we get
    \begin{equation*}
        \lVert \nabla^{2}(-\Delta)^{-s}(\mu_{t}-\nu)\rVert_{L^{\infty}}\lesssim_{d,s,\gamma,K} 1+\log\left(2+\lVert \mu_{t}-\nu\rVert_{\dot H^{\gamma}}\right). 
    \end{equation*}
    This, combined with \eqref{eq:simplified-energy-estimate-integral-form} gives
    \begin{equation*}
        \lVert \nabla^{2}(-\Delta)^{-s}(\mu_{t}-\nu)\rVert_{L^{\infty}}\lesssim_{d,s,\gamma,K,\lVert \bar\mu\rVert_{H^{\gamma}}, \lVert \nu\rVert_{H^{\gamma}}, \tau'} 1+ \int_{0}^{t} \lVert \nabla^{2}(-\Delta)^{-s}(\mu_{r}-\nu)\rVert_{L^{\infty}}dr\qquad \forall t\in [0,\tau').
    \end{equation*}
    Thus, by Gr\"onwall's inequality, $\lVert \nabla^{2}(-\Delta)^{-s}(\mu_{t}-\nu)\rVert_{L^{\infty}}$ is uniformly bounded in $[0,\tau')$, and in particular $\mu \in L^{\infty}([0,\tau');H^{\gamma})$ thanks to \eqref{eq:simplified-energy-estimate-integral-form}. This contradicts \eqref{eq:Hgamma-must-explode-maximal-time} and concludes the proof.
 \end{proof}

\begin{proof}[Proof of \Cref{prop:local-well-posedness-intro}]
    Existence and uniqueness of a maximal solution in $\mathscr{X}_{s}(\T^{d})$ according to \Cref{def:maximal-solutions} is proved in \Cref{prop:maximal-solutions}, along with the blow-up criterion. Propagation of H\"{o}lder and Sobolev regularity from the data to all times in the maximal existence interval is obtained in \Cref{prop:propagation-regularity} and \Cref{prop:propagation-sobolev}, respectively. 
\end{proof}

\section{Quantitative convergence results}\label{sec:convergence}
In this section, we prove our quantitative convergence results for Riesz kernel mean discrepancies. \Cref{thm:convergence-s=1} (the Coulomb case $s=1$) is proved in \Cref{subsec:convergence-s=1}, while \Cref{thm:convergence-s>1} (the case $s>1$) is proved in \Cref{subsec:convergence-s>1}.

\subsection{The case $s=1$}\label{subsec:convergence-s=1}
We begin with the case $s=1$, which corresponds to the Coulomb interaction energy. In \Cref{subsec:max-principle-coulomb} we first establish a maximum principle, a strong structural feature of Coulomb's dynamics that fails for $s>1$. The maximum principle, the dissipation identity \eqref{eq:energy-dissipation-identity}, and the higher order energy estimate \eqref{eq:energy-estimate} are the main ingredients in the proof of \Cref{thm:convergence-s=1}, which is presented in \Cref{subsec:proof-convergence-s=1}. Finally, in \Cref{subsubsec:relaxation-lowerbound-initial}, we show that the lower bound on the initial measure is not necessary to obtain some instance of exponential convergence to the target (see \Cref{rmk:exponential-filling}). 

\subsubsection{The maximum principle}\label{subsec:max-principle-coulomb}

\begin{prop}\label{prop:max-princ-Coulomb}
       Let $s=1$, $\bar\mu, \nu \in \mathscr{P}\cap L^{\infty}(\T^{d})$, and $\mu\in L^{\infty}_{\loc}([0,T);L^{\infty}(\T^{d}))$ be the maximal solution of \eqref{eq:PDE-GF} given by \Cref{prop:maximal-solutions}. Then, $T=\infty$ and 
    \begin{equation}\label{eq:maximum-principle-s=1}
            \min\{\inf \bar{\mu}, \inf \nu\}\le \inf\mu_{t}\le \sup \mu_{t}\le \max \{\sup \bar{\mu}, \sup \nu\}\qquad \forall t\in [0,\infty).
    \end{equation}
\end{prop}
\begin{proof}
   By \Cref{prop:local-well-posedness-intro}, $T$ is finite if and only if $\limsup_{t\to T^{-}}\lVert \mu_{t}\rVert_{L^{\infty}}=\infty$. Therefore it suffices to prove the maximum principle \eqref{eq:maximum-principle-s=1} for every time in the interval of existence.  
   By a mollification argument on the initial and target measures, and the stability result from \Cref{prop:stability-yudovich}, we can restrict ourselves to prove the inequalities for $\bar\mu, \nu \in C^{\infty}$. In fact, the essential upper and lower bounds are preserved in the limit by weak convergence. In this case, $\mu$ along with the associated vector field $v$ and flow map $X$ are smooth in space-time by  \Cref{rmk:regularity-in-time}. We use the method of characteristics: it suffices to prove the upper and lower bounds for $\tilde{\mu}_{t}:=\mu_{t}\circ X_{t}$. By $\partial_{t}\mu_{t}+\diver(\mu_{t}v_{t})=0$, and $v_{t}=-\nabla (-\Delta)^{-1}(\mu_{t}-\nu)$, we deduce that $\tilde{\mu}$ solves
    \begin{equation}\label{eqn:maxpri}
        \frac{d}{dt}\tilde{\mu}_{t}(x)=-\left(\mu_{t}\diver v_{t}\right)\circ X_{t}(x)=-\tilde{\mu}_{t}(x)(\tilde{\mu}_{t}(x)-\nu\circ X_{t}(x))\qquad \forall x\in \T^{d},\, \forall t\in (0,T).
    \end{equation}
   We only prove the upper bound in \eqref{eq:maximum-principle-s=1}, the lower one being similar. Suppose by contradiction that there are $x_{0}\in \T^{d}$ and $t_{0}\in (0,T)$ such that
   $\tilde{\mu}_{t_{0}}(x_{0})> \max\{\max \bar\mu, \max \nu\}.$
 Let $t_{1}$ be the largest $t\in [0,t_{0})$ for which $\tilde{\mu}_{t}(x_{0})= \max\{\max \bar\mu, \max \nu\}$. Then, $\tilde{\mu}_{t}(x_{0})>\max \nu$ for every $t\in (t_{1},t_{0})$ and by \eqref{eqn:maxpri} is a monotonically decreasing function in $[t_1,t_0)$, a contradiction. 
\end{proof}

\subsubsection{Proof of \Cref{thm:convergence-s=1}}\label{subsec:proof-convergence-s=1}
 
\begin{proof}[Proof of \Cref{thm:convergence-s=1}]
In \Cref{prop:max-princ-Coulomb} we already proved that the solution $\mu$ is global in time and that the maximum principle \eqref{eq:max-princ-s=1-intro} holds. In particular, the solution is uniformly bounded:
    $$\lVert \mu_{t}\rVert_{L^{\infty}}\le \max\{\sup \bar\mu, \sup \nu\}=:M\qquad \forall t\in [0,\infty).$$ 
    We divide the proof of the remaining statements into four steps.

    \smallskip
    \noindent \textbf{Step 1:} In this step we prove that $\lVert\mu_{t}-\nu\rVert_{\dot H^{-1}}\to 0$ as $t\to \infty$. Together with the uniform $L^{\infty}$-bound this implies the weak-$*$ convergence to the target in $L^{\infty}$ and completes the proof of point i). Let us call $$\mathscr{P}_{M}(\T^{d}):= \{\eta \in \mathscr{P}\cap L^{\infty}(\T^{d}):\lVert \eta\rVert_{L^{\infty}}\le M\}.$$
    We first claim that 
    \begin{equation}\label{eq:claim-unconditioned-convergence-s=1}
        \forall \eps>0\quad \exists \delta>0:\qquad \lVert \eta-\nu\rVert_{\dot H^{-1}}\ge \eps \implies \int_{\T^{d}}|\nabla K_{1}*(\eta-\nu)|^{2}\eta\ge \delta\qquad \forall \eta\in \mathscr{P}_{M}(\T^{d}).
    \end{equation}
    Suppose by contradiction that there exist $\eps>0$ and $\{\eta_{j}\}_{j\ge 0}\subset \mathscr{P}_{M}(\T^{d})$ such that $$\lVert \eta_{j}-\nu\rVert_{\dot H^{-1}}\ge \eps\quad \forall j\ge 0\qquad \text{and} \qquad \int_{\T^{d}}|\nabla K_{1}*(\eta_{j}-\nu)|^{2}\eta_{j}\to 0\quad \text{as $j\to \infty$}.$$ Up to extracting a subsequence, we have $W_{2}(\eta_{j},\eta_{\infty})\to 0$ as $j\to \infty$ for some $\eta_{\infty}\in \mathscr{P}_{M}(\T^{d})$. In particular, thanks to \Cref{lem:W2-dual} and the uniform bound $\lVert \eta_{j}\rVert_{L^{\infty}}\le M$, we get $\lVert\eta_{j}-\eta_{\infty}\rVert_{\dot H^{-1}}\to 0$. We now prove that
    \begin{equation}\label{eq:zero-energy-dissipation-unconditioned-convergence-s=1}
        \int_{\T^{d}}|\nabla K_{1}*(\eta_{\infty}-\nu)|^{2}\eta_{\infty}=0.
    \end{equation}
    Using $(a+b)^{2}\le 2a^{2}+2b^{2}$ we can bound
\begin{align*}
    \int_{\T^{d}}|\nabla K_{1}*(\eta_{\infty}-\nu)|^{2}\eta_{\infty}&\le 
    \underbrace{\int_{\T^{d}}|\nabla K_{1}*(\eta_{\infty}-\nu)|^{2}(\eta_{\infty}-\eta_{j})}_{I^{1}_{j}}\\
    &\quad\,+\underbrace{2\int_{\T^{d}}|\nabla K_{1}*(\eta_{\infty}-\eta_{j})|^{2}\eta_{j}}_{I^{2}_{j}}+\underbrace{2\int_{\T^{d}}|\nabla K_{1}*(\eta_{j}-\nu)|^{2}\eta_{j}}_{I^{3}_{j}}.
\end{align*}
We have $I^{1}_{j}\to 0$ because $|\nabla K_{1}*(\eta_{\infty}-\nu)|^{2}\in L^{\infty}$ by \Cref{lem:W2-Linfty} and $\eta_{j}\rightharpoonup \eta_{\infty}$ in $L^{2}$. On the other hand, $I^{2}_{j}\le 2M\lVert \eta_{j}-\eta_{\infty}\rVert_{\dot H^{-1}}^{2}\to 0$ and $I^{3}_{j}\to 0$ as assumed. Therefore, we have \eqref{eq:zero-energy-dissipation-unconditioned-convergence-s=1}, and in particular, $\nabla K_{1}*(\eta_{\infty}-\nu)(x)=0$ for $\eta_{\infty}$-a.e.~$x\in \T^{d}$. Hence, by Sobolev regularity, 
    \begin{equation*}
        \nu(x)-\eta_{\infty}(x)=\diver \nabla K_{1}*(\eta_{\infty}-\nu)(x)=0\qquad \text{for $\eta_{\infty}$-a.e.~$x\in \T^{d}$},
    \end{equation*}
    which in turn implies that $\eta_{\infty}$ and $\nu$ actually coincide, since they have the same total mass and $\eta_\infty \ll \mathscr{L}^{d}$. This is a contradiction because $\lVert \eta_{\infty}-\nu\rVert_{\dot H^{-1}}=\lim_{j\to \infty}\lVert \eta_{j}-\nu\rVert_{\dot H^{-1}}\ge \eps>0$, and claim \eqref{eq:claim-unconditioned-convergence-s=1} is proved. 

    Now observe that $\mu_{t}\in \mathscr{P}_{M}(\T^{d})$ and $\lVert \mu_{t}-\nu\rVert_{\dot H^{-1}}$ is decreasing in time by \Cref{prop:gradient-flow-structure}. If we had
    $$\eps:= \lim_{t\to \infty}\lVert \mu_{t}-\nu\rVert_{\dot H^{-1}}^{2}>0,$$
    by the energy dissipation identity \eqref{eq:energy-dissipation-identity} we would get
    \begin{equation*}
        \frac{d}{dt}\lVert \mu_{t}-\nu\rVert_{\dot H^{-1}}^{2}=-2\int_{\T^{d}}|\nabla K_{1}(\mu_{t}-\nu)|^{2}\mu_{t}\le -2\delta\qquad \text{for a.e.~$t\in (0,\infty)$},
    \end{equation*}
    where $\delta>0$ is given by claim \eqref{eq:claim-unconditioned-convergence-s=1}. This is the desired contradiction. 

    \smallskip
    \noindent \textbf{Step 2:} From now on we assume that $\bar\mu, \nu\ge \alpha>0$. In this step we prove point i). By the maximum principle from \eqref{eq:max-princ-s=1-intro}, we know that 
    $$\inf \mu_{t}\ge \alpha\qquad \forall t\in [0,\infty).$$
    In particular, the first inequality in \eqref{eq:exponential-decay-H^-1-coulomb} follows from the following general fact, which can be derived from the Benamou-Brenier formula (see for instance \cite[Exercise A.16, p. 137]{figalli2021invitation}):
    $$W_{2}(\eta_{1},\eta_{2})\le \alpha^{-1/2}\lVert \eta_{1}-\eta_{2}\rVert_{\dot H^{-1}}\qquad \forall \eta_{1},\eta_{2}\in \mathscr{P}(\T^{d}), \quad 0<\alpha\le \eta_{1},\eta_{2}\ll \mathscr{L}^{d}.$$
    
    Moreover, using the energy dissipation identity \eqref{eq:energy-dissipation-identity} we obtain
    \begin{equation*}
        \frac{d}{dt}\lVert \mu_{t}-\nu\rVert_{\dot H^{-1}}^{2}=-2\int_{\T^{d}}|\nabla K_{1}*(\mu_{t}-\nu)|^{2}\mu_{t}\le -2\alpha \lVert \mu_{t}-\nu\rVert_{\dot H^{-1}}^{2}\qquad \text{$\forall t\in (0,\infty)$}.
    \end{equation*}
    Integrating this differential inequality, we get precisely the exponential decay in equation \eqref{eq:exponential-decay-H^-1-coulomb}. 

    \smallskip
    \noindent \textbf{Step 3:} Next we prove point ii). By an approximation argument based on \Cref{prop:stability-yudovich}, similar to the one used, for instance, in the proof of \Cref{prop:max-princ-Coulomb}, it is sufficient to work in the case $\bar\mu,\nu \in C^{\infty}$, in which $\mu\in C^{\infty}([0,\infty)\times \T^{d})$ (see \Cref{rmk:regularity-in-time}). Let $v_{t}=-\nabla K_{1}*(\mu_{t}-\nu)$ be the vector field generated by the solution $\mu$. Using \Cref{lem:W2-Linfty} and \eqref{eq:exponential-decay-H^-1-coulomb} we find the following uniform bound on $v_{t}$:
    \begin{equation}\label{eq:bound-Linfty-vectorfield-s=1}
        \lVert v_{t}\rVert_{L^{\infty}}\le C_{d}M^{\frac{d}{d+1}}W_{2}(\mu_{t},\nu)^{\frac{1}{d+1}}\le C\left(d,\alpha,M, \lVert \bar\mu-\nu\rVert_{\dot H^{-1}}\right)e^{-\frac{\alpha}{d+1}t}\qquad \forall t\in [0,\infty).
    \end{equation}
    Let $X:[0,\infty)\times \T^{d}\to \T^{d}$ be the flow map associated to the vector field $v$. Since $\lVert X_{t}-X_{r}\rVert_{L^{\infty}}\le \int_{r}^{t}\lVert v_{u}\rVert_{L^{\infty}}du$, we deduce from \eqref{eq:bound-Linfty-vectorfield-s=1} that $(X_{t})_{t\ge 0}$ is Cauchy in uniform norm as $t\to \infty$, thus $X_{t}$ converges uniformly as $t\to \infty$ towards some continuous map $X_{\infty}:\T^{d}\to \T^{d}$. More precisely,
    \begin{equation}\label{eq:uniform-convergence-flow-map-s=1}
        \lVert X_{t}-X_{\infty}\rVert_{L^{\infty}}\le \int_{t}^{\infty}\lVert v_{r}\rVert_{L^{\infty}}dr \le C\left(d,\alpha,M, \lVert \bar\mu-\nu\rVert_{\dot H^{-1}}\right)e^{-\frac{\alpha}{d+1}t}\qquad \forall t\in [0,\infty).
    \end{equation}
    We first assume $\nu\in C^{0,\beta}$ for some $\beta\in (0,1)$ and prove that $\mu_{t}$ converges to $\nu$ in uniform norm exponentially fast in time. Indeed, having control over the $C^{0,\beta}$-norm of $\nu$, from \eqref{eq:uniform-convergence-flow-map-s=1} we  deduce 
    \begin{equation}\label{eq:uniform-detachment-target-flow-map-s=1}
        \lVert \nu\circ X_{t}-\nu\circ X_{\infty}\rVert_{L^{\infty}}\le \lVert\nu\rVert_{C^{0,\beta}}\lVert X_{t}-X_{\infty}\rVert_{L^{\infty}}^{\beta}\le   C\left(d,\alpha,\beta, M,\lVert \bar\mu-\nu\rVert_{\dot H^{-1}},\lVert\nu\rVert_{C^{0,\beta}}\right)e^{-\kappa t},
    \end{equation}
   where $\kappa=\alpha\beta/(d+1)$ . We now show that $\tilde{\mu}_{t}:= \mu_{t}\circ X_{t}$ converges uniformly to $\nu\circ X_{\infty}$ exponentially fast in time. In fact, we have
   \begin{align*}
       \frac{d}{dt}(\tilde{\mu}_{t}-\nu\circ X_{\infty})&=-(\mu_{t}\diver v_{t})\circ X_{t}\\
       &=-\tilde{\mu}_{t}(\tilde{\mu}_{t}-\nu\circ X_{t})
       =-\tilde{\mu}_{t}(\tilde{\mu}_{t}-\nu\circ X_{\infty})+\tilde{\mu_{t}}(\nu\circ X_{t}-\nu\circ X_{\infty}).
   \end{align*}
   Integration gives
   \begin{equation*}
       \tilde{\mu}_{t}-\nu\circ X_{\infty}=(\bar\mu-\nu\circ X_{\infty})\exp\left(-\int_{0}^{t}\tilde{\mu}_{r}dr\right)+\int_{0}^{t}\tilde{\mu_{r}}(\nu\circ X_{r}-\nu\circ X_{\infty})\exp\left(-\int_{r}^{t}\tilde{\mu}_{u}du\right) \, dr.
   \end{equation*}
   Therefore, since $\alpha \le \tilde{\mu}_{t}\le M$ for all times, using also \eqref{eq:uniform-detachment-target-flow-map-s=1} we get
   \begin{equation}\label{eq:bound-lagrangian-Linfty-s=1}
       \lVert \tilde{\mu}_{t}-\nu\circ X_{\infty}\rVert_{L^{\infty}}\le C\left(d,\alpha,\beta, M,\lVert \bar\mu-\nu\rVert_{\dot H^{-1}},\|\nu\|_{C^{0,\beta}}\right)e^{-\kappa t}.
   \end{equation}
   Finally, combining \eqref{eq:uniform-detachment-target-flow-map-s=1} and \eqref{eq:bound-lagrangian-Linfty-s=1}, we derive
   \begin{align*}
       \lVert \mu_{t}-\nu\rVert_{L^{\infty}}&=\lVert \tilde{\mu}_{t}-\nu\circ X_{t}\rVert_{L^{\infty}}\\
       &\le \lVert \tilde{\mu_{t}}-\nu\circ X_{\infty}\rVert_{L^{\infty}}+ \lVert \nu\circ X_{t}-\nu\circ X_{\infty}\rVert_{L^{\infty}}\le C\left(d,\alpha,\beta, M,\lVert \bar\mu-\nu\rVert_{\dot H^{-1}},\|\nu\|_{C^{0,\beta}}\right)e^{-\kappa t},
   \end{align*}
    as desired.

    Suppose now that $\nu \in C(\T^{d})$ has only a Dini modulus of continuity, i.e.~there is a continuous concave nondecreasing function $\omega:[0,\infty)\to [0,\infty)$, such that
    \begin{equation*}
        \omega(0)=0,\quad \int_{0}^{1}\frac{\omega(r)}{r}dr<\infty,\qquad \text{and}\qquad |\nu(x)-\nu(y)|\le \omega(|x-y|)\qquad \forall x,y \in \T^{d}.
    \end{equation*}
    Then, by \eqref{eq:uniform-convergence-flow-map-s=1} and a change of variables, we get
    \begin{equation}\label{eq:integrability-time-dini-modulus-flow}
    \begin{aligned}
        \int_{0}^{\infty}\lVert \nu\circ X_{t}-\nu\circ X_{\infty}\rVert_{L^{\infty}}dt&\le \int_{0}^{\infty}\omega\left(\lVert X_{t}-X_{\infty}\rVert_{L^{\infty}}\right)dt\\
        &\le \int_{0}^{\infty}\omega\left(C\left(d,\alpha,M\right)e^{-\frac{\alpha}{d+1}t}\right)dt \\
        &\lesssim_{d,\alpha,M}\int_{0}^{C(d,\alpha,M)}\frac{\omega(r)}{r}dr\lesssim_{d,s,M,\omega}1.
    \end{aligned}
    \end{equation}
    From this, repeating the same exact steps as in case $\nu \in C^{0,\beta}$ above, we obtain $\lVert \mu_{t}-\nu\rVert_{L^{\infty}}\to 0$ as $t\to \infty$.

   \smallskip
   \noindent \textbf{Step 4:} In this final step, we prove point iii). 
   Let $\gamma>d/2$. By the usual approximation argument (see, for instance, the proof of \Cref{prop:propagation-sobolev}) we may reduce to prove the result assuming $\bar\mu, \nu \in C^{\infty}(\T^{d})$, and so $\mu \in C^{\infty}([0,\infty)\times \T^{d})$. Let us call $\sigma_{t}=\mu_{t}-\nu$, and consider the energy estimate \eqref{eq:energy-estimate}.  Interpolating
   $\lVert \sigma_{t}\rVert_{\dot H^{\gamma-1}}$ between $\lVert \sigma_{t}\rVert_{\dot H^{-1}}$ and $\lVert \sigma_{t}\rVert_{\dot H^{\gamma}}$, using Young's inequality, and the lower bound $\nu\ge \alpha>0$, the second line in the right-hand side of \eqref{eq:energy-estimate} can be estimated by
   \begin{align*}
    -2\big(\min_{\T^{d}}\nu\big)\lVert \sigma_{t}\rVert_{\dot H^{\gamma}}^{2}+ C\lVert \nu\rVert_{\dot H^{\gamma+1}}\lVert \sigma_{t}\rVert_{\dot H^{\gamma-1}}\lVert \sigma_{t}\rVert_{\dot H^{\gamma}}&\le C(d,\gamma,\alpha, \lVert \nu\rVert_{\dot H^{\gamma+1}})\lVert \sigma_{t}\rVert_{\dot H^{-1}}^{2}.
\end{align*}
Therefore, recalling the exponential decay of $\lVert \sigma_{t}\rVert_{\dot H^{-1}}$ from \eqref{eq:exponential-decay-H^-1-coulomb}, the energy estimate \eqref{eq:energy-estimate} simplifies to
\begin{equation*}
         \frac{d}{dt}\lVert \sigma_{t}\rVert_{\dot H^{\gamma}}^{2}\le C\lVert \nabla^{2}(-\Delta)^{-1}\sigma_{t}\rVert_{L^{\infty}}   \lVert \sigma_{t}\rVert_{\dot H^{\gamma}}^{2}+ Ce^{-2\alpha t},     
\end{equation*}
which in turn, integrated, yields
\begin{equation}\label{eq:energy-estimate-simplified-s=1}
    \begin{aligned}
         \lVert \sigma_{t}\rVert_{\dot H^{\gamma}}^{2}&\le \lVert \sigma_{0}\rVert_{\dot H^{\gamma}}^{2}\exp\left(C\int_{0}^{t}\lVert\nabla^{2}(-\Delta)^{-1}\sigma_{r}\rVert_{L^{\infty}}dr\right)+C\int_{0}^{t}e^{-2\alpha r}\exp\left(C\int_{r}^{t}\lVert \nabla^{2}(-\Delta)^{-1}\sigma_{u}\rVert_{L^{\infty}}du\right)\\
       &\le \left(\lVert \sigma_{0}\rVert_{\dot H^{\gamma}}^{2}+ C(1-e^{-2\alpha t})\right)\exp\left(C\int_{0}^{t}\lVert\nabla^{2}(-\Delta)^{-1}\sigma_{r}\rVert_{L^{\infty}}dr\right).        
    \end{aligned}
\end{equation}
 Further, note that $\nu \in C^{1}(\T^{d})$ by Sobolev embedding. Therefore, we may use point ii) to get exponential decay of the $L^{\infty}$-norm of $\sigma_{t}$: 
   \begin{equation}\label{eq:Linfty-bound-last-step-convergence-s=1}
       \lVert \sigma_{t}\rVert_{L^{\infty}}\le C(d,\gamma,\alpha,\lVert \bar\mu\rVert_{H^{\gamma}}, \lVert \nu\rVert_{H^{\gamma+1}})e^{-\frac{\alpha}{d+1}t}\qquad \forall t\in [0,\infty).
   \end{equation}
   As a consequence of \Cref{lem:log-interp-BKM-criterion}, \eqref{eq:energy-estimate-simplified-s=1}, and \eqref{eq:Linfty-bound-last-step-convergence-s=1}, the $L^{\infty}$-norm of the gradient of the vector field can be bounded as follows:
   \begin{equation}\label{eq:bound-Lip-gradient-convergence-s=1}
       \begin{aligned}
       \lVert \nabla^{2} (-\Delta)^{-1}\sigma_{t}\rVert_{L^{\infty}}&\le C(d,\gamma)\lVert \sigma_{t}\rVert_{L^{\infty}}\left(1+\log\left(1+\frac{\lVert \sigma_{t}\rVert_{\dot H^{\gamma}}}{\lVert \sigma_{t}\rVert_{L^{\infty}}}\right)\right)\\
       &\le C(d,\gamma)\lVert \sigma_{t}\rVert_{L^{\infty}}\left(1+\log(1+\lVert \sigma_{t}\rVert_{L^{\infty}}^{-1})+\log(1+\lVert \sigma_{t}\rVert_{\dot H^{\gamma}})\right)\\
       &\le C(d,\gamma,\alpha,\lVert \bar\mu\rVert_{H^{\gamma}}, \lVert \nu\rVert_{\dot H^{\gamma+1}})e^{-\frac{\alpha}{d+1} t}\left(1+t+ \int_{0}^{t}\lVert\nabla^{2}(-\Delta)^{-1}\sigma_{r}\rVert_{L^{\infty}}dr\right).
   \end{aligned}
   \end{equation}
Gr\"onwall's inequality then gives
   \begin{equation*}
       \int_{0}^{\infty}\lVert\nabla^{2}(-\Delta)^{-1}\sigma_{r}\rVert_{L^{\infty}}dr\le C(d,\gamma,\alpha,\lVert \bar\mu\rVert_{H^{\gamma}}, \lVert \nu\rVert_{\dot H^{\gamma+1}}),
   \end{equation*}
   and, in particular, by \eqref{eq:energy-estimate-simplified-s=1},
   
   \begin{equation}\label{eq:uniform-bound-Hgamma-convergence-s=1}
       \lVert \sigma_{t}\rVert_{\dot H^{\gamma}}\le C(d,\gamma,\alpha,\lVert \bar\mu\rVert_{H^{\gamma}}, \lVert \nu\rVert_{H^{\gamma+1}}). 
   \end{equation}
   At this point, combining the uniform bound of $\lVert \sigma_{t}\rVert_{\dot H^{\gamma}}$ in \eqref{eq:uniform-bound-Hgamma-convergence-s=1} with the exponential decay of $\lVert \sigma_{t}\rVert_{\dot H^{-s}}$ from \eqref{eq:exponential-decay-H^-1-coulomb}, by Sobolev interpolation we deduce that
   \begin{equation}\label{eq:exponential-decay-lower-sobolev-norms-s=1}
       \lVert \nabla^{2}(-\Delta)^{-1}\sigma_{t}\rVert_{L^{\infty}},\, \lVert \sigma_{t}\rVert_{\dot H^{\gamma-1}}\le  Ce^{-t/C}.
   \end{equation}
   Plugging \eqref{eq:exponential-decay-lower-sobolev-norms-s=1} and \eqref{eq:uniform-bound-Hgamma-convergence-s=1} inside the energy estimate \eqref{eq:energy-estimate} we find
   \begin{equation*}
       \frac{d}{dt}\lVert \sigma_{t}\rVert_{\dot H^{\gamma}}^{2}\le (Ce^{-t/C}-2\alpha)\lVert \sigma_{t}\rVert_{\dot H^{\gamma}}^{2}+Ce^{-t/C},
   \end{equation*}
   and finally, integrating this differential inequality, we obtain
   \begin{align*}
       \lVert \sigma_{t}\rVert_{\dot H^{\gamma}}^{2}&\le \lVert \sigma_{0}\rVert_{\dot H^{\gamma}}^{2}\exp\left(\int_{0}^{t}\left(Ce^{-r/C}-2\alpha\right)dr\right)+\int_{0}^{t}Ce^{-r/C}\exp\left(\int_{r}^{t}\left(Ce^{-u/C}-2\alpha\right)du\right)\\
       &\le \left(\lVert \sigma_{0}\rVert_{\dot H^{\gamma}}^{2}+C(1-e^{-t/C})\right)\exp\left(-2\alpha t+ C(1-e^{-t/C})\right).
   \end{align*}
   Taking the square root of the above we get the desired decay for $\lVert \sigma_{t}\rVert_{\dot H^{\gamma}}$, thus concluding the proof. 
   \end{proof}

\subsubsection{Exponential convergence without lower bound on the initial measure}\label{subsubsec:relaxation-lowerbound-initial}
In this section, we show that the assumption $\bar\mu\ge \alpha>0$ can be relaxed while still obtaining exponential weak convergence to the target. The key ingredient is the lemma below, which asserts that \lq\lq holes'' in the support of a smooth solution are \lq\lq filled-up'' at an exponential rate, provided the target is uniformly bounded from below.
\begin{lem}[Exponential filling of holes]\label{lem:exponential-filling-holes}
    Let $\bar \mu, \nu \in \mathscr{P}\cap C^{\infty}(\T^{d})$, and suppose that $0<\lambda\le \nu\le \Lambda<\infty$ in $\T^{d}$. Let $\mu\in C^{\infty}([0,\infty)\times \T^{d})$ be the corresponding solution of \eqref{eq:PDE-GF}. Then, the following hold:
    $$|\{\mu_{t}\le a\}|\le |\{\bar \mu\le a\}|e^{-(\lambda-a)t}\quad \forall a\in [0,\lambda];\qquad |\{\mu_{t}\ge b\}|\le |\{\bar\mu \ge b\}|e^{-(b-\Lambda)t}\quad \forall b\in [\Lambda,+\infty).$$
\end{lem}
\begin{proof}
    We use the method of characteristics. 
    Recall that $\tilde{\mu}_{t}:= \mu_{t}\circ X_{t}$ solves $$\frac{d}{dt}\tilde{\mu}_{t}(x)= \tilde{\mu}_{t}(x)(\nu\circ X_{t}(x)-\tilde{\mu}_{t}(x)).$$
    From this equation, taking into account that $\lambda\le \nu\le \Lambda$, we deduce 
    \begin{equation}\label{eq:inclusion-levels-exponential-filling}
        \{\tilde{\mu}_{t}\le a\}\subseteq \{\tilde{\mu}_{r}\le a\}\quad \forall a\le \lambda,\, \forall r\le t;\qquad \{\tilde{\mu}_{t}\ge b\}\subseteq \{\tilde{\mu}_{r}\ge b\}\quad \forall b\ge \Lambda,\, \forall r\le t.
    \end{equation}
    For a given $a\in [0,\lambda]$ we have
    $\{\mu_{t}\le a\}=X_{t}\left(\{\tilde{\mu}_{t}\le a\}\right)$. 
    In particular, by the change of variables formula, we find
    $$|\{\mu_{t}\le a\}|= |X_{t}\left(\{\tilde{\mu}_{t}\le a\}\right)|=\int_{\{\tilde{\mu}_{t}\le a\}}\det \nabla X_{t}(x).$$
    Now, let $x\in \{\tilde{\mu}_{t}\le a\}$. By \eqref{eq:inclusion-levels-exponential-filling} we have $\tilde{\mu}_{r}(x)\le a$ for all $r\in [0,t]$. Therefore, 
    $$\det \nabla X_{t}(x)= \exp\left(\int_{0}^{t}\diver v_{r}\circ X_{r}(x)dr\right)= \exp \left(\int_{0}^{t}(\tilde{\mu}_{r}(x)-\nu\circ X_{r}(x))dr\right)\le e^{-(\lambda-a)t}.$$
    Using \eqref{eq:inclusion-levels-exponential-filling} again we finally obtain
    $$|\{\mu_{t}\le a\}|\le |\{\tilde{\mu}_{t}\le a\}|e^{-(\lambda-a)t}\le |\{\tilde{\mu}_{0}\le a\}|e^{-(\lambda-a)t}= |\{\bar\mu\le a\}|e^{-(\lambda-a)t},$$
    as desired. A similar argument applies to superlevel sets $\{\mu_{t}\ge b\}$ when $b\ge \Lambda$.
\end{proof}
    
    \Cref{lem:exponential-filling-holes} implies the exponential weak convergence to equilibrium only assuming a lower bound on the target measure.
    \begin{prop}\label{prop:relaxation-lower-bound-initial-measure}
        Let $\bar\mu,\nu \in \mathscr{P}\cap L^{\infty}(\T^d)$, and suppose that $\nu \ge \alpha>0$. Let $\mu\in L^{\infty}([0,\infty);L^{\infty}(\T^{d}))$ be the corresponding solution of \eqref{eq:PDE-GF}. Then
        $$\lVert \mu_{t}-\nu\rVert_{\dot H^{-1}}\le e^{-\frac{\alpha t}{2}}\left(\lVert \bar\mu-\nu\rVert_{\dot H^{-s}}+C(d)\left(\lVert \bar\mu\rVert_{L^{\infty}}+\lVert \nu\rVert_{L^{\infty}}\right)\big(e^{\frac{\alpha t}{4}}-1\big)\right)\qquad \forall t\in [0,\infty).$$
    \end{prop}
    \begin{proof}
    By the usual approximation argument on initial and target measures, we can assume that $\bar\mu, \nu \in \mathscr{P}\cap C^{\infty}(\T^{d})$. 
    Using the energy dissipation identity \eqref{eq:energy-dissipation-identity} we can bound
        \begin{align*}
            \frac{d}{dt}\lVert \mu_{t}-\nu\rVert_{\dot H^{-1}}^{2}&=-2\int_{\T^{d}}|\nabla K_{1}*(\mu_{t}-\nu)|^{2}\mu_{t}\le -\alpha\int_{\{\mu_{t}>\alpha/2\}}|\nabla K_{1}*(\mu_{t}-\nu)|^{2}\\
            &=-\alpha\int_{\T^{d}}|\nabla K_{1}*(\mu_{t}-\nu)|^{2}+\alpha\int_{\{\mu_{t}\le \alpha/2\}}|\nabla K_{1}*(\mu_{t}-\nu)|^{2}\\
            &\le -\alpha\lVert \mu_{t}-\nu\rVert_{H^{-1}}^{2}+\alpha\lVert \nabla K_{1}*(\mu_{t}-\nu)\rVert_{L^{\infty}}^{2}|\{\mu_{t}\le \alpha/2\}|.
        \end{align*}
    By \Cref{lem:W2-Linfty} and the maximum principle \eqref{eq:max-princ-s=1-intro}, $\lVert \nabla K_{1}*(\mu_{t}-\nu)\rVert_{L^{\infty}}\lesssim_{d}\lVert \bar\mu\rVert_{L^{\infty}}+\lVert \nu\rVert_{L^{\infty}}$. Moreover, \Cref{lem:exponential-filling-holes} ensures that $|\{\mu_{t}\le \alpha/2\}|\le e^{-\alpha t/2}$. Thus, we get
    \begin{equation*}
        \frac{d}{dt}\lVert \mu_{t}-\nu\rVert_{\dot H^{-1}}^{2}\le -\alpha\lVert \mu_{t}-\nu\rVert_{H^{-1}}^{2}+\alpha C(d)\left(\lVert \bar\mu\rVert_{L^{\infty}}+\lVert \nu\rVert_{L^{\infty}}\right)^{2}e^{-\frac{\alpha t}{2}},
    \end{equation*}
    which, integrated, yields the desired decay. 
    \end{proof}

\subsection{The case $s>1$.}\label{subsec:convergence-s>1}
This section is devoted to the proof of \Cref{thm:convergence-s>1}. As explained in \Cref{subsec:idea-proof}, the argument relies on three main ingredients: the $\dot H^{-s}$ energy dissipation identity \eqref{eq:energy-dissipation-identity}, the $\dot H^{\gamma}$ energy estimate in \Cref{lem:energy-estimates}, and classical Sobolev interpolation. Combining these elements, we are able to enforce a suitable \L{}ojasiewicz gradient inequality along the dynamics, under a small discrepancy assumption on the initial data. 

\begin{proof}[Proof of \Cref{thm:convergence-s>1}] 
First we note that it suffices to prove the estimates in \eqref{eq:conclusion-thm-conv-s>1} for all times in the maximal existence interval $[0,T)$: this would automatically imply that $T=\infty$, thanks to the continuation criterion from \Cref{prop:local-well-posedness-intro}.
Secondly, by the usual approximation argument (see for instance the proof of \Cref{prop:propagation-sobolev}), we may reduce to prove \eqref{eq:conclusion-thm-conv-s>1} for smooth data $\bar\mu, \nu \in C^{\infty}(\T^{d})$, in which case $\mu\in C^{\infty}([0,T)\times \T^{d})$ by \Cref{rmk:regularity-in-time}.

    Let us call $\sigma_{t}:=\mu_{t}-\nu$, and consider the energy estimate from \eqref{eq:energy-estimate}. Similarly to Step 4 in the proof of \Cref{thm:convergence-s=1}, we can interpolate $\lVert \sigma_{t}\rVert_{\dot H^{\gamma-s}}$ between $\lVert \sigma_{t}\rVert_{\dot H^{-s}}$ and $\lVert \sigma_{t}\rVert_{\dot H^{\gamma-s+1}}$, and use Young's inequality along with the lower bound $\nu\ge \alpha>0$ to get
    \begin{equation*}
        -2\big(\min_{\T^{d}}\nu\big)\lVert \sigma_{t}\rVert_{\dot H^{\gamma-s+1}}^{2}+ C\lVert \nu\rVert_{\dot H^{\gamma+s}}\lVert \sigma_{t}\rVert_{\dot H^{\gamma-s}}\lVert \sigma_{t}\rVert_{\dot H^{\gamma-s+1}}\le C(d,s,\gamma,\alpha,\lVert \nu\rVert_{\dot H^{\gamma+s}})\lVert \sigma_{t}\rVert_{\dot H^{-s}}^{2}.
    \end{equation*}
    Then, \eqref{eq:energy-estimate} simplifies to
    \begin{equation*}
         \frac{d}{dt}\lVert \sigma_{t}\rVert_{\dot H^{\gamma}}^{2}\le C\lVert \nabla^{2}(-\Delta)^{-s}\sigma_{t}\rVert_{L^{\infty}}   \lVert \sigma_{t}\rVert_{\dot H^{\gamma}}^{2}+ C\lVert \sigma_{t}\rVert_{\dot H^{-s}}^{2},     
\end{equation*}
which, integrated, yields
    \begin{equation}\label{eq:energy-estimate-simplified-s>1}
        \begin{aligned}
         \lVert \sigma_{t}\rVert_{\dot H^{\gamma}}^{2}&\le \lVert \sigma_{0}\rVert_{\dot H^{\gamma}}^{2}\exp\left(C\int_{0}^{t}\lVert\nabla^{2}(-\Delta)^{-s}\sigma_{r}\rVert_{L^{\infty}}dr\right)\\
         &\quad\,+C\int_{0}^{t}\lVert \sigma_{r}\rVert_{\dot H^{-s}}^{2}\exp\left(C\int_{r}^{t}\lVert \nabla^{2}(-\Delta)^{-s}\sigma_{u}\rVert_{L^{\infty}}du\right)dr.       
    \end{aligned}
    \end{equation}
    On the other hand, the energy dissipation identity \eqref{eq:energy-dissipation-identity}, combined with Sobolev interpolation of $\lVert \sigma_{t}\rVert_{\dot H^{-s}}$ between $\lVert \sigma_{t}\rVert_{\dot H^{1-2s}}$ and $\lVert \sigma_{t}\rVert_{\dot H^{\gamma}}$ gives
    \begin{equation}\label{eq:energy-dissipation-s>1-interpolated}
    \begin{aligned}
         \frac{d}{dt}\lVert \sigma_{t}\rVert_{\dot H^{-s}}^{2}&= -2\int_{\T^{d}}|\nabla (-\Delta)^{-s}\sigma_{t}|^{2}\mu_{t}\\
         &\le -2\big(\inf_{\T^{d}}\mu_{t}\big)\lVert \sigma_{t}\rVert_{\dot H^{1-2s}}^{2}\le -2\big(\inf_{\T^{d}}\mu_{t}\big)\lVert \sigma_{t}\rVert_{\dot H^{\gamma}}^{-\frac{2s-2}{\gamma+s}}\lVert \sigma_{t}\rVert_{\dot H^{-s}}^{2+\frac{2s-2}{\gamma+s}}.
    \end{aligned}
    \end{equation}
   At this point, we observe the following. On the one hand, a uniform upper bound on $\lVert\sigma_t\rVert_{\dot H^{\gamma}}$, together with a uniform lower bound on $\inf \mu_t$ for all $t\in[0,T)$, yields the desired polynomial decay of $\lVert\sigma_t\rVert_{\dot H^{-s}}$ by integrating \eqref{eq:energy-dissipation-s>1-interpolated}. On the other hand, sufficiently fast (integrable in time) decay of $\lVert \sigma_t\rVert_{\dot H^{-s}}$ implies a uniform bound on $\lVert \sigma_t\rVert_{\dot H^{\gamma}}$ through \eqref{eq:energy-estimate-simplified-s>1}.
In the remainder of the proof we show that these bounds indeed hold uniformly on $[0,T)$, provided $\lVert\sigma_{0}\rVert_{\dot H^{-s}}$ is chosen sufficiently small. This yields \eqref{eq:conclusion-thm-conv-s>1}. 
    
    Calling $M:= \lVert \sigma_{0}\rVert_{\dot H^{\gamma}}^{2}$, from the heuristics above we are led to consider  
    \begin{equation*}
        \tau:= \sup\left\{t>0: \lVert\sigma_{r}\rVert_{\dot H^{\gamma}}^{2}\le 2M \text{\,\, and \,\,} \inf_{\T^{d}} \mu_{r}\ge \alpha/2,\,\,  \forall r\in [0,t)\right\}\in (0,T].
    \end{equation*}
    Our goal is to prove that if $\lVert \sigma_{0}\rVert_{\dot H^{-s}}$ is sufficiently small, then $\tau=T$. 
    By \eqref{eq:energy-dissipation-s>1-interpolated} and the definition of $\tau$, we have 
    \begin{equation*}
        \frac{d}{dt}\lVert \sigma_{t}\rVert_{\dot H^{-s}}^{2}\le -\kappa\lVert \sigma_{t}\rVert_{\dot H^{-s}}^{2+\frac{2s-2}{\gamma+s}}\qquad \forall t\in [0,\tau),
    \end{equation*}
    where we set for convenience $\beta=(\gamma+s)/(s-1)$ and $\kappa:= \alpha (2M)^{-1/\beta}$.
    After integration we get
    \begin{equation}\label{eq:polynomial-decay-H^-s-s>1}
        \lVert \sigma_{t}\rVert_{\dot H^{-s}}^{2}\le   \lVert \sigma_{0}\rVert_{\dot H^{-s}}^{2}\left(1+K t\right)^{-\frac{\gamma+s}{s-1}}\qquad \forall t\in [0,\tau),
    \end{equation}
    where $K:=\alpha\beta^{-1}(\lVert \sigma_{0}\rVert_{\dot H^{-s}}^{2}/2\lVert \sigma_{0}\rVert_{\dot H^{\gamma}}^{2})^{1/\beta}$.
    Moreover, calling $\eta:= \gamma-d/2>0$ and $\hat{\gamma}:= d/2+2-2s+\eta/2$, the boundedness of $\nabla^{2}(-\Delta)^{-s}: \dot H^{\hat{\gamma}}\to L^{\infty}$, combined with Sobolev interpolation, gives
    \begin{align*}
        \lVert \nabla^{2}(-\Delta)^{-s}\sigma_{t}\rVert_{L^{\infty}}\le C\lVert \sigma_{t}\rVert_{\dot H^{\hat{\gamma}}}\le C\lVert \sigma_{t}\rVert_{\dot H^{-s}}^{1-\theta}\lVert \sigma_{t}\rVert_{\dot H^{\gamma}}^{\theta},\qquad \theta= 1-\frac{2(s-1)+\eta/2}{\gamma+s}.
    \end{align*}
    Hence, from \eqref{eq:polynomial-decay-H^-s-s>1} we get
    \begin{equation*}
        \lVert \nabla^{2}(-\Delta)^{-s}\sigma_{t}\rVert_{L^{\infty}}\le C(2M)^{\theta}\lVert \sigma_{0}\rVert_{\dot H^{-s}}^{1-\theta}(1+Kt)^{-1-\frac{\eta}{4(s-1)}}\qquad \forall t\in [0,\tau),
    \end{equation*}
    and integrating in time between $0$ and $\tau$ we obtain
    \begin{equation}\label{eq:integral-bound-Linfty-vector-field-s>1}
        \begin{aligned}
            \int_{0}^{\tau}\lVert \nabla^{2}(-\Delta)^{-s}\sigma_{r}\rVert_{L^{\infty}}dr&\le C(2M)^{\theta}\lVert \sigma_{0}\rVert_{\dot H^{-s}}^{1-\theta}\int_{0}^{\infty}(1+Kr)^{-1-\frac{\eta}{4(\gamma+s)}}dr\\
        &\le  C(2M)^{\theta}\frac{\lVert \sigma_{0}\rVert_{\dot H^{-s}}^{1-\theta}}{K}\le C(d,s,\gamma,M,\alpha)\lVert \sigma_{0}\rVert_{\dot H^{-s}}^{\frac{\eta}{2(\gamma+s)}}.
        \end{aligned}
    \end{equation}
    From here, we deduce a uniform lower bound on $\mu_{t}$. Indeed, since $\inf \bar\mu \ge \alpha$, and the gradient of the vector field advecting $\mu_{t}$ is precisely $-\nabla^{2}(-\Delta)^{-s}\sigma_{t}$,
    \begin{equation}\label{eq:uniform-lower-bound-local-convergence-s>1}
    \begin{aligned}
        \inf_{\T^{d}} \mu_{t}&\ge \big(\inf_{\T^{d}} \bar \mu\big)\, \exp\left(-C\int_{0}^{t}\lVert \nabla^{2}(-\Delta)^{-s}\sigma_{r}\rVert_{L^{\infty}}dr\right)\\
        &\ge \alpha \exp \left(-C(d,s,\gamma,M,\alpha)\lVert \sigma_{0}\rVert_{\dot H^{-s}}^{\frac{\eta}{2(\gamma+s)}}\right)
    \end{aligned}\qquad \forall t\in [0,\tau).
    \end{equation}
    
    We now conclude the proof by showing that if $\lVert \sigma_{0}\rVert_{\dot H^{-s}}^{2}\le \delta$ and $\delta>0$ is sufficiently small, then $\tau =T$. This, combined with \eqref{eq:polynomial-decay-H^-s-s>1}, proves that the estimates from \eqref{eq:conclusion-thm-conv-s>1} hold for all times $t\in [0,T)$, which in turn implies $T=\infty$, as already pointed out at the beginning of the proof. Suppose by contradiction that $\tau\in (0,T)$. Thanks to \eqref{eq:uniform-lower-bound-local-convergence-s>1}, by choosing $\delta$ small enough, we can make sure that $\inf_{\T^{d}} \mu_{t}\ge 3\alpha/4$ for every $t\in [0,\tau)$, so that necessarily $\lVert \sigma_{\tau}\rVert_{\dot H^{\gamma}}^{2}=2M$. However, \eqref{eq:energy-estimate-simplified-s>1}, together with \eqref{eq:polynomial-decay-H^-s-s>1} and \eqref{eq:integral-bound-Linfty-vector-field-s>1}, gives
    \begin{align*}
        \lVert \sigma_{\tau}\rVert_{\dot H^{\gamma}}^{2}&\le M \exp \left(C\int_{0}^{\tau}\lVert \nabla^{2}(-\Delta)^{-s}\sigma_{t}\rVert_{L^{\infty}}dt\right)\\&\quad\,+C\int_{0}^{\tau}\lVert \sigma_{t}\rVert_{\dot H^{-s}}^{2}\exp\left(\int_{t}^{\tau}C\lVert \nabla^{2}(-\Delta)^{-s}\sigma_{r}\rVert_{L^{\infty}}dr\right)\\
        &\le M\exp \left(C\lVert \sigma_{0}\rVert_{\dot H^{-s}}^{\frac{\eta}{2(\gamma+s)}}\right)\left(1+\frac{\lVert \sigma_{0}\rVert_{\dot H^{-s}}^{2}}{M}\int_{0}^{\infty}(1+Kt)^{-\frac{\gamma+s}{s-1}}dt\right)\le \frac{3}{2}M,
    \end{align*}
    provided that $\lVert \sigma_{0}\rVert_{\dot H^{-s}}\le \delta$ is chosen sufficiently small. This is a contradiction and concludes the proof. 
\end{proof}

\begin{rmk}\label{rmk:improved-regularity-target-convergence}
    The assumption $\nu \in H^{\gamma+s}(\T^d)$ in \Cref{thm:convergence-s>1} is merely technical and relies on the specific version of the Kato-Ponce commutator estimates that we have in \Cref{subsec:appendix-katoponce}. For instance, if we worked only with integer derivatives (i.e., $s, \gamma\in \N$, so that we could apply Leibniz product rule), we could weaken this assumption to $\nu \in H^{\gamma+1}(\T^d)$ (cf. the proof of \Cref{thm:convergence-arccos} in \Cref{subsec:quantitative-convergence-WFR} below). Furthermore, if we had a full Kato-Ponce estimate in the torus with fractional derivatives in the spirit of \cite[Theorem 1.2]{li2019kato}, one could further push the argument to assume only $\nu \in H^{\gamma+\eps}(\T^d)$ for some $\eps > 0$. \fr
\end{rmk}

\section{Convergence for continuous shallow neural networks}\label{sec:arccos}

In this section, we adapt the theory developed so far for Riesz KMD flows to the case of ReLU neural networks, within the framework of \Cref{subsec:arccos-intro}. We first explain how the Wasserstein gradient flow \eqref{eq:WGF-arccos} can be reduced, from the general formulation of \Cref{subsec:arccos-intro}, to the Wasserstein--Fisher--Rao dynamics \eqref{eq:PDE-Wasserstein-Fisher-Rao-sphere10} on the sphere. Next, in \Cref{subsec:analysis-sphere-arccos} we study the spectral properties of the positive semidefinite operator associated with the arccos kernel. In \Cref{subsec:well-posedness-WFR} we establish global well-posedness in the class of measures for \eqref{eq:PDE-Wasserstein-Fisher-Rao-sphere10}. Finally, in \Cref{subsec:quantitative-convergence-WFR} we prove the local quantitative convergence statement of \Cref{thm:convergence-arccos}.

\subsection{Reduction to a Wasserstein--Fisher--Rao flow}
\label{sec:WFR}

Let us show how to exploit the symmetries of the problem, up to a small loss of expressivity power, to  reduce \eqref{eq:WGF-arccos} to equivalent dynamics on $\Sp^{d}$ sharing many analogies with \eqref{eq:PDE-GF}. We explain this reduction in a few steps (continuing the discussion initiated in \Cref{subsec:arccos-intro}, and using the notation introduced there):

\smallskip
\noindent \textbf{From $\R^{d+2}$ to $\sqrt{2}\Sp^{d+1}$.} Leveraging the $2$-homogeneity of $\Phi(w,x)$ in the variable $w$, we may equivalently represent the function $f_{\mu}$ in \eqref{eq:fmu_intro} in terms of the measure $\Tilde{\mu}\in \mathcal{M}_{+}(\sqrt{2}\Sp^{d+1})$ obtained by projecting $\mu$ on $\sqrt{2}\Sp^{d+1}$ with quadratic weight on the radial variable: 
\begin{equation*}
\begin{gathered}
    f_{\mu}(x)=\int_{\R^{d+2}}\Phi(w, x)d\mu(w)=\int_{\sqrt{2}\Sp^{d+1}}\Phi(\theta,x)d\Tilde{\mu}(\theta)=:f_{\tilde{\mu}}(x),\\
    \int_{\sqrt{2}\Sp^{d+1}}\psi d\Tilde{\mu}:=\int_{\R^{d+2}}\frac{|w|^{2}}{2}\psi\left(\sqrt{2}\frac{w}{|w|}\right)d\mu\qquad \forall \psi\in C(\sqrt{2}\Sp^{d}).
\end{gathered}
\end{equation*}
Replacing $\mu_{t}, \nu$, and $\bar\mu$ with the respective projections on $\sqrt{2}\Sp^{d+1}$, \eqref{eq:WGF-arccos} reads as
\begin{equation}\label{eq:WFR-arccos}
    \left\{
    \begin{array}{rclll}
        \partial_{t}\mu_{t}+\diver_{\sqrt{2}\Sp^{d+1}}\left(\mu_{t} v_{t}\right)&=&-4\hat{\mathcal{K}}(\mu_{t}-\nu)\mu_{t}\quad &\text{in $(0,T)\times \sqrt{2}\Sp^{d+1}$},\\
           v_{t}&=&-\nabla_{\sqrt{2}\Sp^{d+1}} \hat{\mathcal{K}}(\mu_{t}-\nu)\quad &\forall t\in (0,T),\\
           \mu_{0}&=&\bar\mu,
    \end{array}
    \right.
\end{equation}

\smallskip
\noindent \textbf{From $\sqrt 2 \Sp^{d+1}$ to $\{\pm 1\}\times \Sp^{d}$.} In order to make a further reduction, we observe that any function $f_{\mu}$, $\mu\in \mathcal{M}_{+}(\sqrt{2}\Sp^{d+1})$ can be represented equivalently as $f_{\mu'}$ for some $\mu'\in \mathcal{M}_{+}(\sqrt{2}\Sp^{d+1})$ with $\supp \mu' \subset \{w=(a,b)\in \R\times \R^{d+1}:|a|=|b|\}$. This is a consequence of the $1$-homogeneity of $\Phi(w,x)$, separately in the two variables $a\in \R$ and $b\in \R^{d+1}$. Furthermore, the dynamics \eqref{eq:WFR-arccos} is closed in the set of measures supported in the cone $\{|a|=|b|\}$, as can be checked by showing that the vector field $v_{t}$ is tangential to $\{|a|=|b|\}$ whenever $\mu_{t}$ and $\nu$ are supported there.
For these reasons, we may reduce to the case in which $\bar\mu, \nu \in \mathcal{M}_{+}(\sqrt{2}\Sp^{d+1})$ are nonnegative finite measures supported in $\{|a|=|b|\}$, so that the solution $\mu_{t}\in \mathcal{M}_{+}(\sqrt{2}\Sp^{d+1})$ of \eqref{eq:WFR-arccos} will also be supported in $\{|a|=|b|\}$. Observe that the intersection of $\sqrt{2}\Sp^{d+1}$ with the cone $\{|a|=|b|\}$ consists of two disjoint copies of $\Sp^{d}$:
\begin{equation*}
    \sqrt{2}\Sp^{d+1}\cap \{|a|=|b|\}=\{\pm 1\}\times \Sp^{d}.
\end{equation*}
Identifying $\mu_{t}$ with a couple of measures $(\mu_{t}^{+},\mu_{t}^{-})\in \mathcal{M}_{+}(\Sp^{d})^{2}$, and similarly for $\bar\mu$ and $\nu$, we may write \eqref{eq:WFR-arccos} as a nonlocal forced continuity equation of two species on the sphere $\Sp^{d}$:
\begin{equation}\label{eq:WFR-two-species}
    \left\{
    \begin{array}{rclll}
        \partial_{t}\mu_{t}^{+}+\diver_{\Sp^{d}}\left(\mu_{t}^{+} v_{t}\right)&=&-4\mathcal{K}\left((\mu_{t}^{+}-\mu_{t}^{-})-(\nu^{+}-\nu^{-})\right)\mu_{t}^{+}\quad &\text{in $(0,T)\times \Sp^{d}$},\\
        \partial_{t}\mu_{t}^{-}-\diver_{\Sp^{d}}\left(\mu_{t}^{-} v_{t}\right)&=&+4\mathcal{K}\left((\mu_{t}^{+}-\mu_{t}^{-})-(\nu^{+}-\nu^{-})\right)\mu_{t}^{-}\quad &\text{in $(0,T)\times \Sp^{d}$},\\
           v_{t}&=&-\nabla_{\Sp^{d}}\mathcal{K}\left((\mu_{t}^{+}-\mu_{t}^{-})-(\nu^{+}-\nu^{-})\right) \quad &\forall t\in (0,T),\\
           \mu_{0}^{+}&=&\bar\mu^{+},\\
           \mu_{0}^{-}&=&\bar\mu^{-},
    \end{array}
    \right.
\end{equation}
where $\mathcal{K}$ is the operator defined in \eqref{eq:def-arccos-operator+kernel}.

\smallskip
\noindent \textbf{From $\{\pm 1\}\times \Sp^{d}$ to $\Sp^{d}$.} Finally, we show that we can take $\bar\mu^{-}=\nu^{-}=0$ (and consequently $\mu_{t}^{-}=0$) in \eqref{eq:WFR-two-species}, thus reducing to the evolution of a single species, provided that we assume sufficient regularity on the target function $f_{\nu}$. Under the identification $\nu=(\nu^{+},\nu^{-})\in \mathcal{M}_{+}(\Sp^{d})^{2}$ from the previous step, we have $f_{\nu}=\mathcal{H}(\nu^{+}-\nu^{-})$, where
\begin{equation}\label{eq:def-repres-operator-H}
    \mathcal{H}(\eta)(x):=\int_{\Sp^{d}}(x\cdot y)_{+}d\eta(y)\qquad \forall \eta\in \mathcal{M}(\Sp^{d}).
\end{equation}
With some explicit computations in terms of spherical harmonics expansions, one can show the following (see \Cref{lem:ReLU-operator}):
\begin{itemize}
    \item Linear functions are the only odd functions on the sphere that can be written as $\mathcal{H}(\eta)$, for some $\eta\in \mathcal{M}(\Sp^{d})$. In particular, all possible functions that we can represent as $f_{\nu}$ are even on the sphere, up to the addition of some linear function.  
    \item  For all $\gamma\ge 0$, $\mathcal{H}$ defines a linear bijective continuous operator from $H^{\gamma}_{\even}(\Sp^{d})$ to $H^{\gamma+s}_{\even}(\Sp^{d})$, where $s=(d+3)/2$. As a consequence, taking $\gamma>d/2$ and using the Sobolev embedding $H^{\gamma}(\Sp^{d})\hookrightarrow L^{\infty}(\Sp^{d})$ we deduce the following: all functions $f\in H^{\gamma+s}_{\even}(\Sp^{d})$ can be written as $f=f_{\nu}+C_{f}$, for some $\nu=(\nu^{+},0)$, $\nu^{+}$ even, and some constant $C_{f}$ such that $|C_{f}|\lesssim_{d,\gamma} \lVert f\rVert_{H^{\gamma+s}}$.
\end{itemize}
In view of the observations above, we may take $\bar\mu^{-}=\nu^{-}=0$ and $\bar\mu^{+},\nu^{+}$ even. This is done at the expense of a small loss of expressivity power: we can only represent sufficiently regular even functions on the sphere, up to the addition of a controlled constant. Identifying $\bar\mu, \nu$, and $\mu_{t}$ with $\bar\mu^{+},\nu^{+}$, and $\mu_{t}^{+}$, respectively, \eqref{eq:WFR-two-species} becomes 
\begin{equation*}
    \left\{
    \begin{array}{rclll}
        \partial_{t}\mu_{t}+\diver_{\Sp^{d}}\left(\mu_{t}v_{t}\right)&=&-4\mathcal{K}(\mu_{t}-\nu)\mu_{t}\qquad &\text{in $(0,T)\times \Sp^{d}$},\\
           v_{t}&=&-\nabla_{\Sp^{d}} \mathcal{K}(\mu_{t}-\nu)\qquad &\forall t\in (0,T),\\
           \mu_{0}&=&\bar\mu,
    \end{array}
    \right.
\end{equation*}
as we wanted to show.

\subsection{Analysis on the sphere and the arccos kernel operator}\label{subsec:analysis-sphere-arccos}
In this section we recall some notions of analysis on the $d$-dimensional unit sphere $\Sp^{d}\subset \R^{d+1}$. Next, we consider the arccos kernel $K:\Sp^{d}\times \Sp^{d}\to \R$ and the corresponding convolution operator $\mathcal{K}$ defined in \eqref{eq:def-arccos-operator+kernel}, obtaining the precise regularization properties of the latter. We refer the reader to \cite[Appendix D]{bach2017breaking} and the book \cite{dai2013approximation} for a broader introduction to the topic. 
\subsubsection{Spherical harmonics and multiplier operators }
We consider an orthonormal basis of $L^{2}(\Sp^{d})$ made of eigenfunctions for the Laplace-Beltrami operator $-\Delta_{\Sp^{d}}$:  
    \begin{equation*}
    \begin{gathered}
    \{Y_{kj}: k\ge 0, \,\,1\le j\le N(d,k)\},\qquad -\Delta_{\Sp^{d}}Y_{kj}=\mu_{k}Y_{kj},\quad \mu_{k}:=k(k+d-1),\\[5pt]
       N(d,0)=1, \qquad N(d,k)= \frac{2k+d-1}{k}\binom{k+d-2}{d-1}\quad \forall k\ge 1.
    \end{gathered}
    \end{equation*}
    For every $k\ge 0$ and $1\le j\le N(d,k)$, $Y_{kj}$ is the restriction to the sphere of a $k$-homogeneous harmonic polynomial in $\R^{d+1}$. For this reason, $Y_{kj}$ are usually called \emph{spherical harmonics}. 
  
For every distribution $f\in \mathscr{D}'(\Sp^{d})$, every $k\ge 0$ and $1\le j\le N(d,k)$, we define $\hat{f}_{kj}\in \R$ and $f_{k}\in C^{\infty}(\Sp^{d})$ respectively as
  \begin{equation*}
      \hat{f}_{kj}:=\langle f, Y_{kj}\rangle\quad 1\le j\le N(d,k), \qquad f_{k}:= \sum_{j=1}^{N(d,k)}\hat{f}_{kj}Y_{kj}\qquad \forall k\ge 0.
  \end{equation*}
  Every function $f\in L^{2}(\Sp^{d})$ admits an expansion in terms of the orthonormal basis of spherical harmonics:
    \begin{equation*}
        f=\overunderset{\infty}{k=0}{\sum}\overunderset{N(d,k)}{j=1}{\sum}\hat{f}_{kj}Y_{kj}=\sum_{k=0}^{\infty}f_{k},\qquad \lVert f\rVert_{L^{2}(\Sp^{d})}^{2}=\sum_{k=0}^{\infty}\sum_{j=1}^{N(d,k)}|\hat{f}_{kj}|^{2}=\sum_{k=0}^{\infty}\lVert f_{k}\rVert_{L^{2}(\Sp^{d})}^{2}.
    \end{equation*}

 For $\gamma\in \R$, the homogeneous $\gamma$-Sobolev seminorm is defined as
 \begin{equation*}
     \lVert f\rVert_{\dot H^{\gamma}(\Sp^{d})}^{2}:= \sum_{k=1}^{\infty}\mu_{k}^{\gamma}\lVert f_{k}\rVert_{L^{2}(\Sp^{d})}^{2}\qquad \forall f\in \mathscr{D}'(\Sp^{d}),
 \end{equation*}
 where $\mu_{k}=k(k+d-1)$ is the $k$-th eigenvalue of the Laplace-Beltrami operator $-\Delta_{\Sp^{d}}$. The corresponding homogeneous Sobolev space $\dot H^{\gamma}(\Sp^{d})$ is obtained as follows:
\begin{equation*}
    \dot H^{\gamma}(\Sp^{d}):=\left\{f\in \mathscr{D}'(\Sp^{d}):f_{0}=0,\, \lVert f\rVert_{\dot H^{\gamma}(\Sp^{d})}^{2}<\infty\right\}.
\end{equation*}
 When $\gamma\ge 0$, we also introduce the inhomogeneous Sobolev space $H^{\gamma}(\Sp^{d})$:
 \begin{equation*}
     H^{\gamma}(\Sp^{d}):= \left\{f\in L^{2}(\Sp^{d}): \lVert f\rVert_{H^{\gamma}(\Sp^{d})}:= \lVert f\rVert_{L^{2}(\Sp^{d})}+\lVert f\rVert_{\dot H^{\gamma}(\Sp^{d})}<\infty\right\}.
 \end{equation*} 
 
 We say that a linear (possibly unbounded) operator $T$ on $L^{2}(\Sp^{d})$ is a multiplier if there is a sequence $\{b_{k}\}_{k\ge 0}\subset \R$ such that $$(Tf)_{k}= b_{k}f_{k}\qquad \forall k\ge 0.$$
  Typical examples of multiplier operators are powers of the Laplace-Beltrami operator $(-\Delta_{\Sp^{d}})^{\gamma}$, where $\gamma\in \R$, for which $b_{k}=\mu_{k}^{\gamma}$.
Notice that $f\in L^{2}(\Sp^{d})$ is even (resp. odd)\footnote{Here symmetry is considered with respect to the origin. In particular $f$ is even is $f(x)=f(-x)$, and odd if $f(x)=-f(-x)$.} if and only if $f_{k}=0$ for all odd (resp. even) $k\ge 0$. In particular, a multiplier operator maps even (resp. odd) functions to even (resp. odd) functions.

\subsubsection{Sobolev spaces on the sphere in terms of angular derivatives}\label{subsec:angular-derivatives}
For every $1\le i<j\le d+1$ the angular derivative in the coordinates $i,j$ is defined as follows:
    \begin{equation*}
        \nabla_{(i,j)}:= x_{i}\partial_{i}-x_{j}\partial_{j}=\partial_{\theta_{ij}},\qquad (x_{i},x_{j})=r_{ij}(\cos\theta_{ij},\sin\theta_{ij}).
    \end{equation*}
    Some useful properties of angular derivatives $\nabla_{(i,j)}$ are listed below (see \cite[Chapters 1-3]{dai2013approximation} for the proofs):

    \begin{itemize}
        \item Let $\nabla_{\Sp^{d}}$ and $-\Delta_{\Sp^{d}}$ be the tangential gradient and the Laplace-Beltrami operator in $\Sp^{d}$, respectively. Then, the following identities hold for sufficiently regular functions $f,g:\Sp^{d}\to \R$:
        \begin{equation*}
        \begin{gathered}
             (\nabla_{\Sp^{d}} f)_{j}(x)=\sum_{\substack{1\le i\le d+1 \\ i\neq j}}x_{i}\nabla_{(i,j)}f(x),\qquad \nabla_{\Sp^{d}} f\cdot \nabla_{\Sp^{d}} g= \sum_{1\le i<j\le d+1}\nabla_{(i,j)}f\nabla_{(i,j)}g,\\
             \Delta_{\Sp^{d}} f= \sum_{1\le i<j\le d+1}\nabla_{(i,j)}^{2}f.
        \end{gathered}
        \end{equation*}
        \item Angular derivatives $\nabla_{(i,j)}$ leave the space of spherical harmonics of order $k$ invariant, for all $k\ge 0$. In particular, $\nabla_{(i,j)}$ commutes with any multiplier operator. 
        \item The Leibniz rule holds: $\nabla_{(i,j)}(fg)=\nabla_{(i,j)}f g+f\nabla_{(i,j)}g$.
        \item Integration by parts holds: $\int_{\Sp^{d}}f\nabla_{(i,j)}g=-\int_{\Sp^{d}}\nabla_{(i,j)}f g$.
    \end{itemize}

\medskip
Given a multi-index $\beta=\{\beta_{(i,j)}:1\le i<j\le d+1\}$ with $\beta_{(i,j)}\in \N\cup \{0\}$, we write $|\beta|=\sum_{1\le i<j\le d+1}\beta_{(i,j)}$, and we define the corresponding higher order angular derivative $\nabla_{\beta}$ as the composition
    \begin{equation*}
        \nabla_{\beta}:=\underset{1\le i<j\le d+1}{\bigcirc}\nabla_{(i,j)}^{\beta_{(i,j)}}.
    \end{equation*}
    For every smooth function $f\in C^{\infty}(\Sp^{d})$, $\gamma\in \N\cup \{0\}$, and $p\in [1,\infty)$, angular Sobolev norms are defined as follows: 
    \begin{equation*}
        \lVert f\rVert_{\dot W^{\gamma,p}}:= \left(\sum_{|\beta|=\gamma}\lVert \nabla_{\beta}f\rVert_{L^{p}}^{p}\right)^{1/p},\qquad \lVert f\rVert_{W^{\gamma,p}}:= \sum_{k=0}^{\gamma}\lVert f\rVert_{\dot W^{k,p}},
    \end{equation*}
    with the classical meaning of supremum norm in the case $p=\infty$.
    The space $W^{\gamma,p}(\Sp^{d})$ is defined as the closure of $C^{\infty}(\Sp^{d})$ with respect to $\lVert \cdot\rVert_{W^{\gamma,p}}$.

    As proved in \cite[Section 3.5.1]{dai2013approximation}, for every $1\le i<j\le d+1$, the Riesz transform $R_{ij}:=\nabla_{(i,j)}(-\Delta_{\Sp^{d}})^{-1/2}$ is bounded in $L^{p}$ for $1<p<\infty$.
       More precisely, we have
        \begin{equation}\label{eq:comparability-homogeneous-angular}
           \max_{1\le i<j\le d+1}\lVert \nabla_{(i,j)}f\rVert_{L^{p}}\approx_{d,p}  \lVert (-\Delta_{\Sp^{d}})^{1/2}f\rVert_{L^{p}}\qquad \forall f\in W^{1,p}(\Sp^{d}),\quad \forall p\in (1,\infty).
        \end{equation}
       An iterative application of \eqref{eq:comparability-homogeneous-angular} gives
    \begin{equation*}
        \lVert f\rVert_{\dot W^{\gamma,p}}\approx_{d,p,\gamma}\lVert (-\Delta_{\Sp^{d}})^{\gamma/2}f\rVert_{L^{p}} \qquad \forall f\in W^{\gamma,p}(\Sp^{d}),\quad \forall\gamma\in \N\cup \{0\},\quad \forall p\in (1,\infty).
    \end{equation*}
    In the case $p=2$, we have $H^{\gamma}(\Sp^{d})=W^{\gamma,2}(\Sp^{d})$. Moreover, homogeneous and angular Sobolev norms coincide:
    \begin{equation*}
       \lVert f\rVert_{\dot W^{\gamma,2}}^{2}=\sum_{|\beta|=\gamma} \lVert \nabla_{\beta}f\rVert_{L^{2}}^{2}=\int_{\Sp^{d}}f(-\Delta_{\Sp^{d}})^{\gamma}f=\lVert f\rVert_{\dot H^{\gamma}}^{2}\qquad \forall f\in W^{\gamma,2}(\Sp^{d}),\quad \forall \gamma \in \N\cup \{0\}.
    \end{equation*}

\subsubsection{ReLU activation function and arccos kernel on the sphere}
In this section, we study the spectral behavior of the representation operator $\mathcal{H}$ associated with the ReLU activation function in \eqref{eq:def-repres-operator-H}. We deduce that the positive semidefinite operator $\mathcal{K}=\mathcal{H}^{2}$ in \eqref{eq:def-arccos-operator+kernel} is comparable to $(-\Delta_{\Sp^{d}})^{-\frac{d+3}{2}}$, at least when acting on positive even frequencies. Later, in \Cref{subsec:quantitative-convergence-WFR} we will exploit this spectral analysis to prove the local polynomial quantitative convergence result for ReLU shallow neural networks (see \Cref{thm:convergence-arccos}). 

In the following lemma, we report the spectral analysis on $\mathcal{H}$ in spherical harmonics, as derived in \cite{bach2017breaking}:

        \newlength{\nl}
\setlength{\nl}{34mm}

\begin{lem}[\protect{\cite[Appendix C.1 \& D.2]{bach2017breaking}}]\label{lem:ReLU-operator}
    Let $\mathcal{H}$ be the operator defined in \eqref{eq:def-repres-operator-H}, and let $s=(d+3)/2$. Then $\mathcal{H}:L^{2}(\Sp^{d})\to L^{2}(\Sp^{d})$ is a multiplier operator such that $(\mathcal{H}f)_{k}=\beta_{k}f_{k}$, where 

\[
{\renewcommand{\arraystretch}{1.05}%
\begin{aligned}
 (d=1) &&\quad \beta_k &:=
\begin{cases}
2, & \hspace{\nl} k=0,\\
\frac{\pi}{2}, &\hspace{\nl} k=1,\\
(-1)^{\frac{k}{2}-1}\frac{2}{k^{2}-1}, & \hspace{\nl}k\in 2\N,\\
0, & \hspace{\nl}k\in 2\N+1,
\end{cases}
\\[10pt]
 (d\ge 2)&&\quad \beta_k &:=
\begin{cases}
\frac{|\Sp^{d}|(d-1)}{2\pi d}, & k=0,\\
\frac{|\Sp^{d}|(d-1)\Gamma(\frac{d}{2})}{4\sqrt{\pi}\,d\,\Gamma(\frac{d+3}{2})}, & k=1,\\
(-1)^{\frac{k}{2}-1}\frac{|\Sp^{d}|(d-1)\Gamma(\frac{d}{2})}{2\pi}
\frac{\Gamma(k-1)}{2^{k}\Gamma(\frac{k}{2})\Gamma(\frac{k+d}{2}+1)}, & k\in 2\N,\\
0, & k\in 2\N+1.
\end{cases}
\end{aligned}}
\]
         Moreover, we have $|\beta_{k}|\approx_{d} k^{-s}$ for all $k\in 2\N$. Thus, for all $\gamma\ge 0$, $\mathcal{H}$ defines a linear continuous bijective operator from $H^{\gamma}_{\even}(\Sp^{d})$ to $H^{\gamma+s}_{\even}(\Sp^{d})$.  
\end{lem}

\begin{rmk}
The previous statement could also be obtained from the Goodey--Weil  identity  relating the cosine transform $\mathcal{C}$ with the (spherical) Radon transform $\mathcal{R}$ on the sphere\footnote{For any $f\in C^\infty(\Sph^{d})$, the cosine and (spherical) Radon transforms are given, respectively, by 
\[
(\mathcal C f ) (x) = \int_{\Sph^{d}} |x \cdot y| f(y) d y\qquad \text{and}\qquad (\mathcal R f ) (x) = \int_{\Sph^{d}\cap x^\perp}  f(v) d\sigma_{u^\perp} (v), \qquad\text{for}\quad x\in \Sph^{d},
\]
where $dy$ is the $d$-dimensional surface measure on $\Sph^{d}$, and  $d\sigma_{u^\perp} (v)$ denotes the $(d-1)$-dimensional surface measure on $u^\perp\cap \Sph^{d}\cong \Sph^{d-1}$. 
} $\Sph^{d}$ (see \cite[Prop. 2.1]{goodey1992centrally}:
\[
\mathcal{C}^{-1} = \frac{1}{2} (\Delta_{\Sph^{d} }+ d )\mathcal{R}^{-1}, 
\]
This is because $\mathcal{H} = \tfrac12 \mathcal{C}$ when acting on even distributions. \fr
\end{rmk}
 
Next, we consider the symmetric kernel $K:\Sp^{d}\times \Sp^{d}\to \R$ and its corresponding convolution operator $\mathcal{K}$, defined in \eqref{eq:def-arccos-operator+kernel}.
$K$ is often called an \lq\lq arccos-type kernel'', for it can be written as a function of the geodesic distance in $\Sp^{d}$ between $x$ and $y$,
$$\theta(x,y):= \arccos(x\cdot y)= \dist_{\mathbb{S}^{d}}(x,y)\in [0,\pi],\qquad x,y\in \mathbb{S}^{d}.$$
In fact, as shown for instance in \cite{cho2009kernel, cho2011analysis}, there is a positive constant $c=c(d)>0$ such that 
\begin{equation*}
    K(x,y)= J(\theta(x,y))\quad \forall x,y \in \Sp^{d},\qquad J(\theta):= c\left(\sin \theta +(\pi-\theta)\cos \theta\right) \quad \forall \theta\in [0,\pi].
\end{equation*}
The function $J:[0,\pi]\to \R$ is smooth, nonnegative, strictly decreasing, and has the following asymptotics as $\theta\to0^{+}, \pi^{-}$:
\begin{equation*}
    J(\theta)=c\left(\pi-\frac{\pi}{2}\theta^{2}+\frac{1}{3}\theta^{3}\right)+o(\theta^{3})\quad \text{as $\theta\to 0^{+}$},\qquad J(\theta)= \frac{c}{3}(\pi-\theta)^{3}+o((\pi-\theta)^{3})\quad \text{as $\theta\to \pi^{-}$}.
\end{equation*}
As a consequence, working in normal coordinates, one can show that $K\in C^{2,1}(\Sp^{d}\times \Sp^{d})$.
In particular, the following estimate holds for the convolution operator $\mathcal{K}$ in \eqref{eq:def-arccos-operator+kernel}:   
\begin{equation}\label{eq:C21-bound-arccos-operator}
    \lVert \mathcal{K\eta}\rVert_{C^{2,1}(\Sp^{d})}\lesssim_{d} \lVert \eta\rVert_{\mathcal{M}(\Sp^{d})}\qquad \forall \eta\in \mathcal{M}(\Sp^{d}).
\end{equation}

Leveraging the particular structure of the kernel $K$ in \eqref{eq:def-arccos-operator+kernel}, one can check that $\mathcal{K}=\mathcal{H}^{2}$. This makes $\mathcal{K}$ a positive semidefinite multiplier operator in spherical harmonics, whose spectral behavior is comparable, at least when restricted to positive even frequencies, with that of $(-\Delta_{\Sp^{d}})^{-\frac{d+3}{2}}$ (see \cite{bietti2020deep} for some related results): 

\begin{lem}\label{lem:spectral-behavior-arccos-operator}
Let $\mathcal{K}$ be the operator defined in \eqref{eq:def-arccos-operator+kernel}, and let $s:=\frac{d+3}{2}$. The following hold:
\begin{itemize}
    \item [i)] $\mathcal{K}:L^{2}(\Sp^{d})\to L^{2}(\Sp^{d})$ is a multiplier operator such that $(\mathcal{K}f)_{k}=\lambda_{k}f_{k}$, where $\lambda_{k}=\beta_{k}^{2}$ and $\beta_{k}$ are given by \Cref{lem:ReLU-operator}.
    In particular, $\lambda_{k}\approx_{d} k^{-2s}$ for all $k\in 2\N$, and the approximate identity for $\mathscr{E}^{\nu}_{\Sp^{d}}$ in \eqref{eq:comparability-energy-arccos-H^-s} holds.
    \item [ii)] The multiplier operator $T:= (-\Delta_{\Sp^{d}})^{s}\mathcal{K}:L^{2}(\Sp^{d})\to L^{2}(\Sp^{d})$ is invertible when restricting domain and image to the space of zero-mean even functions $L^{2}_{2\N}(\Sp^{d})$, and 
    \begin{equation*}
        \lVert Tf\rVert_{L^{p}}\approx_{d,p} \lVert f\rVert_{L^{p}}\qquad \forall f\in L^{p}\cap L^{2}_{2\N}(\Sp^{d}),\quad \forall p\in (1,\infty).
    \end{equation*}
    \item [iii)] There is a constant $\bar c=\bar c(d)>0$ such that 
    \begin{equation*}
        \lVert\left(\mathcal{K}-\bar c(-\Delta_{\Sp^{d}})^{-s}\right)(f)\rVert_{L^{2}}\lesssim_{d}\lVert f\rVert_{\dot H^{-2s-1}}\qquad \forall f\in L^{2}_{2\N}.
    \end{equation*}
\end{itemize}
\end{lem}
\begin{proof} Point i) follows from \Cref{lem:ReLU-operator} and the fact that $\mathcal{K}=\mathcal{H}^{2}$. 

Let us prove point ii). Being $\mu_{k}=k(k+d-1)$ the eigenvalues of $-\Delta_{\Sp^{d}}$, the composition $T=(-\Delta_{\Sp^{d}})^{s}\mathcal{K}$ will have symbol
\begin{equation*}
    (T f)_{k}=c_{k}f_{k},\qquad c_{k}=(k(k+d-1))^{s}\lambda_{k}\qquad \forall k\ge 0,
\end{equation*}
and $c_{k}>0$ for all $k\in 2\N$. We first consider the case $d\ge 2$. Let 
$g:[2,\infty)\to \R$ be the function
    $$g(z)= \frac{|\Sp^{d}|^{2}(d-1)^{2}\Gamma\left(\frac{d}{2}\right)^{2}}{4\pi^{2}}\frac{(z(z+d-1))^{\frac{d+3}{2}}\Gamma(z-1)^{2}}{2^{2z}\Gamma\left(\frac{z}{2}\right)^{2}\Gamma\left(\frac{z+d}{2}+1\right)^{2}}$$
    such that $g(k)=c_{k}$ for all $k\in 2\N$. Thanks to \cite[Theorem 3.3.1]{dai2013approximation}, to prove the content of point ii), all we need to show is that $g(z)$ has a positive limit as $z\to \infty$, and the following
    Mikhlin-type condition holds:
    \begin{equation}\label{eq:Mikhlin-condition}
        \left|\frac{d^{j}}{dz^{j}}g(z)\right|\lesssim z^{-j}\qquad \forall z\in [2,\infty),\quad \forall j\in\left\{0,\dots, \left\lceil \frac{d+1}{2}\right\rceil\right\}.
    \end{equation}
    To show that, one can use Stirling's approximation formula in the following form: 
    $$\Gamma(x)= \sqrt{2\pi}\, x^{x-\frac{1}{2}}e^{-x}e^{r(x)},\qquad r(x)= \overunderset{\infty}{n=2}{\sum}\frac{a_{n}}{(x+1)\cdots (x+n-1)}.$$
    Here the coefficients $a_{n}\in \R$ are such that the corresponding series is convergent for all $x>0$, which makes $r(x)$ an analytic remainder which vanishes as $x\to \infty$, and satisfies the Mikhlin condition \eqref{eq:Mikhlin-condition}. Expanding $g(z)$ as $z\to \infty$ using Stirling's formula above we find
    \begin{equation*}
        g(z)=\bar c+O(z^{-1}),\qquad \bar c:= \frac{2^{d-3}(d-1)^{2}|\Sp^{d}|^{2}\Gamma\left(\frac{d}{2}\right)^{2}}{\pi^{3}}>0.
    \end{equation*}
    Moreover, one may rewrite $g$ as
    \begin{align*}
        g(z)&= C(d)\frac{z^{\frac{d+5}{2}}(z+d-1)^{\frac{d+3}{2}}}{(z-1)^{3}(z+d+2)^{d+1}}\times\\
        &\quad \exp\Big[-z\left(\log\left(1+\frac{1}{z-1}\right)+\log\left(1+\frac{d+3}{z-1}\right)\right)\Big]\times\\
        &\quad \exp\Big[2\left(r(z-1)-r\left(\frac{z}{2}\right)-r\left(\frac{z+d}{2}+1\right)\right)\Big].
    \end{align*}
    Since each of the three factors in the right-hand side above satisfy the Mikhlin condition \eqref{eq:Mikhlin-condition}, then their product does. In the case $d=1$, the same argument using $g(z)=4z^{4}/(z^{2}-1)^{2}$. In this case $g(z)\to \bar c:= 4$ as $z\to \infty$.
    
    Finally, we prove point iii). Let $\bar c=\bar{c}(d)>0$ be the limit of $c_{k}$ obtained above. We have
    \begin{equation*}
        \lambda_{k}-\bar{c}(k(k+d-1))^{-s}= (g(k)-\bar{c})(k(k+d-1))^{-s}\lesssim_{d} k^{-2s-1} \qquad \forall k\in 2\N,
    \end{equation*}
    and finally, for every $f\in L^{2}_{2\N}$,
    \begin{equation*}
        \lVert (\mathcal{K}-\bar{c}(-\Delta_{\Sp^{d}})^{-s})(f)\rVert_{L^{2}}^{2}= \sum_{k\in 2\N}(\lambda_{k}-\bar{c}(k(k+d-1))^{-s})^{2}\lVert f_{k}\rVert_{L^{2}}^{2}\lesssim_{d} \sum_{k\in 2\N}|k|^{-4s-2}\lVert f_{k}\rVert_{L^{2}}^{2}\lesssim_{d} \lVert f\rVert_{\dot H^{-2s-1}}^{2},
    \end{equation*}
    as desired. 
\end{proof}

\subsection{Well-posedness of the Wasserstein--Fisher--Rao dynamics}\label{subsec:well-posedness-WFR}
In this section, we study well-posedness for the non-conservative active-scalar equation \eqref{eq:PDE-Wasserstein-Fisher-Rao-sphere10},
where $\bar\mu, \nu \in \mathcal{M}_{+}(\Sp^{d})$ are two given nonnegative measures, and $\mathcal{K}$ is the convolution operator defined in \eqref{eq:def-arccos-operator+kernel}. The following is the precise notion of weak solution we consider:
\begin{defn}\label{def:solution-active-scalar-WFR}
     Let $\bar\mu, \nu \in \mathcal{M}_{+}(\Sp^{d})$.
We say that a curve of nonnegative measures $\mu \in C_{w^{*}}([0,T); \mathcal{M}_{+}(\Sp^{d}))$ solves \eqref{eq:PDE-Wasserstein-Fisher-Rao-sphere10}
if $\mu_{0}=\bar\mu$, and denoting $v_{t}:= -\nabla_{\Sp^{d}} \mathcal{K}(\mu_{t}-\nu)$, the forced continuity equation $\partial_{t}\mu_{t}+\diver_{\Sp^{d}}(\mu_{t} v_{t})=-4\mathcal{K}(\mu_{t}-\nu)\mu_{t}$ is solved in the distributional sense, equivalently
\begin{equation*}
    \int_{\Sp^{d}}\varphi d\mu_{t}= \int_{\Sp^{d}}\varphi d\bar\mu +\int_{0}^{t}\int_{\Sp^{d}}\nabla_{\Sp^{d}} \varphi \cdot v_{r}d\mu_{r}dr-4\int_{0}^{t}\int_{\Sp^{d}}\varphi\mathcal{K}(\mu_{r}-\nu) d\mu_{r}dr\qquad \forall t\in (0,T),\quad \forall \varphi \in C^{\infty}(\Sph^{d}).
\end{equation*}
\end{defn}

A solution $\mu$ of \eqref{eq:PDE-Wasserstein-Fisher-Rao-sphere10} according to \Cref{def:solution-active-scalar-WFR} is locally bounded in the space of measures. Therefore, by \eqref{eq:C21-bound-arccos-operator}, it generates a velocity field $v\in L^{\infty}_{\loc}([0,T); C^{1,1}(\Sp^{d};\R^{d+1}))$. In particular, the corresponding flow map $X:[0,T)\times \Sp^{d}\to \Sp^{d}$ is well-defined and $C^{1,1}$-regular by the standard Cauchy--Lipschitz theory.

Using $\varphi\equiv 1$ as a test function in the distributional formulation of \eqref{eq:PDE-Wasserstein-Fisher-Rao-sphere10} above, we find 
\begin{align*}
    \frac{d}{dt}\mu_{t}(\Sp^{d})&=-4\int_{\Sp^{d}}\mathcal{K}(\mu_{t}-\nu)d\mu_{t}\\
    &=-4\int_{\Sp^{d}}\mathcal{K}(\mu_{t})d\mu_{t}+4\int_{\Sp^{d}}\mathcal{K}(\nu)d\mu_{t}\le -4\lambda_{0}\mu_{t}(\Sp^{d})^{2}+C(d)\nu(\Sp^{d})\mu_{t}(\Sp^{d}),
\end{align*}
where $\lambda_{0}>0$ is the first eigenvalue of $\mathcal{K}$ from \Cref{lem:spectral-behavior-arccos-operator}, and we have used \eqref{eq:C21-bound-arccos-operator}. From this differential inequality we deduce the following uniform upper bound on the mass of a solution along the dynamics:  
\begin{equation}\label{eq:uniform-bound-mass-WFR}
    \mu_{t}(\Sp^{d})\le \max \{\bar\mu(\Sp^{d}), C(d)\lambda_{0}^{-1}\nu(\Sp^{d})\}\qquad \forall t\in [0,T).
\end{equation}

Before proceeding with the well-posedness result for the Wasserstein--Fisher--Rao dynamics, let us recall the definition of the so called \textit{Bounded-Lipschitz distance} $d_{BL}:\mathcal{M}(\Sp^{d})\times \mathcal{M}(\Sp^{d})\to [0,\infty)$ between finite measures:
\begin{equation*}
    d_{BL}(\mu_{1}, \mu_{2}):= \sup\left\{\int_{\Sp^{d}}f(d\mu_{1}-d\mu_{2}): \lVert f\rVert_{L^{\infty}}\le 1, \Lip(f)\le 1\right\}\qquad \forall \mu_{1},\mu_{2}\in\mathcal{M}(\Sp^{d}).
\end{equation*}
In \cite{piccoli2014generalized} it is shown that $d_{BL}$ metrizes the weak-$*$ convergence on $\mathcal{M}(\Sp^{d})$. 
As a direct consequence of the definition of $d_{BL}$ and the $C^{2,1}$-regularity of the arccos kernel $K$ from \eqref{eq:def-arccos-operator+kernel}, we find
\begin{equation}\label{eq:uniform-bounds-arccos-operator-BL-distance}
    \lVert\mathcal{K}(\mu_{1}-\mu_{2})\rVert_{C^{2}(\Sp^{d})}\lesssim_{d}d_{BL}(\mu_{1},\mu_{2})\qquad \forall \mu_{1},\mu_{2}\in \mathcal{M}(\Sp^{d}).
\end{equation}
The rest of this section is devoted to the proof of the following result:
\begin{prop}\label{prop:well-posedness-WFR}
    Let $\bar\mu, \nu \in \mathcal{M}_{+}(\Sp^{d})$. There exists a unique global solution $\mu\in C_{w^{*}}([0,\infty); \mathcal{M}_{+}(\Sp^{d}))$ of \eqref{eq:PDE-Wasserstein-Fisher-Rao-sphere10} according to \Cref{def:solution-active-scalar-WFR}, and it can be represented as 
    \begin{equation}\label{eq:weighted-push-forward-WFR}
        \mu_{t}=(X_{t})_{\#}(G_{t}\bar\mu),\qquad G_{t}(x):= \exp\left(-4\int_{0}^{t}\mathcal{K}(\mu_{r}-\nu)\circ X_{r}(x)dr\right)\qquad \forall t\in [0,\infty),
    \end{equation}
    where $X:[0,\infty)\times \Sp^{d}\to \Sp^{d}$ is the flow map associated to the velocity field $v_{t}=-\nabla_{\Sp^{d}}\mathcal{K}(\mu_{t}-\nu)$. Moreover, the energy dissipation identity \eqref{eq:energy-dissipation-Fisher-Rao} holds. In addition:
    \begin{itemize}
        \item [i)] (Preservation of parity). If $\bar\mu, \nu$ are even on $\Sp^{d}$, then $\mu_{t}$ is even for all $t\in (0,\infty)$.
        \item [ii)] (Stability). If $\bar\mu^{n}, \nu^{n}\in \mathcal{M}_{+}(\Sp^{d})$ are such that $\bar\mu^{n}\stackrel{*}{\rightharpoonup}\bar\mu$ and $\nu^{n}\stackrel{*}{\rightharpoonup}\nu$, then the corresponding solution $\mu^{n}_{t}$ of \eqref{eq:PDE-Wasserstein-Fisher-Rao-sphere10} converges to $\mu_{t}$ in bounded-Lipschitz distance uniformly in bounded time intervals:
        \begin{equation*}
            \sup_{t\in [0,T)}d_{BL}(\mu_{t}^{n},\mu_{t})\to 0\qquad \text{as $n\to \infty$},\qquad \forall T\in (0,\infty).
        \end{equation*}
        \item [iii)] (Propagation of regularity). If $\bar\mu, \nu \in C^{k,\alpha}(\Sp^{d})$ for some $k\in \N\cup \{0\}$ and $\alpha\in (0,1]$, then $\mu \in L^{\infty}_{\loc}([0,\infty);C^{k,\alpha}(\Sp^{d}))$. In particular, if $\bar\mu, \nu \in C^{\infty}(\Sp^{d})$, then $\mu\in C^\infty([0,\infty)\times \Sp^{d})$. 
    \end{itemize}
\end{prop}

\begin{proof}
We divide the proof into three steps. 

\smallskip
\noindent \textbf{Step 1:} We first show that if a solution $\mu\in C_{w^{*}}([0,T);\mathcal{M}_{+}(\Sp^{d}))$ exists, then it must be of the form \eqref{eq:weighted-push-forward-WFR}. Let $v$ be the vector field generated by $\mu$ and $\mathcal{X}:[0,T)\times[0,T)\times \Sp^{d}\to \Sp^{d}$ be the extended flow map associated to $v$, solving
\begin{equation*}
    \left\{
    \begin{array}{rclll}
        \frac{d}{dt}\mathcal{X}(t,s,x)&=&v_{t}(\mathcal{X}(t,s,x))\qquad &\forall s,t\in (0,T),\\
        \mathcal{X}(s,s,x)&=&x\qquad &\forall s\in [0,T).
    \end{array}
    \right.
\end{equation*}
Note that $\mathcal{X}_{t,0}\equiv X_{t}$ for all $t\in [0,T)$ and the semigroup property $\mathcal{X}_{t,s}\circ\mathcal{X}_{s,r}=\mathcal{X}_{t,r}$ holds for all $r,s,t \in [0,T)$.
By the definition of push-forward, one can show that $\hat \mu_{t}$ defined as follows is another solution of the forced continuity equation solved by $\mu$, with the same velocity field and forcing term:
\begin{equation*}
    \hat{\mu}_{t}:=(\mathcal{X}_{t,0})_{\#}\bar\mu + \int_{0}^{t}(\mathcal{X}_{t,r})_{\#}g_{r}\mu_{r}dr\qquad g_{t}:=-4\mathcal{K}(\mu_{t}-\nu).
\end{equation*}
Then, by linearity, $\sigma_{t}:=\hat{\mu}_{t}-\mu_{t}$ solves $\partial_{t}\sigma_{t}+\diver_{\Sp^{d}}(\sigma_{t}v_{t})=0$ with $\sigma_{0}=0$, and by the standard theory for uniformly Lipschitz vector fields (see for instance \cite[Lecture 16, Section 1]{ambrosio2021lectures}) we deduce $\sigma_{t}\equiv 0$, that is, $\hat{\mu}_{t}=\mu_{t}$ for every $t\in [0,T)$. In particular, by the semigroup properties of the extended flow map, and the linearity of the push-forward operator:
\begin{align*}
    \mu_{t}=(\mathcal{X}_{t,0})_{\#}\bar\mu+\int_{0}^{t}(\mathcal{X}_{t,r})_{\#}g_{r}\mu_{r}dr= (\mathcal{X}_{t,0})_{\#}\eta_{t},\qquad \eta_{t}:= \bar\mu+ \int_{0}^{t}(\mathcal{X}_{0,r})_{\#}g_{r}\mu_{r}dr\qquad \forall t\in [0,T).
\end{align*}
Using the identity $(\mathcal{X}_{0,r})_{\#}g_{r}\mu_{r}= g_{r}\circ \mathcal{X}_{r,0}(\mathcal{X}_{0,r})_{\#}\mu_{r}=g_{r}\circ X_{r}\eta_{r}$ in the expression above, we find
\begin{equation*}
    \eta_{t}=\bar\mu+\int_{0}^{t}g_{r}\circ X_{r}\eta_{r}dr\qquad \forall t\in [0,T),
\end{equation*}
from which we deduce $\eta_{t}= G_{t}\bar\mu$, where $G_{t}:= \exp\left(\int_{0}^{t}g_{r}\circ X_{r}dr\right)$, and finally $\mu_{t}=(X_{t})_{\#}G_{t}\bar\mu$ for all $t\in [0,T)$, as desired. 

From the representation by push-forward, the energy dissipation identity \eqref{eq:energy-dissipation-Fisher-Rao} follows with a direct computation.

\smallskip
\noindent \textbf{Step 2:} Now we prove the stability of solutions with respect to variations of the data, uniformly in bounded time intervals. This addresses both point iii) and uniqueness.  

For $i=1,2$, let $\bar\mu^{i}, \nu^{i}\in \mathcal{M}_{+}(\Sp^{d})$ and $\mu^{i}\in C_{w^{*}}([0,T);\mathcal{M}_{+}(\Sp^{d}))$ be solutions of \eqref{eq:PDE-Wasserstein-Fisher-Rao-sphere10} for some $T\in (0,\infty)$. We call $v_{t}^{i}$ and $X_{t}^{i}$, respectively, the vector field and the flow map generated by the solution $\mu^{i}$, and 
\begin{equation*}
    g_{t}^{i}:=-4\mathcal{K}(\mu^{i}_{t}-\nu^{i}),\qquad G_{t}^{i}:= \exp \left(\int_{0}^{t}g_{r}^{i}\circ X_{r}^{i}dr\right)\qquad \forall t\in [0,T),\quad i=1,2.
\end{equation*}
By Step 1, we can represent the solutions as
\begin{equation}\label{eq:representation-two-solutions-stability-WFR}
    \mu^{i}_{t}=(X_{t}^{i})_{\#}G_{t}^{i}\bar\mu^{i}\qquad \forall t\in [0,T),\quad i=1,2. 
\end{equation}
We will give a uniform estimate of $d_{BL}(\mu_{t}^{1},\mu_{t}^{2})$ for $t\in [0,T)$ implying the stability in point iii), and in particular uniqueness, when $\bar\mu^{1}=\bar\mu^{2}$ and $\nu^{1}=\nu^{2}$.

We first need some preliminary estimates. We define 
\begin{equation*}
    M:= \bar\mu^{1}(\Sp^{d})+\bar\mu^{2}(\Sp^{d})+\nu^{1}(\Sp^{d})+\nu^{2}(\Sp^{d}).
\end{equation*}
Then, by \eqref{eq:uniform-bound-mass-WFR} and \eqref{eq:C21-bound-arccos-operator} we have
\begin{equation}\label{eq:uniform-bounds-several-objects-stability-WFR}
    \mu^{i}_{t}(\Sp^{d}),\,\lVert v_{t}^{i}\rVert_{C^{1,1}},\,\lVert \nabla X_{t}^{i}\rVert_{C^{0,1}},\,\lVert g^{i}_{t}\rVert_{C^{2,1}},\,\lVert G^{i}_{t}\rVert_{C^{2,1}}\le C(d,M,T),\qquad \forall t\in [0,T),\quad i=1,2.
\end{equation}
Moreover, by \eqref{eq:uniform-bounds-arccos-operator-BL-distance} we also have
\begin{equation}\label{eq:bounds-stability-WFR}
    \lVert g_{t}^{1}-g_{t}^{2}\rVert_{L^{\infty}},\, \lVert v_{t}^{1}-v_{t}^{2}\rVert_{L^{\infty}}\lesssim_{d} d_{BL}(\mu_{t}^{1},\mu_{t}^{2})+d_{BL}(\nu^{1},\nu^{2})\qquad \forall t\in [0,T).
\end{equation}
Integrating these estimates we deduce 
\begin{equation}\label{eq:bounds:stability-WFR-BL-distance}
    \lVert G_{t}^{1}-G_{t}^{2}\rVert_{L^{\infty}},\,\lVert X_{t}^{1}-X_{t}^{2}\rVert_{L^{\infty}}\lesssim_{d,M,T} \int_{0}^{t}d_{BL}(\mu_{r}^{1},\mu_{r}^{2})dr+d_{BL}(\nu^{1},\nu^{2})\qquad \forall t\in [0,T).
\end{equation}
To bound $d_{BL}(\mu_{t}^{1},\mu_{t}^{2})$ for $t\in [0,T)$, we take a Lipschitz function $f:\Sp^{d}\to \R$ such that $\lVert f\rVert_{L^{\infty}}, \Lip(f)\le 1$ and compute, using the representation formula \eqref{eq:representation-two-solutions-stability-WFR} along with \eqref{eq:uniform-bounds-several-objects-stability-WFR}, \eqref{eq:bounds-stability-WFR} and \eqref{eq:bounds:stability-WFR-BL-distance}, 
\begin{align*}
    \int_{\Sp^{d}}fd(\mu_{t}^{1}-\mu_{t}^{2})&=\int_{\Sp^{d}}f\circ X_{t}^{1}G_{t}^{1}d\bar\mu^{1}-\int_{\Sp^{d}}f\circ X_{t}^{2}G_{t}^{2}d\bar\mu^{2}\\
    &= \int_{\Sp^{d}}(f\circ X_{t}^{1}-f\circ X_{t}^{2})G_{t}^{1}d\bar\mu^{1}+\int_{\Sp^{d}}f\circ X_{t}^{2}(G_{t}^{1}-G_{t}^{2})d\bar\mu^{1}+\int_{\Sp^{d}}f\circ X_{t}^{2}G_{t}^{2}d(\bar\mu^{1}-\bar\mu^{2})\\
    &\le \left(\lVert X_{t}^{1}-X_{t}^{2}\rVert_{L^{\infty}}\lVert G_{t}^{1}\rVert_{L^{\infty}}+\lVert G_{t}^{1}-G_{t}^{2}\rVert_{L^{\infty}}\right)\bar\mu^{1}(\Sp^{d})+\lVert f\circ X_{t}^{2}G_{t}^{2}\rVert_{C^{0,1}}d_{BL}(\bar\mu^{1},\bar\mu^{2})\\
    &\lesssim_{d,M,T} \int_{0}^{t}d_{BL}(\mu_{r}^{1},\mu_{r}^{2})dr+d_{BL}(\nu^{1},\nu^{2})+d_{BL}(\bar\mu^{1},\bar\mu^{2}).
\end{align*}
By the arbitrariness of $f$, we obtain
\begin{equation*}
    d_{BL}(\mu_{t}^{1},\mu_{t}^{2})\lesssim_{d,M,T} \int_{0}^{t}d_{BL}(\mu_{r}^{1},\mu_{r}^{2})dr+d_{BL}(\nu^{1},\nu^{2})+d_{BL}(\bar\mu^{1},\bar\mu^{2})\qquad \forall t\in [0,T).
\end{equation*}
Then, Gr\"onwall's inequality yields
\begin{equation*}
    d_{BL}(\mu_{t}^{1}, \mu_{t}^{2})\le d_{BL}(\bar\mu^{1},\bar\mu^{2})e^{C(d,M,T)t}+ \left(d_{BL}(\bar\mu^{1},\bar\mu^{2})+d_{BL}(\nu^{1},\nu^{2})\right)\left(e^{C(d,M,T)t}-1\right)\qquad \forall t\in [0,T),
\end{equation*}
which concludes the proof of point iii) and uniqueness. 

\smallskip
\noindent \textbf{Step 3:} It only remains to prove the existence of a solution, and the propagation of regularity and parity from the data. For both claims, it is sufficient to work in a short time interval $[0,\tau)$. In fact, by the same arguments as in \Cref{prop:maximal-solutions}, the solution can be extended as long as $\mu_{t}(\Sp^{d})$ remains bounded, and so to the whole $[0,\infty)$, thanks to the uniform upper bound from \eqref{eq:uniform-bound-mass-WFR}.

We proceed, as usual, by building recursively a sequence of approximate solutions:
\begin{equation*}
    \mu^{0}_{t}=\bar\mu,\qquad \mu^{n+1}_{t}= (X_{t}^{n})_{\#}G_{t}^{n}\bar \mu\qquad \forall t\in [0,\tau),\quad \forall n\ge 0,
\end{equation*}
where $v_{t}^{n}, X_{t}^{n}$ denote the vector field and the flow map generated by $\mu_{t}^{n}$, respectively, and
\begin{equation*}
    g_{t}^{n}:=-4\mathcal{K}(\mu_{t}^{n}-\nu),\qquad G_{t}^{n}:= \exp\left(\int_{0}^{t}g_{r}^{n}\circ X_{r}^{n}dr\right)\qquad \forall t\in [0,\tau),\quad \forall n\ge 0.
\end{equation*}
First of all we prove that the following holds as soon as $\tau>0$ is chosen sufficiently small:
\begin{equation}\label{eq:unif-bound-mass-existence-WFR}
    \mu_{t}^{n}(\Sp^{d})\le 2\bar\mu(\Sp^{d})\qquad \forall t\in [0,\tau),\quad \forall n\ge 0.
\end{equation}
From this uniform bound on the mass, applying \eqref{eq:C21-bound-arccos-operator}, we will deduce
\begin{equation}\label{eq:unif-bounds-WFR-existence}
      \lVert v_{t}^{n}\rVert_{C^{1,1}},\,\lVert \nabla X_{t}^{n}\rVert_{C^{0,1}},\,\lVert \nabla (X_{t}^{n})^{-1}\rVert_{C^{0,1}},\,\lVert g_{t}^{n}\rVert_{C^{2,1}},\,\lVert G^{n}_{t}\rVert_{C^{2,1}}\le C(d,\tau, \bar\mu(\Sp^{d}), \nu(\Sp^{d}))\qquad
       \forall t\in [0,\tau),\quad \forall n\ge 0.
\end{equation}
The bound \eqref{eq:unif-bound-mass-existence-WFR} is trivial for $n=0$. Assume it holds for some $n\ge 0$ and let us prove it holds for $n+1$. By \eqref{eq:C21-bound-arccos-operator}, for all $t\in [0,\tau)$,
\begin{equation*}
    \mu_{t}^{n+1}(\Sp^{d})=\int_{\Sp^{d}}G_{t}^{n}d\bar\mu \le \lVert G_{t}^{n}\rVert_{L^{\infty}}\bar\mu(\Sp^{d})\le \exp\left\{C(d)(2\bar\mu(\Sp^{d})+\nu(\Sp^{d})\tau\right\}\bar\mu(\Sp^{d})\le 2\bar\mu(\Sp^{d}),
\end{equation*}
provided that we choose $\tau$ sufficiently small. 

Once \eqref{eq:unif-bounds-WFR-existence} is known, proceeding as in Step 2, we find
\begin{equation*}
    d_{BL}(\mu_{t}^{n},\mu_{t}^{n+1})\lesssim_{d,\tau,\bar\mu(\Sp^{d}),\nu(\Sp^{d})}\int_{0}^{t}\left(d_{BL}(\mu_{r}^{n}, \mu_{r}^{n+1})+d_{BL}(\mu_{r}^{n-1}, \mu_{r}^{n})\right)dr\qquad \forall t\in [0,\tau), \quad \forall n\ge 1.
\end{equation*}
From here, by the same argument as in Step 2 of the proof of \Cref{prop:existence-yudovich}, up to decreasing $\tau$, we deduce that 
\begin{equation}\label{eq:convergence-subsequent-approximate-solutions-WFR}
    \lim_{n\to \infty}\sup_{k\ge n}\max_{t\in [0,\tau)}d_{BL}(\mu_{t}^{k},\mu_{t}^{k+1})=0.
\end{equation}
To conclude the proof of existence of a solution to \eqref{eq:PDE-Wasserstein-Fisher-Rao-sphere10} in the interval $[0,\tau)$ we argue similarly to Step 3 of the proof of \Cref{prop:existence-yudovich}. First, by \eqref{eq:unif-bounds-WFR-existence}, $X^{n}, G^{n}$ are uniformly bounded and equi-continuous in $[0,\tau)\times \Sp^{d}$. Therefore, by Arzelà-Ascoli we can find continuous $X:[0,\tau)\times \Sp^{d}\to \Sp^{d}$, $G:[0,\tau)\times \Sp^{d}\to [0,\infty)$, and a subsequence $n_{k}\to \infty$ such that
\begin{equation*}
   \lim_{k\to \infty}\left(\lVert X^{n_{k}}-X\rVert_{C([0,\tau)\times \Sp^{d})}+\lVert G^{n_{k}}-G\rVert_{C([0,\tau)\times \Sp^{d})}\right)=0.
\end{equation*}
Defining $\mu_{t}:= (X_{t})_{\#}G_{t}\bar\mu$, we deduce that
\begin{equation}\label{eq:weak-convergence-approximate-solutions-existence-WFR}
    \mu_{t}^{n_{k}}\stackrel{*}{\rightharpoonup} \mu_{t}\qquad \forall t\in [0,\tau),\qquad \mu \in C_{w^{*}}([0,\tau);\mathcal{M}_{+}(\Sp^{d})).
\end{equation}
Then, from \eqref{eq:convergence-subsequent-approximate-solutions-WFR} we also obtain
\begin{equation*}
    \mu_{t}^{n_{k}-1}\stackrel{*}{\rightharpoonup} \mu_{t}\qquad \forall t\in [0,\tau).
\end{equation*}
Finally, setting $g_{t}:=-4\mathcal{K}(\mu_{t}-\nu)$ and $v_{t}:=-\nabla_{\Sp^{d}}\mathcal{K}(\mu_{t}-\nu)$, \eqref{eq:uniform-bounds-arccos-operator-BL-distance} implies
\begin{equation}\label{eq:strong-convergence-vector-fields-existence-WFR}
     \lVert g_{t}^{n_{k}-1}-g_{t}\rVert_{C(\Sp^{d})},\,\lVert v_{t}^{n_{k}-1}-v_{t}\rVert_{C(\Sp^{d})}\to 0\qquad \forall t\in [0,\tau).
\end{equation}
To show that $\mu$ is a solution of \eqref{eq:PDE-Wasserstein-Fisher-Rao-sphere10}, it is then sufficient to pass to the limit via the dominated convergence theorem in both sides of the distributional formulation
\begin{equation*}
    \int_{\Sp^{d}}\varphi \mu_{t}^{n_{k}}=\int_{\Sp^{d}}\varphi \bar\mu+\int_{0}^{t}\int_{\Sp^{d}}\nabla_{\Sp^{d}} \varphi \cdot v_{r}^{n_{k}-1}\mu_{r}^{n_{k}}dr+\int_{0}^{t}\int_{\Sp^{d}}\varphi g_{r}^{n_{k}-1}d\mu_{r}^{n_{k}}dr\qquad \forall t\in [0,\tau),\quad \forall \varphi \in C^{\infty}(\T^{d}),
\end{equation*}
taking into account \eqref{eq:weak-convergence-approximate-solutions-existence-WFR}, \eqref{eq:strong-convergence-vector-fields-existence-WFR}, and the uniform boundedness of $g^{n}, g, v^{n}, v$. This concludes the proof of existence. 

Regarding point ii), observe that if $\bar\mu, \nu$ are even, then, in the approximating sequence above, $\mu_{t}^{n}$ is even for all $t\in [0,\tau)$ and all $n\ge 0$, as can be shown by induction on $n$ using the fact that $\mathcal{K}$ preserves parity. The parity of the approximating sequence is then inherited by the limiting solution. 

Finally, the proof of the propagation of H\"{o}lder regularity from point iv) can be done as in \Cref{prop:propagation-regularity} using the regularizing properties of $\mathcal{K}$ from \eqref{eq:C21-bound-arccos-operator}. The smooth regularity in space-time of solutions with smooth data is then deduced as in \Cref{rmk:regularity-in-time}. We leave the details concerning this point to the reader.  
\end{proof}

\subsection{Quantitative convergence for the Wasserstein--Fisher--Rao dynamics}\label{subsec:quantitative-convergence-WFR}
In this section we prove \Cref{thm:convergence-arccos}. The strategy is based on the combination of the energy dissipation identity \eqref{eq:energy-dissipation-Fisher-Rao} with the higher order energy estimates that we derive in the next lemma (see \Cref{subsec:idea-proof}).  For convenience, in the sequel we work with integer parameters and we use classical angular derivatives in place of fractional derivatives (see the notation introduced in \Cref{subsec:analysis-sphere-arccos}).

\begin{lem} \label{lem:energy-estimate-arccos} Let $d\in \N$ be odd and $s:=(d+3)/2\in \N$. Let $\bar\mu,  \nu\in C^{\infty}(\Sp^{d})$ be even nonnegative functions, and $\mu \in C^{\infty}([0,\infty)\times \Sp^{d})$ the solution of \eqref{eq:PDE-Wasserstein-Fisher-Rao-sphere10} given by \Cref{prop:well-posedness-WFR}. For every $\gamma\in \N$, $\gamma\ge s-1$, the following energy estimates hold for some constant $C=C(d,\gamma)>0$ and all $t\in (0,\infty)$:
    \begin{align}
    \frac{d}{dt}\lVert \mu_{t}\rVert_{\dot H^{\gamma}}^{2}&\le C\left(\lVert \nu\rVert_{H^{\gamma}}+\lVert \mathcal{K}(\mu_{t}-\nu)\rVert_{C^{2}}\right)\left(\lVert \mu_{t}\rVert_{H^{\gamma}}^{2}+\lVert \nu\rVert_{H^{\gamma}}^{2}\right),\label{eq:energy-estimate-arccos-propagation}\\[5pt]
        \frac{d}{dt}\lVert \mu_{t}\rVert_{\dot H^{\gamma}}^{2}&\le C\lVert \mathcal{K}(\mu_{t}-\nu)\rVert_{C^{2}}\left(\lVert \mu_{t}\rVert_{H^{\gamma}}^{2}+\lVert \nu\rVert_{H^{\gamma}}^{2}\right)\label{eq:energy-estimate-arccos}\\
        &\quad\,-\frac{1}{C}\big(\min_{\Sp^{d}}\nu\big)\lVert \mu_{t}-\nu\rVert_{\dot H^{\gamma-s+1}}^{2}+ C\lVert \nu\rVert_{H^{\gamma}}\lVert \mu_{t}-\nu\rVert_{\dot H^{\max\{\gamma,\frac{d}{2}+1\}-s}}\lVert \mu_{t}-\nu\rVert_{\dot H^{\gamma-s+1}}\notag\\
        &\quad\, +C\lVert \nu\rVert_{H^{\gamma}}\lVert \nu\rVert_{\dot H^{\gamma+1}}\lVert \mu_{t}-\nu\rVert_{\dot H^{\gamma+1-2s}}.
\notag
    \end{align}
\end{lem}

\begin{proof}
    For convenience, we denote by $\sigma_{t}:= \mu_{t}-\nu$ the perturbation with respect to the target. Using equation \eqref{eq:PDE-Wasserstein-Fisher-Rao-sphere10} and integrating the divergence term by parts, we get
    \begin{align*}
        \frac{d}{dt}\frac{\lVert \mu_{t}\rVert_{\dot H^{\gamma}}^{2}}{2}&=\sum_{|\beta|=\gamma}\int_{\Sp^{d}}\nabla_{\beta}\partial_{t}\mu_{t}\nabla_{\beta}\mu_{t}\\
        &=\underbrace{-\sum_{|\beta|=\gamma}\int_{\Sp^{d}}\nabla_{\beta}\left[\mu_{t}\nabla_{\Sp^{d}}\mathcal{K}(\sigma_{t})\right]\cdot\nabla_{\Sp^{d}}\nabla_{\beta}\mu_{t}}_{I}+\underbrace{(-4)\sum_{|\beta|=\gamma}\int_{\Sp^{d}}\nabla_{\beta}\left[\mu_{t}\mathcal{K}(\sigma_{t})\right]\nabla_{\beta}\mu_{t}}_{II}.
    \end{align*}
    We first show in detail how to bound $I$, which is the leading order term, as $II$ contains two orders of derivative less.
    
    \smallskip
    \noindent \textbf{Step 1:} From $I$, we isolate the term with $2\gamma+1$ derivatives falling on $\mu_{t}$:
    \begin{align*}
        I&= \underbrace{-\sum_{|\beta|=\gamma}\int_{\Sp^{d}}\nabla_{\beta}\mu_{t}\nabla_{\Sp^{d}}\mathcal{K}(\sigma_{t})\cdot \nabla_{\Sp^{d}}\nabla_{\beta}\mu_{t}}_{I_{1}}+\underbrace{(-1)\sum_{0\le|\beta_{1}|\le \gamma-1}\sum_{|\beta_{2}|=\gamma+1-|\beta_{1}|}\int_{\Sp^{d}}\nabla_{\beta_{1}}\mu_{t}\nabla_{\beta_{2}}\mathcal{K}(\sigma_{t})\nabla_{\beta_{1}+\beta_{2}}\mu_{t}}_{I_{2}}.
    \end{align*}
    Using the identity $\nabla_{\Sp^{d}}(\nabla_{\beta}\mu_{t})^{2}=2\nabla_{\beta}\mu_{t}\nabla_{\Sp^{d}}\nabla_{\beta}\mu_{t}$ and integration by parts, we obtain
    \begin{equation}\label{eq:ee-nn-I1}
        I_{1}= \frac{1}{2}\sum_{|\beta|=\gamma}\int_{\Sp^{d}}\Delta_{\Sp^{d}}\mathcal{K}(\sigma_{t})|\nabla_{\beta}\mu_{t}|^{2}\le C\lVert \nabla^{2}_{\Sp^{d}}\mathcal{K}(\sigma_{t})\rVert_{L^{\infty}}\lVert \mu_{t}\rVert_{\dot H^{\gamma}}^{2}.
    \end{equation}
    Next, writing $\mu_{t}=\nu+\sigma_{t}$, we further divide $I_{2}$ into three terms:
    \begin{align*}
        I_{2}&= -\sum_{0\le|\beta_{1}|\le \gamma-1}\sum_{|\beta_{2}|=\gamma+1-|\beta_{1}|}\int_{\Sp^{d}}\nabla_{\beta_{1}}\sigma_{t}\nabla_{\beta_{2}}\mathcal{K}(\sigma_{t})\nabla_{\beta_{1}+\beta_{2}}\mu_{t}\\
        &\quad\, -\sum_{0\le|\beta_{1}|\le \gamma-1}\sum_{|\beta_{2}|=\gamma+1-|\beta_{1}|}\int_{\Sp^{d}}\nabla_{\beta_{1}}\nu\nabla_{\beta_{2}}\mathcal{K}(\sigma_{t})\nabla_{\beta_{1}+\beta_{2}}\sigma_{t}\\
        &\quad\, -\sum_{0\le|\beta_{1}|\le \gamma-1}\sum_{|\beta_{2}|=\gamma+1-|\beta_{1}|}\int_{\Sp^{d}}\nabla_{\beta_{1}}\nu\nabla_{\beta_{2}}\mathcal{K}(\sigma_{t})\nabla_{\beta_{1}+\beta_{2}}\nu=:I_{2,1}+I_{2,2}+I_{2,3}.
    \end{align*}

    \smallskip
    \noindent \textbf{Step 2:} In this step, we prove the following estimate for $I_{2,1}$:
    \begin{equation}\label{eq:ee-nn-I21}
        I_{2,1}\le C\lVert \nabla^{2}_{\Sp^{d}}\mathcal{K}(\sigma_{t})\rVert_{L^{\infty}}\lVert \sigma\rVert_{H^{\gamma}}\lVert \mu_{t}\rVert_{\dot H^{\gamma}}.
    \end{equation}
    In each integral defining $I_{2,1}$, integrating by parts, we may distribute one derivative from $\nabla_{\beta_{1}+\beta_{2}}\mu_{t}$ to the product $\nabla_{\beta_{1}}\sigma_{t}\nabla_{\beta_{2}}\sigma_{t}$. This allows us to write $I_{2,1}$ as the integral of a finite sum of terms of the type $\nabla_{\alpha_{1}}\sigma_{t}\nabla_{\alpha_{2}}\mathcal{K}(\sigma_{t})\nabla_{\alpha_{3}}\mu_{t}$, where 
    \begin{equation}\label{eq:ee-nn-I21-condition-exponents-derivatives}
        0\le |\alpha_{1}|\le \gamma,\quad 2\le |\alpha_{2}|\le \gamma+2,\quad |\alpha_{1}|+|\alpha_{2}|=\gamma+2,\quad |\alpha_{3}|=\gamma.
    \end{equation}
     Then, to conclude \eqref{eq:ee-nn-I21}, it suffices to show that, for every $\alpha_{1},\alpha_{2}$ satisfying \eqref{eq:ee-nn-I21-condition-exponents-derivatives}, the following holds:
    \begin{equation}\label{eq:claim-energy-estimate-NonLin}
        \lVert\nabla_{\alpha_{1}}\sigma_{t}\nabla_{\alpha_{2}}\mathcal{K}(\sigma_{t})\rVert_{L^{2}}\le C\lVert \nabla^{2}_{\Sp^{d}}\mathcal{K}(\sigma_{t})\rVert_{L^{\infty}}\lVert \sigma_{t}\rVert_{H^{\gamma}}.
    \end{equation}
    We will assume that $0\le |\alpha_{1}|<\gamma$ and $2<|\alpha_{2}|\le \gamma+2$, as the case in which $|\alpha_{1}|=\gamma$ and $|\alpha_{2}|=2$ follows easily by H\"{o}lder inequality. 
    We write $\sigma_{t}= (\sigma_{t})_{0}+T^{-1}(-\Delta_{\Sp^{d}})^{s}\mathcal{K}(\sigma_{t})$, where $T=(-\Delta_{\Sp^{d}})^{s}\mathcal{K}$ is the zeroth order operator from point ii) in \Cref{lem:spectral-behavior-arccos-operator}, and $(\sigma_{t})_{0}$ is the mean of $\sigma_{t}$. 
    If $|\alpha_{1}|>0$, then $\nabla_{\alpha_{1}}\sigma_{t}=T^{-1}(-\Delta_{\Sp^{d}})^{s}\nabla_{\alpha_{1}}\mathcal{K}(\sigma_{t})$, therefore, applying H\"{o}lder inequality, point ii) in \Cref{lem:spectral-behavior-arccos-operator}, and Gagliardo--Nirenberg interpolation inequality, we get
    \begin{align*}
        \lVert\nabla_{\alpha_{1}}\sigma_{t}\nabla_{\alpha_{2}}\mathcal{K}(\sigma_{t})\rVert_{L^{2}}
        &\le \lVert T^{-1}(-\Delta_{\Sp^{d}})^{s}\nabla_{\alpha_{1}}\mathcal{K}(\sigma_{t})\rVert_{L^{2\frac{\gamma+2s-2}{\gamma+2s-|\alpha_{2}|}}}\lVert \nabla_{\alpha_{2}}\mathcal{K}(\sigma_{t})\rVert_{L^{2\frac{\gamma+2s-2}{|\alpha_{2}|-2}}}\\
        &\lesssim\lVert \nabla_{\Sp^{d}}^{2s+|\alpha_{1}|}\mathcal{K}(\sigma_{t})\rVert_{L^{2\frac{\gamma+2s-2}{\gamma+2s-|\alpha_{2}|}}}\lVert \nabla_{\Sp^{d}}^{|\alpha_{2}|}\mathcal{K}(\sigma_{t})\rVert_{L^{2\frac{\gamma+2s-2}{|\alpha_{2}|-2}}}\\
        &\lesssim\lVert \nabla^{2}_{\Sp^{d}}\mathcal{K}(\sigma_{t})\rVert_{L^{\infty}}^{\frac{|\alpha_{2}|-2}{\gamma+2s-2}}\lVert \nabla^{2s+\gamma}_{\Sp^{d}}\mathcal{K}(\sigma_{t})\rVert_{L^{2}}^{\frac{\gamma+2s-|\alpha_{2}|}{\gamma+2s-2}}\lVert \nabla^{2}_{\Sp^{d}}\mathcal{K}(\sigma_{t})\rVert_{L^{\infty}}^{\frac{\gamma+2s-|\alpha_{2}|}{\gamma+2s-2}}\lVert \nabla^{2s+\gamma}_{\Sp^{d}}\mathcal{K}(\sigma_{t})\rVert_{L^{2}}^{\frac{|\alpha_{2}|-2}{\gamma+2s-2}}\\
        &=\lVert \nabla^{2}_{\Sp^{d}}\mathcal{K}(\sigma_{t})\rVert_{L^{\infty}}\lVert \nabla^{2s+\gamma}_{\Sp^{d}}\mathcal{K}(\sigma_{t})\rVert_{L^{2}}\\
&\lesssim \lVert \nabla^{2}_{\Sp^{d}} \mathcal{K}(\sigma_{t})\rVert_{L^{\infty}}\lVert \sigma_{t}\rVert_{\dot H^{\gamma}}.
    \end{align*}
    If $|\alpha_{1}|=0$, the same bound holds, provided that we use the full norm $\lVert \sigma_{t}\rVert_{H^{\gamma}}$ instead of $\lVert \sigma_{t}\rVert_{\dot H^{\gamma}}$. This concludes the proof of \eqref{eq:ee-nn-I21}.

\smallskip 
\noindent \textbf{Step 3:} Next, we prove the following bound for $I_{2,2}$:
\begin{equation}\label{eq:ee-nn-I22}
    I_{2,2}\le -\frac{1}{C}\big(\min_{\Sp^{d}}\nu\big)\lVert \sigma_{t}\rVert_{\dot H^{\gamma-s+1}}^{2}+C\lVert \nu\rVert_{H^{\gamma}}\left(\lVert \sigma_{t}\rVert_{\dot H^{\max\{\gamma,\frac{d}{2}+1\}-s}}\lVert \sigma_{t}\rVert_{\dot H^{\gamma-s+1}}+\lVert \nabla^{2}_{\Sp^{d}}\mathcal{K}(\sigma_{t})\rVert_{L^{\infty}}\lVert \sigma_{t}\rVert_{H^{\gamma}}\right).
\end{equation}
In each term appearing in the definition of $I_{2,2}$, we use integration by parts to distribute $s$ derivatives from $\nabla_{\beta_{1}+\beta_{2}}\sigma_{t}$ to the product $\nabla_{\beta_{1}}\nu \nabla_{\beta_{2}}\mathcal{K}(\sigma_{t})$, and we split the result according to the number of derivatives falling on $\nu$:
\begin{align*}
    I_{2,2}&= -\sum_{0\le|\alpha_{1}|\le \gamma+s-1}\sum_{|\alpha_{2}|=\gamma+s+1-|\alpha_{1}|}\int_{\Sp^{d}}\nabla_{\alpha_{1}}\nu\nabla_{\alpha_{2}}\mathcal{K}(\sigma_{t})(-\Delta_{\Sp^{d}})^{-s}\nabla_{\alpha_{1}+\alpha_{2}}\sigma_{t}\\
    &= -\sum_{|\alpha|=\gamma+s+1}\int_{\Sp^{d}}\nu\nabla_{\alpha}\mathcal{K}(\sigma_{t})(-\Delta_{\Sp^{d}})^{-s}\nabla_{\alpha}\sigma_{t}\\
    &\quad\, -\sum_{1\le|\alpha_{1}|\le \gamma-1}\sum_{|\alpha_{2}|=\gamma+s+1-|\alpha_{1}|}\int_{\Sp^{d}}\nabla_{\alpha_{1}}\nu\nabla_{\alpha_{2}}\mathcal{K}(\sigma_{t})(-\Delta_{\Sp^{d}})^{-s}\nabla_{\alpha_{1}+\alpha_{2}}\sigma_{t}\\
    &\quad\, -\sum_{\gamma\le|\alpha_{1}|\le \gamma+s-1}\sum_{|\alpha_{2}|=\gamma+s+1-|\alpha_{1}|}\int_{\Sp^{d}}\nabla_{\alpha_{1}}\nu\nabla_{\alpha_{2}}\mathcal{K}(\sigma_{t})(-\Delta_{\Sp^{d}})^{-s}\nabla_{\alpha_{1}+\alpha_{2}}\sigma_{t}=:I_{2,2,1}+I_{2,2,2}+I_{2,2,3}.
\end{align*}
We first consider $I_{2,2,1}$. Let $\bar{c}=\bar c(d)>0$ be the constant appearing in point iii) of \Cref{lem:spectral-behavior-arccos-operator}. Adding and subtracting $\bar c(-\Delta_{\Sp^{d}})^{-s}$ to $\mathcal{K}$, and using Cauchy--Schwarz inequality, we get
\begin{equation}\label{eq:ee-nn-I221}
    \begin{aligned}
        I_{2,2,1}&= -\bar c\sum_{|\alpha|=\gamma+s+1}\int_{\Sp^{d}}\nu\left((-\Delta_{\Sp^{d}})^{-s}\nabla_{\alpha}\sigma_{t}\right)^{2}\\&\quad\,-\sum_{|\alpha|=\gamma+s+1}\int_{\Sp^{d}}\nu\left(\mathcal{K}-\bar c(-\Delta_{\Sp^{d}})^{-s}\right)(\nabla _{\alpha}\sigma_{t})(-\Delta_{\Sp^{d}})^{-s}(\nabla_{\alpha}\sigma_{t})\\
        &\le -\frac{1}{C}\big(\min_{\Sp^{d}} \nu\big)\lVert \sigma_{t}\rVert_{\dot H^{\gamma-s+1}}^{2}+C\lVert \nu\rVert_{L^{\infty}}\lVert \sigma_{t}\rVert_{\dot H^{\gamma-s}}\lVert \sigma_{t}\rVert_{\dot H^{\gamma-s+1}}.
    \end{aligned}
\end{equation}
To bound $I_{2,2,2}$, we use H\"{o}lder inequality and Sobolev embedding for each integral appearing in its definition. If $\gamma-|\alpha_{1}|>d/2$, since $|\alpha_{2}|\le \gamma+s$, 
\begin{equation*}
    \lVert \nabla_{\alpha_{1}}\nu\nabla_{\alpha_{2}}\mathcal{K}(\sigma_{t})\rVert_{L^{2}}\le \lVert \nabla_{\alpha_{1}}\nu\rVert_{L^{\infty}}\lVert \nabla_{\alpha_{2}}\mathcal{K}(\sigma_{t})\rVert_{L^{2}}\le C\lVert \nu\rVert_{\dot H^{\gamma}}\lVert \sigma_{t}\rVert_{\dot H^{\gamma-s}}. 
\end{equation*}
Suppose $\gamma-|\alpha_{1}|<d/2$. Then, setting $p=2d/(d-2\gamma+2|\alpha_{1}|)$ and $q=2d/(d-2r+2|\alpha_{2}|)$, where $r=d/2+s+1$, we get
\begin{equation*}
    \lVert \nabla_{\alpha_{1}}\nu\nabla_{\alpha_{2}}\mathcal{K}(\sigma_{t})\rVert_{L^{2}}\le \lVert \nabla_{\alpha_{1}}\nu\rVert_{L^{p}}\lVert \nabla_{\alpha_{2}}\mathcal{K}(\sigma_{t})\rVert_{L^{q}}\le C\lVert \nu\rVert_{\dot H^{\gamma}}\lVert \sigma_{t}\rVert_{\dot H^{\frac{d}{2}-s+1}}. 
\end{equation*}
Therefore, 
\begin{equation}\label{eq:ee-nn-I222}
    I_{2,2,2}\le C\lVert \nu\rVert_{\dot H^{\gamma}}\lVert \sigma_{t}\rVert_{\dot H^{\max\{\gamma,\frac{d}{2}+1\}-s}}\lVert \sigma_{t}\rVert_{\dot H^{\gamma-s+1}}.
\end{equation}
Regarding $I_{2,2,3}$, integrating by parts so as to distribute $|\alpha_{1}|-\gamma$ derivatives from $\nabla_{\alpha_{1}}\nu$ to the other two terms, we reduce to bound integrals analogous to those defining $I_{2,1}$, with $\nu$ in place of $\mu_{t}$. The same arguments of Step 2 above then lead to the following estimate:
\begin{equation}\label{eq:ee-nn-I223}
    I_{2,2,3}\le C\lVert \nabla^{2}_{\Sp^{d}}\mathcal{K}(\sigma_{t})\rVert_{L^{\infty}}\lVert \sigma_{t}\rVert_{\dot H^{\gamma}}\lVert \nu\rVert_{\dot H^{\gamma}}.
\end{equation}
Gathering \eqref{eq:ee-nn-I221}, \eqref{eq:ee-nn-I222}, and \eqref{eq:ee-nn-I223}, we get \eqref{eq:ee-nn-I22}.

\smallskip
\noindent \textbf{Step 4:} Finally, we bound $I_{2,3}$:
\begin{equation}\label{eq:ee-nn-I23}
    I_{2,3}\le C\lVert \nu\rVert_{H^{\gamma}}\lVert \nu\rVert_{\dot H^{\gamma+1}}\lVert \sigma_{t}\rVert_{\dot H^{\gamma+1-2s}}.
\end{equation}
The strategy is similar to the one adopted for $I_{2,2,2}$: we use H\"{o}lder inequality and Sobolev embedding for each integral appearing in its definition. If $\gamma-|\beta_{1}|>d/2$, since $|\beta_{2}|\le \gamma+1$, we have
\begin{equation*}
    \lVert \nabla_{\beta_{1}}\nu\nabla_{\beta_{2}}\mathcal{K}(\sigma_{t})\rVert_{L^{2}}\le \lVert \nabla_{\beta_{1}}\nu\rVert_{L^{\infty}}\lVert \nabla_{\beta_{2}}\mathcal{K}(\sigma_{t})\rVert_{L^{2}}\le C\lVert \nu\rVert_{H^{\gamma}}\lVert \sigma_{t}\rVert_{\dot H^{\gamma+1-2s}}.
\end{equation*}
Suppose $\gamma-|\beta_{2}|<d/2$. Setting $p=2d/(d-2\gamma+2|\beta_{1}|)$ and $q=2d/(d-2r+2|\beta_{2}|)$, where $r=d/2+1$, we get
\begin{equation*}
    \lVert \nabla_{\beta_{1}}\nu\nabla_{\beta_{2}}\mathcal{K}(\sigma_{t})\rVert_{L^{2}}\le \lVert \nabla_{\beta_{1}}\nu\rVert_{L^{p}}\lVert \nabla_{\beta_{2}}\mathcal{K}(\sigma_{t})\rVert_{L^{q}}\le C\lVert \nu\rVert_{H^{\gamma}}\lVert \sigma_{t}\rVert_{\dot H^{\frac{d}{2}+1-2s}}.
\end{equation*}
Since, by assumption $\gamma>d/2$, the bound in the right-hand side for the first case is larger, and we find \eqref{eq:ee-nn-I23}. 

If we pre-distribute one derivative from $\nabla_{\beta_{1}+\beta_{2}}\nu$ to the other two terms, the same approach leads to 
\begin{equation}\label{eq:ee-nn-I23-propagation}
     I_{2,3}\le C\lVert \nu\rVert_{H^{\gamma}}^{2}\lVert \sigma_{t}\rVert_{\dot H^{\gamma+2-2s}}.
\end{equation}

\smallskip
\noindent \textbf{Step 5:} 
Gathering \eqref{eq:ee-nn-I1}, \eqref{eq:ee-nn-I21}, \eqref{eq:ee-nn-I22}, \eqref{eq:ee-nn-I23}, and \eqref{eq:ee-nn-I23-propagation}, we get
\begin{equation}\label{eq:ee-nn-I}
\begin{aligned}
        I&\le \lVert \nabla^{2}_{\Sp^{d}}\mathcal{K}(\sigma_{t})\rVert_{L^{\infty}}\left(\lVert \mu_{t}\rVert_{H^{\gamma}}^{2}+\lVert \nu\rVert_{H^{\gamma}}^{2}\right)\\
        &\quad\, -\frac{1}{C}\big(\min_{\Sp^{d}}\nu\big)\lVert \sigma_{t}\rVert_{\dot H^{\gamma-s+1}}^{2}+C\lVert \nu\rVert_{H^{\gamma}}\lVert \sigma_{t}\rVert_{\dot H^{\max\{\gamma,\frac{d}{2}+1\}-s}}\lVert \sigma_{t}\rVert_{\dot H^{\gamma-s+1}}\\
        &\quad\, +C \lVert \nu\rVert_{H^{\gamma}}\min\{\lVert \nu\rVert_{\dot H^{\gamma+1}}\lVert \sigma_{t}\rVert_{\dot H^{\gamma+1-2s}}, \lVert \nu\rVert_{\dot H^{\gamma}}\lVert \sigma_{t}\rVert_{\dot H^{\gamma+2-2s}}\}.
\end{aligned}
\end{equation}
To bound $II$, one can follow the same steps as for $I$, which are made even simpler by the fact that $II$ contains two derivatives less than $I$:
\begin{equation}\label{eq:ee-nn-II}
    II\le C\lVert \mathcal{K}(\sigma_{t})\rVert_{L^{\infty}}\left(\lVert \mu_{t}\rVert_{H^{\gamma}}^{2}+\lVert \nu\rVert_{H^{\gamma}}^{2}\right) + C\lVert \nu\rVert_{H^{\gamma}}\lVert \sigma_{t}\rVert_{\dot H^{\gamma-s}}^{2}+C\lVert \nu\rVert_{H^{\gamma}}^{2}\lVert \sigma_{t}\rVert_{\dot H^{\gamma-2s}}.
\end{equation}
Finally, \eqref{eq:ee-nn-I} and \eqref{eq:ee-nn-II} together yield both \eqref{eq:energy-estimate-arccos-propagation} and \eqref{eq:energy-estimate-arccos}.
\end{proof}

Combining the energy estimate \eqref{eq:energy-estimate-arccos-propagation} with the mass upper bound in \eqref{eq:uniform-bound-mass-WFR}, we may obtain a control on the (exponential) growth of $H^{\gamma}$-norms of smooth solutions, depending only on $d, \gamma$, and the $H^{\gamma}$-norms of the data $\bar\mu, \nu$. In particular, with the usual approximation argument (see for instance the proof of \Cref{prop:propagation-sobolev}), we obtain the following:
\begin{cor}\label{cor:arccos-propagation-sobolev}
    Let $d\in \N$ be odd and let $s=(d+3)/2 \in \N$. Let $\gamma\in \N$ be such that $\gamma\ge s-1$, let $\bar\mu, \nu \in H^{\gamma}(\Sp^{d})$ be even nonnegative functions, and let $\mu$ be the solution of \eqref{eq:PDE-Wasserstein-Fisher-Rao-sphere10} given by \Cref{prop:well-posedness-WFR}. Then, $\mu\in L^{\infty}_{\loc}([0,\infty);H^{\gamma}(\Sp^{d}))$.   
\end{cor}

\begin{proof}[Proof of \Cref{thm:convergence-arccos}]
By approximation, we may reduce to prove \eqref{eq:conclusion-convergence-arccos} assuming $\bar\mu,\nu \in C^{\infty}(\Sp^{d})$, and so $\mu \in C^{\infty}([0,\infty)\times \Sp^{d})$.

   In the sequel, for convenience, we call $\sigma_{t}:=\mu_{t}-\nu$. Moreover, we introduce the quantity
    \begin{equation*}
        E_{\gamma}(t):= \lVert \mu_{t}\rVert_{\dot H^{\gamma}}^{2}+\mathscr{E}^{\nu}_{\Sp^{d}}(\mu_{t})+\lVert \nu\rVert_{H^{\gamma}}^{2} \qquad \forall t\in [0,\infty).
    \end{equation*}
    Note that by \eqref{eq:comparability-energy-arccos-H^-s}, $E_{\gamma}(t)\approx_{d,\gamma} \lVert \mu_{t}\rVert_{H^{\gamma}}^{2}+\lVert \nu\rVert_{H^{\gamma}}^{2}$.
    
    Consider the second line in the right-hand side of the energy estimate \eqref{eq:energy-estimate-arccos}. Since $\max\{\gamma,\frac{d}{2}+1\}-s \in (-s,\gamma-s+1)$, we may use Sobolev interpolation of $\lVert \sigma_{t}\rVert_{\dot H^{\max\{\gamma,\frac{d}{2}+1\}-s}}$ between $\lVert \sigma_{t}\rVert_{\dot H^{-s}}$ and $\lVert \sigma_{t}\rVert_{\dot H^{\gamma-s+1}}$, Young's inequality, and the lower bound $\nu\ge \alpha>0$ to absorb the positive term inside the negative one, up to a lower order term:
    \begin{equation*}
        -\frac{1}{C}\big(\min_{\T^{d}}\nu\big)\lVert \sigma_{t}\rVert_{\dot H^{\gamma-s+1}}^{2}+ C\lVert \nu\rVert_{H^{\gamma}}\lVert \sigma_{t}\rVert_{\dot H^{\max\{\gamma, \frac{d}{2}+1\}-s}}\lVert \sigma_{t}\rVert_{\dot H^{\gamma-s+1}}
        \le C\lVert \sigma_{t}\rVert_{\dot H^{-s}}^{2}\le C\mathscr{E}^{\nu}_{\Sp^{d}}(\mu_{t}),
    \end{equation*}
    where in the last step we used \eqref{eq:comparability-energy-arccos-H^-s}. 
    Plugging the estimate above in \eqref{eq:energy-estimate-arccos}, and using fact that, by \eqref{eq:energy-dissipation-Fisher-Rao}, $\mathscr{E}^{\nu}_{\Sp^{d}}(\mu_{t})$ is decreasing in time, we deduce
    \begin{equation*}
         \frac{d}{dt}E_{\gamma}(t)\le C\lVert \mathcal{K}(\sigma_{t})\rVert_{C^{2}}E_{\gamma}(t)+ C\mathscr{E}^{\nu}_{\Sp^{d}}(\mu_{t})+ C\lVert \sigma_{t}\rVert_{\dot H^{\gamma+1-2s}}.     
\end{equation*}
which, after integration, yields
    \begin{equation}\label{eq:energy-estimate-simplified-arccos}
        \begin{aligned}
         E_{\gamma}(t)&\le \left(E_{\gamma}(0)+C\int_{0}^{t}\left(\mathscr{E}^{\nu}_{\Sp^{d}}(\mu_{r})+\lVert \sigma_{r}\rVert_{\dot H^{\gamma+1-2s}}\right)dr\right)\exp\left(C\int_{0}^{t}\lVert\mathcal{K}(\sigma_{r})\rVert_{C^{2}}dr\right).       
    \end{aligned}
    \end{equation}
    On the other hand, the energy dissipation identity \eqref{eq:energy-dissipation-Fisher-Rao}, once combined with \eqref{eq:comparability-energy-arccos-H^-s} and Sobolev interpolation of $\lVert \sigma_{t}\rVert_{\dot H^{-s}}$ between $\lVert \sigma_{t}\rVert_{\dot H^{1-2s}}$ and $\lVert \sigma_{t}\rVert_{\dot H^{\gamma}}$ gives
    \begin{equation}\label{eq:energy-dissipation-arccos-interpolated}
    \begin{aligned}
         \frac{d}{dt}\mathscr{E}^{\nu}_{\Sp^{d}}(\mu_{t})&= -4\int_{\Sp^{d}}\left(|\nabla_{\Sp^{d}}\mathcal{K}(\sigma_{t})|^{2}+|\mathcal{K}(\sigma_{t})|^{2}\right)\mu_{t}\\
         &\le -\frac{1}{C}\big(\inf_{\Sp^{d}}\mu_{t}\big)\left(\sigma_{t}(\Sp^{d})^{2}+\lVert \sigma_{t}\rVert_{\dot H^{1-2s}}^{2}\right)\\
         &\le -\frac{1}{C}\big(\inf_{\Sp^{d}}\mu_{t}\big)\left(\sigma_{t}(\Sp^{d})^{2}+\lVert \sigma_{t}\rVert_{\dot H^{\gamma}}^{-\frac{2s-2}{\gamma+s}}\lVert \sigma_{t}\rVert_{\dot H^{-s}}^{2+\frac{2s-2}{\gamma+s}}\right)
         \le -\frac{1}{C}\big(\inf_{\Sp^{d}}\mu_{t}\big)E_{\gamma}(t)^{-\frac{s-1}{\gamma+s}}\mathscr{E}^{\nu}_{\Sp^{d}}(\mu_{t})^{1+\frac{s-1}{\gamma+s}}.
    \end{aligned}
    \end{equation}
    
    Let us call $M:= E_{\gamma}(0)$ and consider  
    \begin{equation*}
        \tau:= \sup\left\{t>0: E_{\gamma}(r)\le 2M \text{\,\, and \,\,} \inf_{\Sp^{d}} \mu_{r}\ge \alpha/2,\,  \forall r\in [0,t)\right\}\in (0,\infty].
    \end{equation*}
    Our goal is to prove that if $\mathscr{E}_{\Sp^{d}}^{\nu}(\bar\mu)$ is sufficiently small, then $\tau=\infty$. 
    By \eqref{eq:energy-dissipation-arccos-interpolated} and the definition of $\tau$ we have 
    \begin{equation*}
        \frac{d}{dt}\mathscr{E}^{\nu}_{\Sp^{d}}(\mu_{t})\le -\kappa\mathscr{E}^{\nu}_{\Sp^{d}}(\mu_{t})^{1+\frac{s-1}{\gamma+s}}\qquad \forall t\in [0,\tau),
    \end{equation*}
    After integration we get
    \begin{equation}\label{eq:polynomial-decay-energy-arccos}
        \mathscr{E}^{\nu}_{\Sp^{d}}(\mu_{t})\le   \mathscr{E}^{\nu}_{\Sp^{d}}(\bar\mu)\left(1+K t\right)^{-\frac{\gamma+s}{s-1}}\quad \forall t\in [0,\tau),\qquad K=C\mathscr{E}^{\nu}_{\Sp^{d}}(\bar\mu)^{\frac{s-1}{\gamma+s}}.
    \end{equation}
    By Sobolev interpolation, we deduce a bound for $\lVert \sigma_{t}\rVert_{\dot H^{\gamma+1-2s}}$:
    \begin{equation}\label{eq:polynomial-decay-Hgamma+1-2s-arccos}
    \begin{aligned}
        \lVert \sigma_{t}\rVert_{\dot H^{\gamma+1-2s}}&\le \lVert \sigma_{t}\rVert_{\dot H^{-s}}^{\frac{2s-1}{\gamma+s}}\lVert \sigma_{t}\rVert_{\dot H^{\gamma}}^{\frac{\gamma+1-s}{\gamma+s}}\\&\le C\mathscr{E}^{\nu}_{\Sp^{d}}(\mu_{t})^{\frac{2s-1}{2(\gamma+s)}}E_{\gamma}(t)^{\frac{\gamma+1-s}{2(\gamma+s)}}\le C\mathscr{E}^{\nu}_{\Sp^{d}}(\bar\mu)^{\frac{2s-1}{2(\gamma+s)}}(1+Kt)^{-\frac{2s-1}{2s-2}}
    \end{aligned}\qquad \forall t\in [0,\tau).
    \end{equation}
    Similarly, we may obtain a bound for $\lVert\mathcal{K}(\sigma_{t})\rVert_{C^{2}}$: Calling $\eta:= \gamma-d/2>0$ and $\hat{\gamma}:= d/2+2-2s+\eta/2$, the boundedness of $\nabla^{2}(-\Delta)^{-s}: \dot H^{\hat{\gamma}}\to L^{\infty}$ together with Sobolev interpolation gives
    \begin{equation}\label{eq:polynomial-decay-Lip-vectorfield-arccos}
        \begin{aligned}
             \lVert \mathcal{K}(\sigma_{t})\rVert_{C^{2}}&\le C\lVert \mathcal{K}(\sigma_{t})\rVert_{H^{\hat{\gamma}+2s}}\\
             &\le  C\mathscr{E}^{\nu}_{\Sp^{d}}(\mu_{t})^{\frac{2s-2+\eta/2}{2(\gamma+s)}}E_{\gamma}(t)^{\frac{\gamma+2-\eta/2-s}{2(\gamma+s)}}\le C \mathscr{E}^{\nu}_{\Sp^{d}}(\bar\mu)^{\frac{2s-2+\eta/2}{2(\gamma+s)}}(1+Kt)^{-\frac{2s-2+\eta/2}{2s-2}}
        \end{aligned}\qquad \forall t\in [0,\tau).
    \end{equation}

    Integrating in the time interval $[0,\tau)$ the bounds from \eqref{eq:polynomial-decay-energy-arccos}, \eqref{eq:polynomial-decay-Hgamma+1-2s-arccos}, and \eqref{eq:polynomial-decay-Lip-vectorfield-arccos}, we obtain the following integral estimates (with constants independent from $\tau$):
    \begin{equation}\label{eq:integral-estimates-convergence-arccos}
        \int_{0}^{\tau}\mathscr{E}^{\nu}_{\Sp^{d}}(\mu_{t})dt\le C\mathscr{E}^{\nu}_{\Sp^{d}}(\bar\mu)^{\frac{\gamma+1}{\gamma+s}},\quad \int_{0}^{\tau}\lVert \sigma_{t}\rVert_{\dot H^{\gamma+1-2s}}dt\le C\mathscr{E}^{\nu}_{\Sp^{d}}(\bar\mu)^{\frac{1}{2(\gamma+s)}},\quad \int_{0}^{\tau}\lVert \mathcal{K}(\sigma_{t})\rVert_{C^{2}}dt\le C\mathscr{E}^{\nu}_{\Sp^{d}}(\bar\mu)^{\frac{\gamma-d/2}{4(\gamma+s)}}.
    \end{equation}
    From here, we deduce a uniform lower bound on $\mu_{t}$. Indeed, since $\inf \bar\mu \ge \alpha$, by the Lagrangian representation formula for the solution $\mu$ obtained in \Cref{prop:well-posedness-WFR}, we get
    \begin{equation}\label{eq:uniform-lower-bound-local-convergence-arccos}
    \begin{aligned}
        \inf_{\Sp^{d}} \mu_{t}&\ge \big(\inf_{\Sp^{d}} \bar \mu\big)\, \exp\left(-C\int_{0}^{t}\lVert \mathcal{K}(\sigma_{r})\rVert_{C^{2}}dr\right) \ge \alpha \exp \left(-C\mathscr{E}^{\nu}_{\Sp^{d}}(\bar\mu)^{\frac{\gamma-d/2}{4(\gamma+s)}}\right)
    \end{aligned}\qquad \forall t\in [0,\tau).
    \end{equation}
    
    We now conclude the proof by showing that if $\mathscr{E}^{\nu}_{\Sp^{d}}(\bar\mu)$ is sufficiently small, then $\tau =\infty$. This combined with \eqref{eq:polynomial-decay-energy-arccos} shows that the estimates in \eqref{eq:conclusion-convergence-arccos} hold for all times $t\in [0,\infty)$. Suppose by contradiction that $\tau<\infty$. Thanks to \eqref{eq:uniform-lower-bound-local-convergence-arccos}, by choosing $\mathscr{E}^{\nu}_{\Sp^{d}}(\bar\mu)$ small enough, we can make sure that $\inf_{\Sp^{d}} \mu_{t}\ge 3\alpha/4$ for every $t\in [0,\tau)$, so that necessarily $E_{\gamma}(\tau)=2M$. However, \eqref{eq:energy-estimate-simplified-arccos}, together with \eqref{eq:integral-estimates-convergence-arccos}, gives
    \begin{align*}
        E_{\gamma}(t)\le \left(M+ C\left(\mathscr{E}^{\nu}_{\Sp^{d}}(\bar\mu)^{\frac{\gamma+1}{\gamma+s}}+\mathscr{E}^{\nu}_{\Sp^{d}}(\bar\mu)^{\frac{1}{2(\gamma+s)}}\right) \right)\exp \left(C\mathscr{E}^{\nu}_{\Sp^{d}}(\bar\mu)^{\frac{\gamma-d/2}{4(\gamma+s)}}\right)\le \frac{3}{2}M,
    \end{align*}
    provided that $\mathscr{E}^{\nu}_{\Sp^{d}}(\bar\mu)$ is chosen sufficiently small. This is a contradiction and concludes the proof.
\end{proof}

\begin{rmk}\label{rmk:improved-regularity-target-convergence-arccos}
    As in \Cref{rmk:improved-regularity-target-convergence}, the assumption $\nu\in H^{\gamma+1}(\Sph^d)$ is technical. 
Working with fractional Sobolev regularity in the neural-network setting, one can relax it to $\nu\in H^{\gamma+\eps}(\Sph^d)$ for any fixed $\eps>0$, 
by applying a fractional Leibniz rule on $\Sph^d$ and redistributing $\eps$ derivatives in \eqref{eq:ee-nn-I23} onto the other two factors (instead of a full derivative); 
see, e.g., \cite[Corollary 1.2]{bui2025bilinear}. \fr
\end{rmk}

\begin{proof}[Proof of \Cref{cor:fnu}]
    Let $\gamma:=\beta-s-1 \ge s-1$, and consider the multiplier operator $\mathcal{H}$ defined \eqref{eq:def-repres-operator-H} and studied in \Cref{lem:ReLU-operator}. We have $\nu=\mathcal{H}^{-1}(g+S)=\mathcal{H}^{-1}(g)+c(d)S$, for some positive constant $c(d)>0$. Therefore, by Sobolev embedding,
    \begin{equation*}
        \min_{\Sp^{d}}\nu \ge c(d)S-C(d,\beta)\lVert\mathcal{H}^{-1}(g)\rVert_{H^{\gamma+1}}\ge c(d)K(1+\lVert g\rVert_{H^{\beta}})-C(d,\beta)\lVert g\rVert_{H^{\beta}}\ge 2
    \end{equation*}
    provided we choose $K=K(d,\beta)>0$ sufficiently large. With this choice of $K$, we have $\lVert \nu\rVert_{H^{\gamma+1}}\lesssim_{d,\beta} \lVert g\rVert_{H^{\beta}}+1$. 
    
    Next, note that $\mathscr{E}^{\nu}_{\Sp^{d}}(\bar\mu)=\lVert f-f_{\bar\mu}\rVert_{L^{2}}^{2}\le \delta$. Hence, calling $\eta:= \gamma-d/2>0$, by Sobolev embedding and interpolation,
    \begin{align*}
        \lVert \bar\mu-\nu\rVert_{L^{\infty}}&\le C(d,\beta)\lVert \bar\mu-\nu\rVert_{H^{\frac{d+\eta}{2}}}\\
        &\le C(d,\beta)\mathscr{E}^{\nu}_{\Sp^{d}}(\bar\mu)^{\frac{2\gamma-d-\eta}{4(\gamma+s)}}\lVert \bar\mu-\nu\rVert_{H^{\gamma}}^{\frac{d+2s+\eta}{4(\gamma+s)}}\le C(d,\beta)\delta^{\frac{2\gamma-d-\eta}{4(\gamma+s)}}\left(M+\lVert g\rVert_{H^{\beta}}\right)^{\frac{d+2s+\eta}{4(\gamma+s)}}\le 1,
    \end{align*}
    provided we choose $\delta$ small enough depending only on $d, \beta, M$, and $\lVert g\rVert_{H^{\beta}}$. With this choice, $\bar\mu \ge 1$. At this point, the corollary follows from a direct application of \Cref{thm:convergence-arccos}, noting that $\lVert f_{\mu_{t}}-f\rVert_{L^{2}}^{2}=\mathscr{E}^{\nu}_{\Sp^{d}}(\mu_{t})$, and $\lVert f_{\mu_{t}}\rVert_{H^{\beta-1}}\approx_{d,\beta} \lVert \mu_{t}\rVert_{H^{\gamma}}$. 
\end{proof}

\addcontentsline{toc}{section}{Appendix}

\appendix
\section*{Appendix}

\refstepcounter{section}

The appendix is divided in two parts. In \Cref{subsec-appendix-kernels} we collect several useful estimates concerning the Riesz kernels $K_{s}$, for $s\ge 1$. Then, in \Cref{subsec:appendix-katoponce} we extend some fractional Kato--Ponce commutator estimates to the periodic setting. 

\subsection{Riesz Kernel estimates}\label{subsec-appendix-kernels} 
Here, for any $s\ge 1$, $K_{s}:\T^{d}\to \R$ is the periodic Riesz kernel defined in \eqref{eq:def-Riesz-kernel}.
We start with the following lemma, that gives the precise asymptotic behavior of $K_{s}$ in $0$ (see, for instance, \cite[Theorem 3.4]{fahrenwaldt2017off} for a much more general result). 
\begin{lem}\label{lem:asymptotic-behavior-kernels}
    Let $s\ge 1$ and $K_{s}$ be the kernel defined in \eqref{eq:def-Riesz-kernel}.
    Let $\varphi \in C^{\infty}(\T^{d};[0,1])$ be a cut-off function such that $\varphi=1$ in $B_{1/8}$ and $\varphi=0$ in $B_{1/4}$. Then we have
    \begin{equation}\label{eq:asymptotic-behavior-kernel}
        K_{s}= c_{d,s}K_{s}^{{\rm sing}}\varphi+K_{s}^{{\rm reg}},\qquad K_{s}^{{\rm sing}}(x):=\begin{cases}
            |x|^{2s-d}\qquad &\text{for $2s-d \not\in 2\N\cup\{0\}$},\\
            |x|^{2s-d}\log(|x|)\qquad &\text{for $2s-d\in 2\N\cup\{0\}$},
        \end{cases}\qquad K_{s}^{{\rm reg}}\in C^{\infty}(\T^{d}).
    \end{equation}
    In particular, the following bounds hold:
    \begin{equation}\label{eq:upper-bounds-derivatives-kernel}
        |\nabla K_{s}(x)|\lesssim_{d,s} \begin{cases}
            |x|^{2s-d-1}\quad &\text{for $s\in \left[1,\frac{d}{2}+\frac{1}{2}\right)$},\\
            1\quad &\text{for $s\ge \frac{d}{2}+\frac{1}{2}$},
        \end{cases}\qquad |\nabla^{2}K_{s}(x)|\lesssim_{d,s} \begin{cases}
            |x|^{2s-d-2}\quad &\text{for $s\in \left[1,\frac{d}{2}+1\right)$},\\
            \log(2+|x|^{-1})\quad &\text{for $s=\frac{d}{2}+1$},\\
            1\quad &\text{for $s> \frac{d}{2}+1$}.
        \end{cases}
    \end{equation}
\end{lem}

Next, we give a bound on the modulus of continuity of $\nabla K_{s}*\eta$, for $\eta$ belonging to the space $\mathscr{X}_{s}(\T^{d})$ defined in \eqref{eq:def-yudovich-class}.
\begin{lem}\label{lem:modulus-continuity-vector-field}
    Let $s\ge 1$ and $\mathscr{X}_{s}(\T^{d})$ be defined in \eqref{eq:def-yudovich-class}. Then, for every $\eta\in \mathscr{X}_{s}(\T^{d})$, we have
    \begin{equation}\label{eq:modulus-continuity-vector-field}
        |\nabla K_{s}*\eta(x)-\nabla K_{s}*\eta(y)|\lesssim_{d,s}\lVert \eta\rVert_{\mathscr{X}_{s}}\omega_{s}(|x-y|)\qquad \forall x,y\in \T^{d}.
    \end{equation}
    where the modulus of continuity $\omega_{s}:[0,\infty)\to [0,\infty)$ is given by
    \begin{equation}\label{eq:def:mod-continuity-vect-field}
        \omega_{s}(r):=\begin{cases}
            r\qquad &\text{if $s\not \in \{1,d/2+1\}$},\\
            r\log(2+r^{-1})\qquad &\text{if $s\in \{1,d/2+1\}$}.
        \end{cases}
    \end{equation}
\end{lem}
\begin{proof}We divide the proof according to the choice of $s\ge 1$.    

    \smallskip
    \noindent \textbf{$\bm{(s=1)}$.} We take $x,y\in \T^{d}$, we call $r:=|x-y|$, and we assume without loss of generality that $r<1/8$. Then, 
    \begin{equation*}
        |\nabla K_{1}*\eta(x)-\nabla K_{1}*\eta(y)|\le \lVert \eta\rVert_{L^{\infty}}\int_{\T^{d}}|\nabla K_{1}(x-z)-\nabla K_{1}(y-z)|dz.
    \end{equation*}
    We split the integral in the right-hand side in the two regions $B_{2r}(x)$ and $\T^{d}\setminus B_{2r}(x)$. For the first term, we have
    \begin{align*}
        \int_{B_{2r}(x)}|\nabla K_{1}(x-z)-\nabla K_{1}(y-z)|dz&\le \int_{B_{2r}(x)}|\nabla K_{1}(x-z)|dz+\int_{B_{3r}(x)}|\nabla K_{1}(y-z)|dz\\
        &\lesssim_{d} \int_{B_{3r}}|w|^{1-d}\lesssim_{d} r,
    \end{align*}
    where we used the bound on $|\nabla K_{s}|$ from \eqref{eq:upper-bounds-derivatives-kernel}.
    On the other hand, observe that for every $z\in \T^{d}\setminus B_{2r}(x)$, by the mean value theorem there exists $p_{z}$ in the segment connecting $x$ to $y$ such that 
    \begin{equation*}
        |\nabla K_{1}(x-z)-\nabla K_{1}(y-z)|\le |\nabla^{2}K_{1}(p_{z}-z)||x-y|\lesssim_{d} r|p_{z}-z|^{-d}\lesssim r|x-z|^{-d},
    \end{equation*}
    where we used the bound on $|\nabla^{2}K_{s}|$ from \eqref{eq:upper-bounds-derivatives-kernel}. Therefore: 
    \begin{align*}
         \int_{\T^{d}\setminus B_{2r}(x)}|\nabla K_{1}(x-z)-\nabla K_{1}(y-z)|dz&\lesssim_{d} r\int_{\T^{d}\setminus B_{2r}}|w|^{-d}\lesssim_{d} r\log(2+r^{-1}).
    \end{align*}
    Collecting the two estimates we get precisely \eqref{eq:modulus-continuity-vector-field} in the case $s=1$.
    
    \smallskip
    \noindent \textbf{$\bm{(1<s<d/2+1)}$.} Let $p=d/(2s-2)$ and $q=p'=d/(d-2s+2)$. By H\"{o}lder inequality in Lorentz spaces,
    \begin{align*}
        \lVert \nabla^{2}K_{s}*\eta\rVert_{L^{\infty}}\lesssim_{d,s} \lVert \nabla^{2}K_{s}\rVert_{L^{q,\infty}}\lVert \eta\rVert_{L^{p,1}}\lesssim_{d,s}\lVert \eta\rVert_{L^{p,1}},
    \end{align*}
    where we used the bound $|\nabla^{2}K_{s}(x)|\lesssim_{d,s}|x|^{2s-d-2} \in L^{q,\infty}$ from \eqref{eq:upper-bounds-derivatives-kernel}.
    
    \smallskip
    \noindent \textbf{$\bm{(s=d/2+1)}$.} In this case, \eqref{eq:asymptotic-behavior-kernel} gives
    \begin{equation*}
        |\nabla K_{d/2+1}(x)-\nabla K_{d/2+1}(y)|\lesssim_{d}|x-y|\log(2+|x-y|^{-1})\qquad \forall x,y\in \T^{d}. 
    \end{equation*}
    Therefore, 
    \begin{equation*}
        |\nabla K_{d/2+1}*\eta(x)-\nabla K_{d/2+1}*\eta(y)|\lesssim_{d}\lVert \eta\rVert_{\mathcal{M}}|x-y|\log(2+|x-y|^{-1}).
    \end{equation*}
    
    \smallskip
    \noindent \textbf{$\bm{(s>d/2+1)}$.} By \eqref{eq:upper-bounds-derivatives-kernel}, it holds $\lVert \nabla^{2}K_{s}*\eta\rVert_{L^{\infty}}\le \lVert \nabla^{2}K_{s}\rVert_{L^{\infty}}\lVert \eta\rVert_{\mathcal{M}}\lesssim_{d,s}\lVert \eta\rVert_{\mathcal{M}}$.
\end{proof}
In the following two lemmas, we prove bounds on $\nabla K_{s}*(\eta_{1}-\eta_{2})$ in Lebesgue spaces in terms of the Wasserstein distance $W_{2}(\eta_{1},\eta_{2})$ and a suitable integrability control on $\eta_{1}$ and $\eta_{2}$. 
\begin{lem}\label{lem:W2-dual}
    Let $s\in \left[1,d/2+1\right)$, $p:=d/(2s-2)\in (1,\infty]$, $q:=p'=d/(d-2s+2)\in [1,\infty)$, and $\eta_{1},\eta_{2}\in \mathscr{P}\cap L^{p}(\T^{d})$. Then, the following estimate holds:
        \begin{equation*}
            \lVert \nabla K_{s}*(\eta_{1}-\eta_{2})\rVert_{L^{2q}}\lesssim_{d,s} \max\{\lVert \eta_{1}\rVert_{L^{p}},\lVert \eta_{2}\rVert_{L^{p}}\}^{1/2} W_{2}(\eta_{1},\eta_{2}).
        \end{equation*}
    \end{lem}
\begin{proof}
    Since $\eta_{1},\eta_{2}\ll \mathscr{L}^{d}$, there exists a unique optimal transport map $T:\T^{d}\to \T^{d}$ with respect to the quadratic cost. Call $\xi:\T^{d}\to \R^{d}$ the vector field given by the displacement $\xi(x)=T(x)-x$ and let $X:[1,2]\times \T^{d}\to \T^{d}$ be the associated flow map $X_{t}(x)=(2-t)x +(t-1)T(x)$, such that $X_{1}=\id$ and $X_{2}=T$. We thus define the displacement interpolation $$\eta_{t}:= (X_{t})_{\#}\eta_{1}\qquad \forall t\in [1,2].$$
    As it is well-known, $\eta \in AC([1,2];\mathscr{P}(\T^{d}))$ is the unique constant speed geodesic connecting $\eta_{1}$ to $\eta_{2}$ in the Wasserstein space, and it satisfies the continuity equation
    \begin{equation}\label{eq:cont-equation-displacement-interpolation}
        \partial_{t}\eta_{t}+\diver(\eta_{t}w_{t})=0,\qquad w_{t}=\xi \circ X_{t}^{-1},\qquad \lVert w_{t}\rVert_{L^{2}(\T^{d}, \eta_{t};\R^{d})}= W_{2}(\eta_{1},\eta_{2}).
    \end{equation}
    Moreover, the $L^{p}$-norm of the interpolation $\eta_{t}$ is bounded by the norms of the extrema $\eta_{1},\eta_{2}$:
    \begin{equation}\label{eq:bound-norm-interpolation-displacement}
        \lVert\eta_{t}\rVert_{L^{p}}\le \max\{\lVert\eta_{1}\rVert_{L^{p}},\lVert\eta_{2}\rVert_{L^{p}}\}\qquad \forall t\in [1,2].
    \end{equation}
    We can write, in the sense of distributions,
    \begin{align}\label{eq:rewriting-vector-field-geodesics}
        \nabla K_{s}*(\eta_{1}-\eta_{2})=\nabla K_{s}*\int_{1}^{2}\partial_{t}\eta_{t}dt= -\int_{1}^{2}\nabla K_{s}*\diver(\eta_{t}w_{t})dt=\int_{1}^{2}\nabla^{2} K_{s}*(\eta_{t}w_{t})dt.
    \end{align}
    Now, let $\beta:= 2q/(2q-1)\in (1,2]$ be such that $1/\beta=1/2p+1/2$. By H\"{o}lder inequality along with \eqref{eq:cont-equation-displacement-interpolation} and \eqref{eq:bound-norm-interpolation-displacement} we find
    \begin{equation*}
        \lVert \eta_{r}w_{r}\rVert_{L^{\beta}}\le \lVert \eta_{r}\rVert_{L^{p}}^{1/2}\lVert \eta_{r}^{1/2}w_{r}\rVert_{L^{2}}\le \max\{\lVert\eta_{1}\rVert_{L^{p}},\lVert\eta_{2}\rVert_{L^{p}}\}^{1/2}W_{2}(\eta_{1},\eta_{2})\qquad \forall r\in [1,2].
    \end{equation*}
    Then, by \eqref{eq:rewriting-vector-field-geodesics}, 
    \begin{align*}
        \lVert \nabla K_{s}*(\eta_{1}-\eta_{2})\rVert_{L^{2q}}&\le \int_{1}^{2}\lVert \nabla^{2}K_{s}*(\eta_{r}w_{r})\rVert_{L^{2q}}dr\\
        &\lesssim_{d,s} \int_{1}^{2}\lVert \eta_{r}w_{r}\rVert_{L^{\beta}}dr\le \max\{\lVert\eta_{1}\rVert_{L^{p}},\lVert\eta_{2}\rVert_{L^{p}}\}^{1/2}W_{2}(\eta_{1},\eta_{2}),
    \end{align*}
    where in the second inequality we used $\nabla^{2}K_{s}*=(-\Delta)^{1-s}\nabla^{2}K_{1}*$, and we combined the Calder\'on-Zygmund estimate $\nabla^{2}K_{1}*:L^{2q}\to L^{2q}$ with the sobolev embedding for zero averaged functions $(-\Delta)^{1-s}:L^{2q}\to L^{\beta}$. This is the desired inequality and the proof is concluded.
\end{proof}

\begin{lem}\label{lem:W2-Linfty}
    Let $s\ge 1$, and $\eta_{1},\eta_{2}\in \mathscr{P}(\T^{d})$. Then, the following estimates hold:
    \begin{itemize}
        \item [i)] If $s\in \left[1,d/2+1\right)$, $p=d/(2s-2)$ and $\eta_{1},\eta_{2} \in L^{p}(\T^{d})$:
        \begin{equation*}
            \lVert \nabla K_{s}*(\eta^{1}-\eta^{2})\rVert_{L^{\infty}}\lesssim_{d,s}\max\{\lVert \eta_{1}\rVert_{L^{p}},\lVert \eta_{2}\rVert_{L^{p}}\}^{\frac{d+2-2s}{d+3-2s}}W_{2}(\eta^{1},\eta^{2})^{\frac{1}{d+3-2s}}.
        \end{equation*}
        \item [ii)] If $s=d/2+1$: 
        \begin{equation*}
            \lVert \nabla K_{s}*(\eta_{1}-\eta_{2})\rVert_{L^{\infty}}\lesssim_{d}W_{2}(\eta_{1},\eta_{2})\log\left(2+W_{2}(\eta_{1},\eta_{2})^{-1}\right).
        \end{equation*}
        \item [iii)] If $s>d/2+1$:
        \begin{equation*}
             \lVert \nabla K_{s}*(\eta_{1}-\eta_{2})\rVert_{L^{\infty}}\lesssim_{d,s}W_{2}(\eta_{1},\eta_{2}).
        \end{equation*}
    \end{itemize}
\end{lem}
\begin{proof}
    We split the proof according to the choice of $s\ge 1$.

    \smallskip
    \noindent \textbf{$\bm{(1\le s<d/2+1)}$.} Let $\psi\in C^{\infty}_{c}(B_{1/4};[0,1])$ be a cut-off function such that $\psi=1$ in $B_{1/8}$, and let $\eps\in (0,1)$ be some parameter to be chosen later. For a fixed $x\in \T^{d}$, we can split
    \begin{align*}
        \nabla K_{s}*(\eta_{1}-\eta_{2})(x)&= \int_{\T^{d}}\nabla K_{s}(x-y)(\eta_{1}-\eta_{2})(y)\psi\left(\eps^{-1}(x-y)\right)\\
        &\quad\,+\int_{\T^{d}}\nabla K_{s}(x-y)(\eta_{1}-\eta_{2})(y)\left(1-\psi\left(\eps^{-1}(x-y)\right)\right)= (I)+(II).
    \end{align*}
    Being $p=d/(2s-2)$ and $q=p'=d/(d-2s+2)$,
    for the first term we use H\"{o}lder inequality and the bound on $|\nabla K_{s}|$ from \eqref{eq:upper-bounds-derivatives-kernel}:
    \begin{align*}
        |(I)|&\le \int_{B_{\eps/4}(x)}|\nabla K_{s}(x-y)||(\eta_{1}-\eta_{2})(y)|dy
        \lesssim  \max\{\lVert \eta_{1}\rVert_{L^{p}},\lVert \eta_{2}\rVert_{L^{p}}\}\left(\int_{B_{\eps/4}}|\nabla K_{s}|^{q}\right)^{1/q}\\&\lesssim_{d,s} \max\{\lVert \eta_{1}\rVert_{L^{p}},\lVert \eta_{2}\rVert_{L^{p}}\}\left(\int_{B_{\eps/4}}|x|^{q(2s-d-1)}\right)^{1/q}\lesssim_{d,s}\max\{\lVert \eta_{1}\rVert_{L^{p}},\lVert \eta_{2}\rVert_{L^{p}}\} \eps,
    \end{align*}
    To bound the second term, observe that the function $\nabla K_{s}(x-\cdot)\left(1-\psi\left(\eps^{-1}(\cdot-x)\right)\right)$ has Lipschitz constant $\lesssim_{d,s} \eps^{2s-d-2}$ because of \eqref{eq:upper-bounds-derivatives-kernel}. Therefore:
    \begin{equation*}
        |(II)|\lesssim_{d,s} \eps^{2s-d-2}\sup_{\Lip(f)\le 1}\int_{\T^{d}}f(\eta_{1}-\eta_{2})= \eps^{2s-d-2}W_{1}(\eta_{1},\eta_{2})\le \eps^{2s-d-2}W_{2}(\eta_{1},\eta_{2}),
    \end{equation*}
    where we used the well-known Kantorovich-Rubinstein duality formula for the $1$-Wasserstein distance $W_{1}$, along with the trivial estimate $W_{1}\le W_{2}$. 
    Choosing $c=c(d)$ so that $$\eps= c\left(\frac{W_{2}(\eta_{1},\eta_{2})}{\max\{\lVert \eta_{1}\rVert_{L^{p}},\lVert \eta_{2}\rVert_{L^{p}}\}}\right)^{\frac{1}{d+3-2s}}\in (0,1),$$
    and gathering the estimates for $(I)$ and $(II)$, we obtain
    \begin{align*}
        \lVert \nabla K_{s}*(\eta_{1}-\eta_{2})\rVert_{L^{\infty}}&\lesssim_{d,s} \eps \max\{\lVert \eta_{1}\rVert_{L^{p}},\lVert \eta_{2}\rVert_{L^{p}}\}+\eps^{2s-d-2}W_{2}(\eta_{1},\eta_{2})\\&\lesssim_{d} \max\{\lVert \eta_{1}\rVert_{L^{p}},\lVert \eta_{2}\rVert_{L^{p}}\}^{\frac{d+2-2s}{d+3-2s}}W_{2}(\eta_{1},\eta_{2})^{\frac{1}{d+3-2s}},
    \end{align*}
    as desired.

    \smallskip
    \noindent \textbf{$\bm{(s=d/2+1)}$.} The same argument as above with the choice $\eps=cW_{2}(\eta_{1},\eta_{2})\log\left(2+W_{2}(\eta_{1},\eta_{2})^{-1}\right)\in (0,1)$ yields
    \begin{align*}
        \lVert \nabla K_{s}*(\eta_{1}-\eta_{2})\rVert_{L^{\infty}}&\lesssim_{d,s} \eps +\log(2+\eps^{-1})W_{2}(\eta_{1},\eta_{2})\\&\lesssim_{d} W_{2}(\eta_{1},\eta_{2})\log\left(2+W_{2}(\eta_{1},\eta_{2})^{-1}\right).
    \end{align*}
    \smallskip
    \noindent \textbf{$\bm{(s>d/2+1)}$.} In this case $\lVert\nabla^{2} K_{s}\rVert_{L^{\infty}}\lesssim_{d,s}1$. Therefore, for any $x\in \T^{d}$,
    \begin{equation*}
        |\nabla K_{s}*(\eta_{1}-\eta_{2})(x)|=\left|\int_{\T^{d}}\nabla K_{s}(x-y)(\eta_{1}-\eta_{2})dy\right|\lesssim_{d,s}\sup_{\Lip(f)\le 1}\int f(\eta_{1}-\eta_{2})\le W_{2}(\eta_{1},\eta_{2}),
    \end{equation*}
    as desired.    
\end{proof}

Finally, we give $L^{\infty}$ bounds on $\nabla^{2}K_{s}*f$ in critical spaces with the help of a logarithm of a slightly stronger norm.  
\begin{lem}\label{lem:log-interp-BKM-criterion}
    Let $s\ge 1$. The following hold:
    \begin{itemize}
        \item [i)] If $s=1$, $\alpha\in (0,1)$, and $f\in C^{0,\alpha}(\T^{d})$, then 
        \begin{equation*}
            \lVert \nabla^{2}K_{1}*f\rVert_{L^{\infty}}\lesssim_{d} \lVert f \rVert_{L^{\infty}}\left(1+\log\left(1+\frac{[f]_{C^{0,\alpha}}}{\lVert f \rVert_{L^{\infty}}}\right)\right).
        \end{equation*}
        \item [ii)] If $s\in \left(1,d/2+1\right]$, $p= d/(2s-2)\in [1,\infty)$, and $f\in L^{r}(\T^{d})$ for some $r>p$ then
        \begin{equation*}
            \lVert \nabla^{2}K_{s}*f\rVert_{L^{\infty}}\lesssim_{d,s,r} \lVert f \rVert_{L^{p}}\left(1+\log\left(1+\frac{\lVert f\rVert_{L^{r}}}{\lVert f \rVert_{L^{p}}}\right)\right).
        \end{equation*}
    \end{itemize}
\end{lem}
\begin{proof}
    We first consider the case $s=1$. Let $\psi\in C^{\infty}_{c}(B_{1/4};[0,1])$ be a radial cut-off function such that $\psi=1$ in $B_{1/8}$, and let $\eps\in (0,1)$ be some parameter to be chosen later. For a given $x\in \T^{d}$, we split
    \begin{equation*}
        \nabla^{2}K_{1}*f(x)=\underbrace{\int_{B_{\eps/4}(x)}\nabla^{2}K_{1}(x-y)\psi((y-x)/\eps)f(y)}_{(I)}+\underbrace{\int_{\T^{d}\setminus B_{\eps/4}(x)}\nabla^{2}K_{1}(x-y)\left(1-\psi((y-x)/\eps)\right)f(y)}_{(II)}.
    \end{equation*}
    By \eqref{eq:upper-bounds-derivatives-kernel}, we can estimate $(II)$ with
    \begin{equation*}
        |(II)|\lesssim_{d} \lVert f\rVert_{L^{\infty}}\int_{\T^{d}\setminus B_{\eps/4}(x)}|y-x|^{-d}\lesssim_{d}\lVert f\rVert_{L^{\infty}} \log(1+\eps^{-1}).
    \end{equation*}
    To bound $(I)$, note that by symmetry, $\nabla^{2}K_{1}$ is zero averaged in $\partial B_{t}$, for all $t\in (0,1/4)$. Since $\psi$ is radial, we deduce that $\nabla^{2}K_{1}(x-\cdot)\psi((\cdot-x)/\eps)$ is zero averaged in $B_{\eps/4}(x)$. Therefore, using \eqref{eq:upper-bounds-derivatives-kernel}, we may bound
    \begin{align*}
        |(I)|=\left|\int_{B_{\eps/4}(x)}\nabla^{2}K_{1}(x-y)\psi((y-x)/\eps)(f(y)-f(x))\right|\lesssim_{d}[f]_{C^{0,\alpha}}\int_{B_{\eps/4}(x)}|y-x|^{\alpha-d}\lesssim_{d,\alpha}[f]_{C^{0,\alpha}}\eps^{\alpha}.
    \end{align*}
    Combining the two estimates and choosing $\eps=\min\{1/2, (\lVert f\rVert_{L^{\infty}}/[f]_{C^{0,\alpha}})^{1/\alpha}\}$ we derive the desired estimate. 

    The case $s\in (1,d/2+1]$ can be handled in the same way, using H\"{o}lder inequality and \eqref{eq:upper-bounds-derivatives-kernel} to obtain analogous estimates of the two terms $(I), (II)$ above, with $\lVert f\rVert_{L^{p}}$ in the place of $\lVert f\rVert_{L^{\infty}}$ and $\lVert f\rVert_{L^{r}}$ in the place of $[f]_{C^{0,\alpha}}$, where $p=d/(2s-2)$ and $r>p$.
\end{proof}
\subsection{Kato--Ponce Commutator estimates on the $d$-dimensional torus}\label{subsec:appendix-katoponce}
In this section, we extend some Kato--Ponce fractional commutator estimates (\cite{kato1988commutator, li2019kato}) to the periodic setting. We adopt the strategy proposed in \cite{benyi2025fractional}, which consists of reducing to the corresponding estimate in $\R^{d}$ by means of Poisson's summation formula\footnote{For every Schwartz function $u\in \mathcal{S}(\R^{d})$: $$\sum_{k\in \Z^{d}}\mathcal{F}_{\R^{d}}(u)(k)e^{2\pi i k\cdot x}=\sum_{k\in \Z^{d}}u(x+k)\qquad \forall x\in \R^{d}.$$}. In the following, for any $\beta\in \R$, we denote $D^{\beta}_{\T^{d}}:=(-\Delta_{\T^{d}})^{\beta/2}$. 

\begin{lem}\label{lem:kato-ponce}
    Let $\gamma>0$ and $2\le p_{1},q_{1},p_{2},q_{2}\le \infty$ be such that $1/2=1/p_{i}+1/q_{i}$, $i=1,2$. Then, for every $f,g \in C^{\infty}(\T^{d})$ we have
    \begin{equation}\label{eq:kato-ponce}
        \lVert D_{\T^{d}}^{\gamma}(fg)-fD_{\T^{d}}^{\gamma}g\rVert_{L^{2}(\T^{d})}\lesssim_{d,\gamma,p_{1},q_{1},p_{2},q_{2}} \lVert \nabla f\rVert_{L^{p_{1}}(\T^{d})}\lVert D_{\T^{d}}^{\gamma-1}g\rVert_{L^{q_{1}}(\T^{d})}+\lVert D_{\T^{d}}^{\gamma}f\rVert_{L^{p_{2}}(\T^{d})}\rVert g\rVert_{L^{q_{2}}(\T^{d})}.
    \end{equation}
\end{lem}

\begin{proof}
    We will prove the result by reducing to the corresponding inequality in $\R^{d}$, which is proved in \cite[Corollary 5.2]{li2019kato}. Without loss of generality, we can assume that $f,g\in C^{\infty}(\T^{d})$ have zero average in $\T^{d}$. We call $F, G\in C^{\infty}(\R^{d})$ the periodic extensions of $f,g$ respectively. We consider a periodic partition of unity, namely a collection of $1$-periodic functions $\varphi_{1},\dots,\varphi_{N}\in C^{\infty}(\R^{d};[0,1])$ such that
    \begin{equation*}
        \sum_{j=1}^{N}\varphi_{j}(x)=1\quad \forall x\in \R^{d},\qquad \supp \varphi_{j}\subset \bigcup_{k\in \Z^{d}}\left\{Q_{1/2}(\bar x_{j})+k\right\}\quad \forall j\in \{1,\dots,N\}, 
    \end{equation*}
    where $Q_{1/2}(\bar x_{j})$ is the open cube with center $\bar x_{j}\in [0,1]^{d}$ and side length $1/2$. We denote by $G_{j}$ the one-period localization of $G\varphi_{j}$:
    $$G_{j}:= G\varphi_{j}\mathbbm{1}_{Q_{1}(\bar x_{j})}\qquad \forall j\in \{1,\dots,N\}.$$
    For every $k\in \Z^{d}$, we can write
    \begin{align*}
        \mathcal{F}_{\T^{d}}(g)(k)&=\int_{\T^{d}}e^{-2\pi i k\cdot x}g(x)=\int_{[0,1]^{d}}e^{-2\pi ik\cdot x}G(x)\\
        &=\sum_{j=1}^{N}\int_{[0,1]^{d}}e^{-2\pi i k\cdot x}G(x)\varphi_{j}(x)=\sum_{j=1}^{N}\int_{Q_{1}(\bar x_{j})}e^{-2\pi i k\cdot x}G(x)\varphi_{j}(x)= \sum_{j=1}^{N}\mathcal{F}_{\R^{d}}(G_{j})(k),
    \end{align*}
    where $\mathcal{F}_{\T^{d}}$ and $\mathcal{F}_{\R^{d}}$ denote Fourier transform in $\T^{d}$ and $\R^{d}$, respectively.
    Similarly,
    $$\mathcal{F}_{\T^{d}}(fg)(k)=\sum_{j=1}^{N}\mathcal{F}_{\R^{d}}(FG_{j})(k).$$
    In particular, the commutator $[D^{\gamma}_{\T^{d}}, f](g)(x)$ can be written as follows:
    \begin{align*}
        D^{\gamma}_{\T^{d}}(fg)(x)-fD^{\gamma}_{\T^{d}}g(x)&= \sum_{k\in \Z^{d}}(2\pi |k|)^{\gamma}\Big[\mathcal{F}_{\T^{d}}(fg)(k)-f(x)\mathcal{F}_{\T^{d}}(g)(k)\Big]e^{2\pi i k\cdot x}\\
        &=\sum_{j=1}^{N}\sum_{k\in \Z^{d}}(2\pi |k|)^{\gamma}\Big[\mathcal{F}_{\R^{d}}(FG_{j})(k)-F(x)\mathcal{F}_{\R^{d}}(G_{j})(k)\Big]e^{2\pi i k\cdot x}\\
        &=\sum_{j=1}^{N}\sum_{k\in \Z^{d}}\Big[\mathcal{F}_{\R^{d}}\left(D^{\gamma}_{\R^{d}}(FG_{j})\right)(k)-F(x)\mathcal{F}_{\R^{d}}\left(D^{\gamma}_{\R^{d}}G_{j}\right)\Big]e^{2\pi i k\cdot x}\\
        &= \sum_{j=1}^{N}\sum_{k\in \Z^{d}}\Big[D^{\gamma}_{\R^{d}}(FG_{j})-FD^{\gamma}_{\R^{d}}G_{j}\Big](x+k),
    \end{align*}
    where $D^{\gamma}_{\R^{d}}:=(-\Delta_{\R^{d}})^{\gamma/2}$, and we used Poisson's summation formula in the last step.
    In particular, we can bound
    \begin{align*}
        \lVert D^{\gamma}_{\T^{d}}(fg)-fD^{\gamma}_{\T^{d}}g\rVert_{L^{2}(\T^{d})}&\lesssim_{d} \sum_{j=1}^{N}\Big\lVert \sum_{k\in \Z^{d}}\big[D^{\gamma}_{\R^{d}}(FG_{j})-FD^{\gamma}_{\R^{d}}G_{j}\big](\cdot+k)\Big\rVert_{L^{2}([0,1]^{d})}\\
        &\lesssim_{d} \sum_{j=1}^{N}\sum_{\substack{k\in \Z^{d}\\ |k|\le 2\sqrt{d}} }\left\lVert \big[D^{\gamma}_{\R^{d}}(FG_{j})-FD^{\gamma}_{\R^{d}}G_{j}\big](\cdot+k)\right\rVert_{L^{2}([0,1]^{d})}\\
        &\quad\, +\sum_{j=1}^{N}\Big\lVert \sum_{\substack{k\in \Z^{d} \\ |k|>2\sqrt{d}}}\big[D^{\gamma}_{\R^{d}}(FG_{j})-FD^{\gamma}_{\R^{d}}G_{j}\big](\cdot+k)\Big\rVert_{L^{2}([0,1]^{d})}
        =: (I)+(II).
    \end{align*}
    
    Let us bound $(I)$. For every $j\in \{1,\dots,N\}$, we take a cutoff function $\psi_{j}\in C^{\infty}_{c}(Q_{4/5}(\bar x_{j}); [0,1])$ such that $\psi_{j}=1$ on $Q_{3/4}(\bar x_{j})$, and we call $F_{j}:= F\psi_{j}$, so that $FG_{j}=F_{j}G_{j}$. Let $k\in \Z^{d}$ be such that $|k|\le 2\sqrt{d}$, then
    \begin{align*}
        \left\lVert \big[D_{\R^{d}}^{\gamma}(FG_{j})-FD_{\R^{d}}^{\gamma}G_{j}\big](\cdot+k)\right\rVert_{L^{2}([0,1]^{d})}&\le \left\lVert D_{\R^{d}}^{\gamma}(FG_{j})-FD_{\R^{d}}^{\gamma}G_{j}\right\rVert_{L^{2}(\R^{d})}\\
        &\le \left\lVert D_{\R^{d}}^{\gamma}(F_{j}G_{j})-F_{j}D_{\R^{d}}^{\gamma}G_{j}\right\rVert_{L^{2}(\R^{d})}+\left\lVert F(1-\psi_{j})D_{\R^{d}}^{\gamma}G_{j}\right\rVert_{L^{2}(\R^{d})}.
    \end{align*}
    To bound the first term on the right-hand side of the last inequality we use the Euclidean commutator estimate from \cite[Corollary 5.2]{li2019kato}:
    \begin{align*}
        \left\lVert D_{\R^{d}}^{\gamma}(F_{j}G_{j})-F_{j}D_{\R^{d}}^{\gamma}G_{j}\right\rVert_{L^{2}(\R^{d})}&\lesssim_{d,\gamma,p_{1},q_{1},p_{2},q_{2}}\lVert \nabla F_{j}\rVert_{L^{p_{1}}(\R^{d})}\lVert D_{\R^{d}}^{\gamma-1}G_{j}\rVert_{L^{q_{1}}(\R^{d})}+\lVert D_{\R^{d}}^{\gamma}F_{j}\rVert_{L^{p_{2}}(\R^{d})}\lVert G_{j}\rVert_{L^{q_{2}}(\R^{d})}\\
        &\lesssim_{d,\gamma,p_{1},q_{1},p_{2},q_{2}} \lVert \nabla f\rVert_{L^{p_{1}}(\T^{d})}\lVert D_{\T^{d}}^{\gamma-1}g\rVert_{L^{q_{1}}(\T^{d})}+\lVert D_{\T^{d}}^{\gamma}f\rVert_{L^{p_{2}}(\T^{d})}\lVert g\rVert_{L^{q_{2}}(\T^{d})},
    \end{align*}
    Next, we consider the second term. Calling $\mathcal{K}_{\gamma}:\R^{d}\to \R$ the kernel associated to the operator $D^{\gamma}_{\R^{d}}$\footnote{$\mathcal{K}_{\gamma}$ is such that $D^{\gamma}_{\R^{d}}u=\mathcal{K}_{\gamma}*u$ for every $u\in \mathcal{S}(\R^{d})$.}, we have
    \begin{equation*}
        |\mathcal{K}_{\gamma}(x)|\lesssim_{d,\gamma}|x|^{-d-\gamma}\qquad \forall x\in \R^{d}.
    \end{equation*}
    In particular, using H\"{o}lder inequality, along with $\psi_{j}\equiv 1$ on $Q_{3/4}(\bar x_{j})$ and $\supp G_{j}\subset Q_{1/2}(\bar x_{j})$, we get 
    \begin{align*}
    \lVert F(1-\psi_{j})D^{\gamma}_{\R^{d}}G_{j}\rVert_{L^{2}}
    & \le \left(\sum_{\ell \in \Z^{d}}\int_{(\ell+[0,1]^{d})\setminus Q_{3/4}(\bar x_{j})}\left(F\mathcal{K}_{\gamma}* G_{j}\right)^{2}\right)^{1/2}\\
    &\le \left(\sum_{\ell \in \Z^{d}}\lVert F\rVert_{L^{p_{2}}\left((\ell+[0,1]^{d})\setminus Q_{3/4}(\bar x_{j})\right)}^{2} \lVert \mathcal{K}_{\gamma}*G_{j}\rVert_{L^{q_{2}}\left((\ell+[0,1]^{d})\setminus Q_{3/4}(\bar x_{j})\right)}^{2}\right)^{1/2}\\
    &\lesssim_{d,\gamma} \lVert f\rVert_{L^{p_{2}}(\T^{d})}\left(\sum_{\ell \in \Z^{d}}|\ell|^{-2d-2\gamma}\lVert g\rVert_{L^{q_{2}}(\T^{d})}^{2}\right)^{1/2}\lesssim_{d,\gamma}\lVert D_{\T^{d}}^{\gamma}f\rVert_{L^{p_{2}}(\T^{d})}\lVert g\rVert_{L^{q_{2}}(\T^{d})}.
    \end{align*}
    Therefore:
    $$(I)\lesssim_{d,\gamma,p_{1},q_{1},p_{2},q_{2}} \lVert \nabla f\rVert_{L^{p_{1}}(\T^{d})}\lVert D_{\T^{d}}^{\gamma-1}g\rVert_{L^{q_{1}}(\T^{d})}+\lVert D_{\T^{d}}^{\gamma}f\rVert_{L^{p_{2}}(\T^{d})}\lVert g\rVert_{L^{q_{2}}(\T^{d})}.$$
    
    Finally, let us bound $(II)$. For any $k\in \Z^{d}$ such that $|k|>2\sqrt{d}$, any $x\in [0,1]^{d}$, and any $y\in \supp G_{j}$, we have $|x+k-y|\gtrsim |k|$. Therefore,
    \begin{align*}
        \left|\big[D^{\gamma}_{\R^{d}}(FG_{j})-FD^{\gamma}_{\R^{d}}G_{j}\big](x+k)\right|&=\left|\int_{\R^{d}}K_{\gamma}(x+k-y)\big[F(y)-F(x)\big]G_{j}(y)dy\right|\\&\lesssim |k|^{-d-\gamma}\int_{\R^{d}}\left(|F(x)|+|F(y)|\right)|G_{j}(y)|dy.
    \end{align*}
    Summing over $k$ we obtain
    \begin{equation*}
        \Bigg| \sum_{\substack{k\in \Z^{d} \\ |k|>2\sqrt{d}}}\Big[D^{\gamma}_{\R^{d}}(FG_{j})-FD^{\gamma}_{\R^{d}}G_{j}\Big](x+k)\Bigg|\lesssim \int_{\R^{d}}\left(|F(x)|+|F(y)|\right)|G_{j}(y)|dy \le (|F(x)|+\lVert f\rVert_{L^{p_{2}}(\T^{d})})\lVert g\rVert_{L^{q_{2}}(\T^{d})}.
    \end{equation*}
    We conclude that
    \begin{equation*}
        (II)\lesssim \lVert f\rVert_{L^{p_{2}}(\T^{d})}\lVert g\rVert_{L^{q_{2}}(\T^{d})}\lesssim \lVert D_{\T^{d}}^{\gamma}f\rVert_{L^{p_{2}}(\T^{d})}\lVert g\rVert_{L^{q_{2}}(\T^{d})}.
    \end{equation*}
    Combining the bounds on $(I)$ and $(II)$ we get \eqref{eq:kato-ponce}, so the proof is concluded.
\end{proof}

\begin{rmk}
When $\gamma\in (0,1]$ and $2\le p, q\le \infty$ are such that $1/p+1/q=1/2$, the following refinements of \eqref{eq:kato-ponce} can be obtained as above, using the corresponding Euclidean bounds from \cite[Lemma 5.2]{li2019kato}: 
    \begin{gather}
        \lVert D_{\T^{d}}(fg)-fD_{\T^{d}}g\rVert_{L^{2}(\T^{d})}\lesssim_{d,p,q} \lVert \nabla f\rVert_{L^{p}(\T^{d})}\lVert g\rVert_{L^{q}(\T^{d})},\label{eq:kato-ponce-gamma1}\\[5pt]
        \lVert D_{\T^{d}}^{\gamma}(fg)-fD_{\T^{d}}^{\gamma}g\rVert_{L^{2}(\T^{d})}\lesssim_{d,\gamma,p,q} \lVert D_{\T^{d}}^{\gamma}f\rVert_{L^{p}(\T^{d})}\lVert g\rVert_{L^{q}(\T^{d})}\qquad \text{for $\gamma \in (0,1)$}.
    \end{gather}\fr
\end{rmk}
\begin{rmk}
As a direct consequence of \Cref{lem:kato-ponce} and Gagliardo--Nirenberg (Brezis--Mironescu \cite{brezis2018gagliardo}, in the fractional case) interpolation inequalities, for every $f\in C^{\infty}(\T^{d})$, we deduce 
    \begin{equation}\label{eq:kato-ponce-non-linear-term}
        \lVert D_{\T^{d}}^{\gamma}(\nabla fD_{\T^{d}}^{\beta}f)-\nabla f D_{\T^{d}}^{\gamma+\beta}f\rVert_{L^{2}}\lesssim_{d,\gamma,\beta} \lVert \nabla^{2} f\rVert_{L^{\infty}}\lVert D_{\T^{d}}^{\gamma+\beta-1}f\rVert_{L^{2}},\qquad \forall\gamma>1,\quad \forall\beta\ge 2.
    \end{equation}
Indeed, choosing
    \begin{equation*}
        p=\frac{2(\gamma+\beta-3)}{\gamma-1},\qquad q=\frac{2(\gamma+\beta-3)}{\beta-2},
    \end{equation*}
    by fractional Gagliardo--Nirenberg inequalities \cite{brezis2018gagliardo}, 
    \begin{equation*}
        \lVert D_{\T^{d}}^{\gamma}\nabla f\rVert_{L^{p}}\lVert D_{\T^{d}}^{\beta}f\rVert_{L^{q}}\lesssim \lVert \nabla^{2} f\rVert_{L^{\infty}}^{\frac{2}{q}}\lVert D_{\T^{d}}^{\gamma+\beta-1}f\rVert_{L^{2}}^{\frac{2}{p}} \lVert \nabla^{2} f\rVert_{L^{\infty}}^{\frac{2}{p}}\lVert D_{\T^{d}}^{\gamma+\beta-1}f\rVert_{L^{2}}^{\frac{2}{q}}=\lVert \nabla^{2} f\rVert_{L^{\infty}}\lVert D_{\T^{d}}^{\gamma+\beta-1}f\rVert_{L^{2}}.
    \end{equation*}
    Now \eqref{eq:kato-ponce-non-linear-term} follows from \eqref{eq:kato-ponce}, with the choice $(p_{1},q_{1},p_{2},q_{2})=(2,\infty,p,q)$. \fr
\end{rmk}

\smallskip
\paragraph{Acknowledgements.} M.C., R.C., and X.F.~are supported by the Swiss State Secretariat for Education, Research and Innovation (SERI) under contract number MB22.00034 through the project TENSE. R.C. and X.F. ~are also supported by the Swiss National Science Foundation (SNF grant PZ00P2\_208930). X.F. is further supported  by the AEI project PID2021-125021NA-I00 (Spain).

\bibliography{RefsRieszGF.bib}
\bibliographystyle{alpha}

\end{document}